\let\chooseClass0   %Theory and Applications of Categories journal
\let\chooseClass4   %article, 12pt, for arXiv

\ifx\chooseClass0
\documentclass{tac}
\fi

\ifx\chooseClass1
\RequirePackage{fix-cm}
\documentclass[smallcondensed,numbook,envcountsame]{svjour3}
\smartqed
\journalname{Applied Categorical Structures}
\fi

\ifx\chooseClass2
\documentclass[14pt]{extarticle}
\fi

\ifx\chooseClass3
\documentclass[17pt]{extarticle}
\fi

\ifx\chooseClass4
\documentclass[12pt]{article}
\fi

\ifx\chooseClass0
	\else
\makeatletter
\def\@seccntformat#1{\csname the#1\endcsname.\quad}
\renewcommand\section{\@startsection {section}{1}{\z@}%
                                   {-3.5ex \@plus -1ex \@minus -.2ex}%
                                   {2.3ex \@plus.2ex}%
                                   {\normalfont\large\bfseries}}
\renewcommand\subsection{\@startsection{subsection}{2}{\z@}%
                        {3.25ex plus 1ex minus .2ex}{-.5em}%
                        {\normalfont\normalsize\bfseries}}
\renewcommand\subsubsection{\@startsection{subsubsection}{3}{\z@}%
                        {3.25ex plus 1ex minus .2ex}{-.5em}%
                        {\normalfont\normalsize\it}}
\@addtoreset{equation}{section}
\makeatother

\usepackage{amsthm}
\newtheoremstyle{boldhead}%     name
{\topsep}%                      abovespace
{\topsep}%                      belowspace
{\slshape}%                     bodyfont
{}%                             indentation=noindent
{\bfseries}%                    headfont
{.}%                            headpunctuation
{ }%                            headspace=interword space
{\thmname{#1}\thmnumber{ #2}\thmnote{ (#3)}}%  custom head specification

\newtheoremstyle{boldremark}%   name
{\topsep}%                      abovespace
{\topsep}%                      belowspace
{\upshape}%                     bodyfont
{}%                             indentation=noindent
{\bfseries}%                    headfont
{.}%                            headpunctuation
{ }%                            headspace=interword space
{\thmname{#1}\thmnumber{ #2}\thmnote{ (#3)}}%  custom head specification

\swapnumbers

\theoremstyle{boldhead}
\newtheorem{theorem}[subsection]{Theorem}

\newtheorem{corollary}[subsection]{Corollary}
\newtheorem{lemma}[subsection]{Lemma}
\newtheorem{proposition}[subsection]{Proposition}

\theoremstyle{boldremark}
\newtheorem*{acknowledgement}{Acknowledgement}

\newtheorem{definition}[subsection]{Definition}

\newtheorem{example}[subsection]{Example}

\newtheorem{remark}[subsection]{Remark}

\fi%\chooseClass2,3

\usepackage{	%abbrevs,
				amsmath,
                amsfonts,
                amssymb,
                amsxtra,
                fancybox,
}

\numberwithin{equation}{section}

\ifx\chooseClass0
\else
\usepackage{ifpdf}
\ifpdf
	\usepackage[pdftex]{hyperref}
\else
    \usepackage[hypertex]{hyperref}
\fi
\fi

\ifx\chooseClass1

  \else
\usepackage[mathcal]{euscript}
\fi

\usepackage{diagrams}
\diagramstyle[height=2em,balance,righteqno,PostScript=dvips,nohug]

\usepackage[sans]{dsfont}
\newcommand\1{{\mathds 1}}
\newlength{\mylabelwidths}
{\end{list}}

\def\rhaha{\raise.20ex\hbox{$\rightharpoonup$}\kern-1em\lower.20ex\hbox{$\rightharpoondown$}}%.24ex
\def\lhaha{\raise.20ex\hbox{$\leftharpoonup$}\kern-1em\lower.20ex\hbox{$\leftharpoondown$}}%.24ex
\def\dhaha{\kern.01em\downharpoonleft\kern-.26em\downharpoonright\kern.01em}%-.24em.02em
\def\uhaha{\kern.01em\upharpoonleft\kern-.26em\upharpoonright\kern.01em}%-.24em.02em
\newarrowhead{twoharpoons}\rhaha\lhaha\dhaha\uhaha

\newarrow{DashTo}{}{dash}{}{dash}{->}
\newarrow{Epi}----{triangle}
\newarrow{Id}===={twoharpoons}
\newarrow{Mono}{boldhook}---{->}
\newarrow{TTo}----{->}

\newcommand\LL{{\mathbb L}}
\newcommand\NN{{\mathbb N}}
\newcommand\RR{{\mathbb R}}
\newcommand\ZZ{{\mathbb Z}}

\newcommand{\ca}{{\mathcal A}}
\newcommand{\cb}{{\mathcal B}}
\newcommand{\cc}{{\mathcal C}}
\newcommand{\cd}{{\mathcal D}}
\newcommand{\cf}{{\mathcal F}}
\newcommand{\co}{{\mathcal O}}
\newcommand{\cu}{{\mathcal U}}
\newcommand{\cv}{{\mathcal V}}

\newcommand{\fa}{{\mathfrak a}}
\newcommand{\fA}{{\mathfrak A}}
\newcommand{\fb}{{\mathfrak b}}
\newcommand{\fB}{{\mathfrak B}}
\newcommand{\fc}{{\mathfrak c}}
\newcommand{\fC}{{\mathfrak C}}
\newcommand{\fd}{{\mathfrak d}}
\newcommand{\fD}{{\mathfrak D}}

\newcommand{\mcC}{{\mathsf C}}
\newcommand{\mcD}{{\mathsf D}}

\newcommand{\bdb}{\boxtimes\dots\boxtimes}

\newcommand{\bull}{{\scriptscriptstyle\bullet}}

\newcommand{\htdht}{\hat\otimes\cdots\hat\otimes}

\newcommand{\tdt}{\otimes\dots\otimes}
\newcommand{\yi}{\ddot\imath}
\newcommand{\yii}{\dddot\imath}

\newcommand{\ainf}[1]{$A_\infty$\nobreakdash-\hspace{0pt}}
\newcommand{\n}[1]{\nobreakdash-\hspace{0pt}}

\let\boxt\boxtimes

\let\emptyset\varnothing
\let\eps\varepsilon
\let\ge\geqslant

\let\le\leqslant
\let\mb\mathbf

\let\rto\xrightarrow

\let\tens\otimes

\let\und\underline
\let\wh\widehat

\newcommand{\leftunit}{{\mathbf l}}
\newcommand{\rightunit}{{\mathbf r}}

\newcommand{\modul}{\textup{-mod}}
\newcommand{\modulc}{\textup{-mod}^{\textup c}}
\newcommand{\Quiv}{\textup{-\sf Quiv}}

\DeclareMathOperator{\Ab}{\mathsf{Ab}}
\DeclareMathOperator{\acCat}{acCat}
\DeclareMathOperator{\Alg}{Alg}
\DeclareMathOperator{\acncCat}{acncCat}

\DeclareMathOperator{\Coalg}{Coalg}

\DeclareMathOperator{\Coder}{Coder}
\DeclareMathOperator{\Coker}{Coker}

\DeclareMathOperator{\coCat}{coCat}
\DeclareMathOperator\dg{\mathbf{dg}}

\DeclareMathOperator{\ev}{ev}

\DeclareMathOperator\gr{\mathbf{gr}}
\DeclareMathOperator{\grAb}{\mathsf{grAb}}
\DeclareMathOperator{\cgrAb}{\mathsf{cgrAb}}

\DeclareMathOperator\id{id}
\DeclareMathOperator\Id{Id}
\DeclareMathOperator\im{Im}
\DeclareMathOperator\inj{in}
\DeclareMathOperator\Ker{Ker}

\DeclareMathOperator\Mor{Mor}
\DeclareMathOperator{\ncCat}{ncCat}
\DeclareMathOperator{\cncCat}{cncCat}
\DeclareMathOperator\Ob{Ob}

\newcommand{\op}{{\operatorname{op}}}

\DeclareMathOperator\pr{pr}

\DeclareMathOperator\Set{\mathcal Set}

\newcommand{\apref}[1]{Appendix~\ref{#1}}
\newcommand{\corref}[1]{Corollary~\ref{#1}}
\newcommand{\defref}[1]{Definition~\ref{#1}}
\newcommand{\exaref}[1]{Example~\ref{#1}}
\newcommand{\lemref}[1]{Lemma~\ref{#1}}
\newcommand{\propref}[1]{Proposition~\ref{#1}}
\newcommand{\remref}[1]{Remark~\ref{#1}}
\newcommand{\secref}[1]{Section~\ref{#1}}
\newcommand{\thmref}[1]{Theorem~\ref{#1}}

\author{Volodymyr Lyubashenko}
\title{Filtered cocategories}

\ifx\chooseClass0
\copyrightyear{2019}
\address{Institute of Mathematics NASU, 3 Tereshchenkivska st., Kyiv,
01024, Ukraine}
\eaddress{lub@imath.kiev.ua}
\fi

\begin{document}

\ifx\chooseClass1
\institute{V. Lyubashenko \at
              Institute of Mathematics NASU,
			  3 Tereshchenkivska st.,
			  Kyiv, 01024, Ukraine \\
              Tel.: +380-44-2357819,
              Fax: +380-44-2352010\\
              \email{lub@imath.kiev.ua}
}
\date{Received: date / Accepted: date}
\fi

\maketitle

\allowdisplaybreaks[1]

\begin{abstract}
We recall the notions of a graded cocategory, conilpotent cocategory,
morphisms of such (cofunctors), coderivations and define their analogs
in $\mathbb L$-filtered setting.
The difference with the existing approaches: we do not impose any 
restriction on $\Lambda$\n-modules of morphisms (unlike Fukaya and 
collaborators), we consider a wider class of filtrations than 
De~Deken and Lowen (including directed groups $\mathbb L$).
Results for completed filtered conilpotent cocategories include: 
cofunctors and coderivations with value in completed tensor cocategory 
are described, a partial internal hom is constructed as the tensor 
cocategory of certain coderivation quiver, 
when the second argument is a completed tensor cocategory.
\end{abstract}

\subsection*{Introduction.}
The subject of usual (non-filtered) \ainf-categories is absorbed to some
extent by the subject of $\dg$\n-categories since any non-filtered 
\ainf-category is equivalent to a $\dg$\n-category.
On the other hand, \ainf-categories arising in symplectic geometry 
(Fukaya categories) are naturally $\RR$\n-filtered.
Hence the necessity to study filtered \ainf-categories per se.
Such a study began in works of Fukaya ({\it e.g.} 
\cite{Fukaya:FloerMirror-II}) continued in his works with Oh, Ohta 
and Ono ({\it e.g.} \cite{FukayaOhOhtaOno:Anomaly}).
The restriction imposed in these works (dictated by the geometric origin
of Fukaya categories) is that the modules of morphisms are torsion free 
over a graded commutative filtered ring $\Lambda$, the Novikov ring.
The various freeness requirements are removed in the approach 
of \cite{MR3856926}.
However, they work with $\LL$\n-filtered modules, where the commutative 
monoid $\LL$ has partial ordering such that 
the neutral element 0 is the smallest element of $\LL$.
For instance, such is $\LL=\RR_{\ge0}$, but not $\LL=\RR$.

Here we relax the conditions on partially ordered commutative monoid 
$\LL$, whose elements index the filtration, 
thereby including directed groups.
And we keep the feature of not necessarily torsion free modules 
of morphisms.
Thus we combine the features of works of Fukaya, Oh, Ohta and Ono on 
the one hand and of works of De~Deken and Lowen on the other.
I hope that this combination will be useful for articles on Homological 
Mirror Conjecture of \cite{Kontsevich:alg-geom/9411018}.

In the present article we deal mostly with a predecessor of 
\ainf-categories -- $\LL$\n-filtered $\ZZ$\n-graded cocategories $\fA$ 
over a graded commutative complete $\LL$\n-filtered ring $\Lambda$.
Among $\fA$ we distinguish conilpotent cocategories $\fa$, especially, 
tensor quivers $T\fb$ with cut comultiplication, and their completions 
$\hat\fa$ and $\wh{T\fb}$, respectively.
The completions are taken with respect to uniform structure coming from 
the filtration.
The uniform structure is one of the tools we use in the study 
of filtered cocategories.
Our results include description of morphisms (cofunctors) from 
a completed conilpotent cocategory $\hat\fa$ to $\wh{T\fb}$.
Furthermore, we describe coderivations between such cofunctors.
We define a partial internal hom between completed conilpotent 
cocategories.
Partial because the considered second arguments are only of the form 
$\wh{T\fb}$.
This internal hom is the tensor cocategory of coderivation quiver.
The latter has cofunctors as objects and coderivations as morphisms.
We define also the evaluation cofunctor and prove its property which 
justifies the name of evaluation.
Composition of cofunctors \(\wh{T\fa}\to \wh{T\fb}\to \wh{T\fc}\) 
extends to composition (cofunctor) of internal homs.
When the source $\wh{Ts\ca}$ and the target $\wh{Ts\cb}$ are filtered 
\ainf-categories (equipped with a differential of degree 1 preserving 
the filtration), so is the coderivation quiver (up to a shift).

\tableofcontents

\subsection*{Plan of the article.}
In the first section we deal with non-filtered graded cocategories.
To some extent this is a recollection of \cite{Lyu-AinfCat} and serves 
as an introduction to the filtered case.
The new exposition differs from \cite{Lyu-AinfCat} in the use of 
conilpotent cocategories as a source instead of tensor categories.
Also the proofs of the main results (\propref{prop-psi-ev}) are new.

The second section is devoted to the main subject -- $\LL$\n-filtered 
$\ZZ$\n-graded cocategories, especially, to completed conilpotent 
cocategories.
We begin with conditions on commutative partially ordered monoid $\LL$.
We study (complete) $\LL$\n-filtered $\ZZ$\n-graded abelian groups and 
later (complete) $\LL$\n-filtered $\Lambda$\n-modules 
(\secref{sec-Complete-Lambda-modules}), where $\Lambda$ is a graded 
commutative complete $\LL$\n-filtered ring, for instance, 
the universal Novikov ring.
In \secref{sec-Complete-cocategories} we define completed conilpotent 
cocategories and their morphisms (cofunctors).
We describe cofunctors with values in a completed tensor cocategory in 
\thmref{thm-Cofunctors-1st-projection}.
In \secref{sec-Coderivations} we study coderivations, in particular, in 
\propref{pro-coderivations-bijection} we describe coderivations with 
values in a completed tensor cocategory.
In \secref{sec-Evaluation-filtered} we define the evaluation cofunctor 
and prove in \thmref{thm-Ev-psi} its property, which justifies the name 
of evaluation.

In the third section we apply these results to differential graded 
completed tensor cocategories as known as filtered \ainf-categories.
We prove in \propref{pro-B-F-KS-LH} that the coderivation quiver for two
filtered \ainf-categories is a filtered \ainf-category itself 
(up to a shift).
Examples are given.

\subsection*{Conventions.}
We work in Tarski--Grothendieck set theory originated in 
\cite{Tarski-well-ordered-subsets}.
In this theory everything is a set (or an element of a set) and any set 
is an element of some Grothendieck universe.
In particular, any universe is an element of some (bigger) universe.

Let $\cv$ be a symmetric monoidal category.
By a lax plain/symmetric/\hspace{0pt}braided monoidal $\cv$\n-category 
$\cc$ we mean a $\cv$\n-category equipped with $\tens^I$, $\lambda^f$, 
$\rho^L$ with the properties listed in 
\cite[Definition~2.10]{BesLyuMan-book} (the one with natural 
transformations 
\(\lambda^f:\tens^{i\in I}M_i\to\tens^{j\in J}\tens^{i\in f^{-1}j}M_i\) 
for non-decreasing/arbitrary/arbitrary map of finite ordered 
sets $f:I\to J$).
The same notion bears the name 'oplax` in the works of Day, Street, 
Leinster, Schwede, Shipley and many other authors.
Any finite ordered set $I$ is isomorphic to \(\mb n=\{1<2<\dots<n\}\) 
for a unique $n\ge0$.
In order to reduce the data we assume that 
\(\tens^I=\bigl(\cc^I\cong\cc^n\rTTo^{\tens^n} \cc)\); $\lambda^f$ 
identifies with $\lambda^g$, where \(f:I\to J\) is a map of finite 
ordered sets and \(g:\mb n\to\mb m\) comes from the commutative square
\begin{diagram}
I &\rTTo^\cong &\mb n
\\
\dTTo<f &&\dTTo>g
\\
J &\rTTo^\cong &\mb m
\end{diagram}
$\rho^L$ for 1-element set $L$ reduces to \(\rho^1:\tens^1\to\Id\).
This reduction is used only for easier writing and one can get rid off 
it whenever needed.
Similarly, in the definition of a ($\cv$)-multicategory $\mcC$ we assume
that \(\mcC((M_i)_{i\in I};N)=\mcC(M_{\phi(1)},\dots,M_{\phi(n)};N)\) 
for the only non-decreasing bijection \(\phi:\mb n\to I\), with 
the corresponding requirement on compositions for $\mcC$.
Summing up, the notion reduces to \(I\in\{\mb n\mid n\in\ZZ_{\ge0}\}\) 
and we may simply write \(\mcC(M_1,\dots,M_n;N)\).
The results of the article extend obviously to the picture indexed by 
arbitrary finite ordered sets, which is anyway isomorphic to 
the picture in which only $\mb n$ are used as indexing sets.

Composition of two morphisms of certain degrees \(f:X\to Y\) and 
\(g:Y\to Z\) is mostly denoted \(fg=f\cdot g\).
When the sign issues are irrelevant the composition may be denoted
\(gf=g\circ f\).
Applying a mapping $f$ of certain degree to an element $x$ of certain 
degree we typically write \(xf=(x)f\).
When there are no sign issues the same may be written as \(fx=f(x)\).

\ifx\chooseClass0
\subsubsection*{Acknowledgement.}
	I am really grateful to Kaoru Ono for explaining the geometric side 
	of Fukaya categories.
\else
\begin{acknowledgement}
	I am really grateful to Kaoru Ono for explaining the geometric side 
	of Fukaya categories.
\end{acknowledgement}
\fi

\section{Conilpotent cocategories}\label{sec-intro}
Let $\cv$ be the complete additive symmetric monoidal category with 
small coproducts and directed colimits 
\(\cv=\gr=\gr_\Lambda=\Lambda\modul\), where $\Lambda$ is a 
$\ZZ$\n-graded commutative ring, \(\Lambda\modul\) means the category of
$\ZZ$\n-graded abelian groups which are also $\Lambda$\n-modules 
(and the action \(\Lambda\tens M\to M\) has degree 0).
Besides these properties sometimes we use also that $\cv$ is closed 
symmetric monoidal.
Examples of $\Lambda$ are the universal Novikov ring 
\(\Lambda_{0,nov}(R)\) and its localization \(\Lambda_{nov}(R)\), 
see \cite[\S1.7 (Conv. 4)]{FukayaOhOhtaOno:Anomaly}.
Left $\Lambda$\n-modules are viewed as commutative 
$\Lambda$\n-bimodules, 
\(\pm m\lambda\equiv\mu\tau(\lambda\tens m)=\lambda m\).
In general, commutativity is considered with respect to the symmetry 
\(\tau(x\tens y)=(-)^{xy}y\tens x=(-1)^{\deg x\cdot\deg y}y\tens x\).
Thus, \((\cv,\tens,\1,\tau,\und\cv)\) means  
\((\gr_\Lambda,\tens_\Lambda,\Lambda,\tau,\und\gr)\), where 
the inner hom is
\[ \und\gr(M,N)^d =\{f\in\prod_{n\in\ZZ}\Ab(M^n,N^{n+d}) \mid  
\forall n\in\ZZ\;\forall p\in\ZZ\;\forall\lambda\in\Lambda^p\; 
\lambda f=f\lambda:M^n\to N^{n+d+p}\}.
\]

\begin{definition}
A $\cv$\n-quiver $\fa$ is a set of objects \(\Ob\fa\) and an object 
\(\fa(X,Y)\in\Ob\cv\) given for each pair of objects \(X,Y\in\Ob\fa\).
The category of $\cv$\n-quivers \(\cv\Quiv\) has as morphisms 
\(f:\fa\to\fb\) collections consisting of a map 
\(f=\Ob f:\Ob\fa\to\Ob\fb\) and morphisms \(f:\fa(X,Y)\to\fb(fX,fY)\) 
for each pair of objects \(X,Y\in\Ob\fa\).
\end{definition}

\begin{example}
Let $S$ be a set.
One forms a $\cv$\n-quiver $\1S$ with \(\Ob\1S=S\),
\[ \1S(X,Y) =
\begin{cases}
\1, &\textup{if } X=Y, \\
0 , &\textup{if } X\ne Y.
\end{cases}
\]
A mapping \(f:S\to Q\) induces a quiver morphism \(\1f:\1S\to\1Q\) with 
\(\Ob\1f=f\) and \(\1f=\id_\1:\1S(X,X)\to\1Q(fX,fX)\) for any \(X\in S\).
\end{example}

The category \(\cv\Quiv\) is symmetric monoidal with the tensor product 
$\boxt$
\begin{align*}
\Ob\fa\boxt\fb &=\Ob\fa\times\Ob\fb, \\
(\fa\boxt\fb)((A,B),(A',B')) &=\fa(A,A')\tens\fb(B,B').
\end{align*}
The unit object quiver $\1$ has one-element set \(\Ob\1=\{*\}\) and 
\(\1(*,*)=\1\).
The symmetry comes from that of $\cv$.
Since the symmetric monoidal category $\cv$ is closed, so is 
\(\cv\Quiv\) with inner hom object \(\und{\cv\Quiv}(\fa,\fb)\), which is
the $\cv$\n-quiver with 
\(\Ob\und{\cv\Quiv}(\fa,\fb)=\Set(\Ob\fa,\Ob\fb)\), and for any pair of 
maps \(f,g:\Ob\fa\to\Ob\fb\)
\[ \und{\cv\Quiv}(\fa,\fb)(f,g) 
=\prod_{X,Y\in\Ob\fa} \und\cv(\fa(X,Y),\fb(fX,gY)).
\]
The evaluation morphism \(\ev:\fa\boxt\und{\cv\Quiv}(\fa,\fb)\to\fb\) 
(the adjunct of \(1_{\und{\cv\Quiv}(\fa,\fb)}\)) is given by
\begin{multline*}
\hfill (X,f) \mapsto fX, \hfill
\\
\hskip\multlinegap \fa(X,Y)\tens\prod_{X',Y'\in\Ob\fa} 
\und\cv(\fa(X',Y'),\fb(fX',gY')) \rTTo^{1\tens\pr_{X,Y}} \hfill
\\[-0.9em]
\fa(X,Y)\tens \und\cv(\fa(X,Y),\fb(fX,gY)) \rTTo^\ev \fb(fX,gY).
\end{multline*}

By definition there is a functor \(\Ob:\cv\Quiv\to\Set\), 
\(\fa\mapsto\Ob\fa\).
Consider the fiber \(\cv\Quiv_S\) of this functor over a set $S$, 
that is,
\begin{align*}
\Ob\cv\Quiv_S &= \{\fa\in\cv\Quiv \mid \Ob\fa =S\},
\\
\Mor\cv\Quiv_S &= \{f\in\Mor\cv\Quiv \mid \Ob f =\id_S\}.
\end{align*}
Since $\cv$ is abelian, so is \(\cv\Quiv_S\) being isomorphic to 
\(\cv^{S\times S}\).

\begin{definition}
The category \(\cv\Quiv_S\) is monoidal with the tensor product $\tens$
\[ (\fa\tens\fb)(X,Z) =\coprod_{Y\in S} \fa(X,Y)\tens\fb(Y,Z),
\]
and the unit object $\1S$.
\end{definition}

For an arbitrary $\cv$\n-quiver $\fa$ we denote 
\(T^n\fa=\fa^{\tens n}\in\cv\Quiv_{\Ob\fa}\), \(n\ge0\), 
\(T^0\fa=\1\Ob\fa=\1\fa\).
We use \(\1\fa\) as a shorthand for \(\1\Ob\fa\) and \(\1f\) as a 
shorthand for \(\1\Ob f\), where $f$ is a morphism of quivers.

Let $\fa_1$, \dots, $\fa_n$, $\fb_1$, \dots, $\fb_n$ be $\cv$\n-quivers 
with \(\Ob\fa_1=\dots=\Ob\fa_n=S\), \(\Ob\fb_1=\dots=\Ob\fb_n=Q\).
Let \(f_i:\fa_i\to\fb_i\), \(1\le i\le n\), be morphisms of quivers 
such that \(\Ob f_i=f:S\to Q\).
Then the following morphism is well-defined
\begin{multline}
\hfill f_1\tdt f_n: \fa_1\tdt\fa_n \to \fb_1\tdt\fb_n, \hfill
\\
\Ob f_1\tdt f_n =f,
\\[-0.3em]
\hskip\multlinegap f_1\tdt f_n =\Bigl[ \coprod_{X_1,\dots,X_{n-1}\in S} 
\fa_1(X_0,X_1)\tdt\fa_n(X_{n-1},X_n) \rTTo^{f_1\tdt f_n} \hfill
\\[-0.3em]
\coprod_{X_1,\dots,X_{n-1}\in S} 
\fb_1(fX_0,fX_1)\tdt\fb_n(fX_{n-1},fX_n) 
\rTTo^{(\inj_{fX_1,\dots,fX_{n-1}})_{X_1,\dots,X_{n-1}\in S}}
\\[-0.5em]
\coprod_{Y_1,\dots,Y_{n-1}\in Q} 
\fb_1(fX_0,Y_1)\tdt\fb_n(Y_{n-1},fX_n) \Bigr].
\label{eq-f1tdtfn}
\end{multline}
In the case $n=0$ (when a map \(f:S\to Q\) is given) \(f_1\tdt f_n\) is
the morphism \(\1f:\1S\to\1Q\).

\begin{definition}\label{def-cocategory}
A cocategory $\fc$ is a coalgebra in the monoidal category 
\(\cv\Quiv_S\).
Of course, \(S=\Ob\fc\).
In other words, $\fc$ is a $\cv$\n-quiver equipped with a coassociative 
comultiplication \(\Delta:\fc\to\fc\tens\fc\) and the counit 
\(\eps:\fc\to\1\fc\) which satisfies the usual counitality equations.
Morphisms of cocategories (cofunctors) \(f:\fb\to\fc\) are morphisms of 
$\cv$\n-quivers compatible with the comultiplication and the counit 
in the sense that
\begin{equation}
\begin{diagram}[inline]
\fb &\rTTo^f &\fc \\
\dTTo<\Delta &= &\dTTo>\Delta \\
\fb\tens\fb &\rTTo^{f\tens f} &\fc\tens\fc
\end{diagram}
\qquad,\qquad
\begin{diagram}[inline]
\fb &\rTTo^f &\fc \\
\dTTo<\eps &= &\dTTo>\eps \\
\1\fb &\rTTo^{\1f} &\1\fc
\end{diagram}
\qquad.
\label{eq-dia-Delta-epsilon}
\end{equation}
The category of cocategories is denoted $\coCat$.
\end{definition}

\begin{example}
For any set $S$ the $\cv$\n-quiver $\1S$ is a cocategory with the 
identity morphism \(\id_{\1S}\)  as $\eps$ and the isomorphism 
\(\Delta:\1S\to\1S\tens\1S\), coming from the canonical isomorphism 
\(\1\cong\1\tens\1\).
\end{example}

\begin{definition}\label{def-augmented-cocategory}
An augmented cocategory $\fc$ is a coalgebra morphism 
\(\eta:\1\fc\to\fc\) in \(\cv\Quiv_{\Ob\fc}\).
Morphisms of augmented cocategories \(f:\fb\to\fc\) are morphisms of 
cocategories compatible with the augmentation, that is,
\begin{diagram}[LaTeXeqno]
\1\fb &\rTTo^{\1f} &\1\fc \\
\dTTo<\eta &= &\dTTo>\eta \\
\fb &\rTTo^f &\fc
\label{dia-augmentation-preserving}
\end{diagram}
The category of augmented cocategories is denoted $\acCat$.
\end{definition}

Notice that \(\Ob\eta=\id_{\Ob\fc}\) for an augmented cocategory $\fc$.
It follows from \eqref{eq-dia-Delta-epsilon} that 
\(\eta\cdot\eps=1_\fc\).

Recall that the category $\cv=\gr$ is idempotent complete.
An augmented cocategory $\fc$ splits into a direct sum 
\(\fc=\1\fc\oplus\bar{\fc}\) in \(\cv\Quiv_{\Ob\fc}\) so that $\eps$ 
becomes \(\pr_1\) and $\eta$ becomes $\inj_1$.
A non-counital comultiplication, induced on \(\bar{\fc}\),
\begin{equation}
\bar\Delta =\Delta -\eta\tens1 -1\tens\eta: 
\bar\fc \to \bar\fc\tens\bar\fc \in \cv\Quiv_{\Ob\fc}
\label{eq-bar-Delta}
\end{equation}
is coassociative.
In fact, \(\pr_2:\fc\to\bar\fc\) identifies with the canonical projection
\(\pi:\fc\to\Coker\eta=\fc/\im\eta\) and $\bar\Delta$ can be found from
\begin{diagram}[LaTeXeqno]
\fc &\rTTo^\Delta &\fc\tens\fc
\\
\dTTo<\pi &= &\dTTo>{\pi\tens\pi}
\\
\bar\fc &\rTTo^{\bar\Delta} &\bar\fc\tens\bar\fc
\label{piDelta=Deltapipi}
\end{diagram}

\begin{definition}\label{def-conilpotent-cocategory}
An augmented cocategory $\fc$ is called \emph{conilpotent} when 
\((\bar\fc,\bar\Delta)\) is conilpotent, that is,
\[ \bigcup_{n>1} \Ker(\bar\Delta^{(n)}: 
\bar{\fc} \to \bar{\fc}^{\tens n}) =\bar{\fc}.
\]
The full subcategory of $\acCat$ whose objects are conilpotent 
cocategories is denoted $\ncCat$.
\end{definition}

%\begin{question}
%Given a cocategory \((\fc,\Delta,\eps)\), its conilpotency does not depend on the choice of augmentation $\eta$?
%\end{question}

\begin{example}\label{exa-T-conilpotent-cocategory}
For an arbitrary $\cv$\n-quiver $\fa$ there is the tensor quiver 
\(T\fa=\coprod_{n\ge0}T^n\fa=\coprod_{n\ge0}\fa^{\tens n}\equiv\oplus_{n\ge0}\fa^{\tens n}\).
Define comultiplication \(\Delta:T\fa\to T\fa\tens T\fa\) as the sum of 
canonical isomorphisms 
\(\fa^{\tens n}\to\fa^{\tens k}\tens\fa^{\tens l}\), \(k+l=n\), 
\(k,l\ge0\).
On elements
\[ \Delta(h_1\tens h_2\tens\dots\tens h_n) 
=\sum_{k=0}^n h_1\tdt h_k\bigotimes h_{k+1}\tdt h_n
\]
is the cut comultiplication.
The counit is \(\eps=\pr_0:T\fa\to T^0\fa=\1\fa\), the augmentation 
is \(\eta=\inj_0:T^0\fa\to T\fa\).
The direct summand \(\overline{T\fa}=T^{>0}\fa=\oplus_{n>0}T^n\fa\) is 
equipped with the reduced comultiplication 
\(\bar\Delta:T^{>0}\fa\to T^{>0}\fa\tens T^{>0}\fa\) which is the sum of
canonical isomorphisms 
\(\fa^{\tens n}\to\fa^{\tens k}\tens\fa^{\tens l}\), \(k+l=n\), 
\(k,l>0\).
On elements
\[ \bar\Delta(h_1\tens h_2\tens\dots\tens h_n) 
=\sum_{k=1}^{n-1} h_1\tdt h_k\bigotimes h_{k+1}\tdt h_n.
\]
Since it is conilpotent, the augmented cocategory \(T\fa\) 
is conilpotent.
It is called the cofree conilpotent cocategory.
The reasons are clear from the following
\end{example}

\begin{proposition}\label{pro-tensor-ncCat}
The tensor cocategory construction extends to a functor 
\(T:\cv\Quiv\to\ncCat\), representing the left hand side of a natural 
bijection 
\[ \{\phi:\fa\to\fb\in\cv\Quiv\mid \eta\cdot\phi=0\} 
\cong\ncCat(\fa,T\fb).
\]
\end{proposition}

\begin{proof}
For any \(g:\fa\to\fb\in\cv\Quiv\) the morphisms \(T^ng=g^{\tens n}\) 
constructed in \eqref{eq-f1tdtfn} form $Tg$ and extend $T$ to a functor.
Any morphism \(f:\fa\to T\fb\in\ncCat\) is uniquely determined by the 
composition \(\check{f}=f\cdot\pr_1:\fa\to\fb\in\cv\Quiv\) such that 
\(\eta\cdot\check f=0\), as the following commutative diagram shows
\begin{diagram}[width=4em]
&&T\fb &&
\\
&\ruTTo^f &\dTTo<{\Delta^{(k)}} &\rdTTo^{\pr_k} &
\\
\fa &&(T\fb)^{\tens k} &\rTTo^{\hspace*{-2em}\pr_1^{\tens k}} 
&\fb^{\tens k}
\\
&\rdTTo<{\Delta^{(k)}} &\uTTo<{f^{\tens k}} 
&\ruTTo_{\check{f}^{\tens k}} &
\\
&&\fa^{\tens k} &&
\end{diagram}
Namely, the following expression makes sense in notation 
\(x_{(1)}\tdt x_{(k)} \equiv\Delta^{(k)}(x)\)
\[ f(x) =\sum_{k\ge0} \check f(x_{(1)})\tdt\check f(x_{(k)}) 
=(x)\eps\cdot(\1f)\cdot\inj_0 
+\sum_{k\ge1} \check f(x_{(1)})\tdt\check f(x_{(k)}),
\]
since \(\bar\Delta^{(k)}(x)=0\) for large $k$ and 
\(\eta\cdot\check f=0\).
\end{proof}

The category $\coCat$ is symmetric monoidal with the tensor product 
$\boxt$ given by
\begin{multline*}
\Big( \fa\boxt\fb, \Delta =\big(\fa\boxt\fb \rTTo^{\Delta\boxt\Delta} 
(\fa\tens\fa)\boxt(\fb\tens\fb) \rTTo^{\oplus1\tens c\tens1} 
(\fa\boxt\fb)\tens(\fa\boxt\fb)\big),
\\
\eps =\big(\fa\boxt\fb 
\rTTo^{\eps\boxt\eps} \1\fa\boxt\1\fb \cong \1(\Ob\fa\times\Ob\fb) 
=\1(\fa\boxt\fb)\big)\Big).
\end{multline*}
The isomorphism \(\tau_{(23)}=\oplus1\tens\tau\tens1: 
(\fa\tens\fa)\boxt(\fb\tens\fb)\to(\fa\boxt\fb)\tens(\fa\boxt\fb)\) 
(the middle four interchange) is the direct sum of isomorphisms
\[ 1\tens\tau\tens1: \fa(X,Y)\tens\fa(Y,Z)\tens\fb(U,V)\tens\fb(V,W) \to
\fa(X,Y)\tens\fb(U,V)\tens\fa(Y,Z)\tens\fb(V,W).
\]
Hence the category \(\acCat\) is symmetric monoidal.
The augmentation for the tensor product \(\fa\boxt\fb\) of augmented 
cocategories $\fa$, $\fb$ is
\[ \eta =\big(\1(\fa\boxt\fb) =\1(\Ob\fa\times\Ob\fb) \cong 
\1\fa\boxt\1\fb \rTTo^{\eta\boxt\eta} \fa\boxt\fb\big).
\]

\begin{proposition}
	\label{pro-nilpotent-full-monoidal-subcategory-augmented}
The category $\ncCat$ is a full monoidal subcategory of $\acCat$.
\end{proposition}

\begin{proof}
Given two conilpotent cocategories $\fc$ and $\fd$, let us prove that 
\(\fc\boxt\fd\) is conilpotent.
The canonical projections \(\pi:\fc\to\bar\fc\) and 
\(\pi:\fd\to\bar\fd\) allow to write 
\(\pi:\fc\boxt\fd\to\overline{\fc\boxt\fd}\) as
\[ (\pi\boxt\pi,\pi\boxt\eps,\eps\boxt\pi): \fc\boxt\fd \to 
\bar\fc\boxt\bar\fd \oplus \bar\fc\boxt\1\fd \oplus \1\fc\boxt\bar\fd.
\]
Diagram~\eqref{piDelta=Deltapipi} implies
\begin{diagram}
\fc &\rTTo^{\Delta^{(l)}} &\fc^{\tens l}
\\
\dTTo<\pi &= &\dTTo>{\pi^{\tens l}}
\\
\bar\fc &\rTTo^{\bar\Delta^{(l)}} &\bar\fc^{\tens l}
\end{diagram}
and similarly for $\fd$ and \(\fc\boxt\fd\).
Hence,
\[ \hspace{-4.5em}
\begin{diagram}[h=0.9em,inline,nobalance]
\HmeetV &&\rLine^{\Delta^{(l)}\boxt\Delta^{(l)}} &&\HmeetV
\\
\uLine &&= &&\dTTo
\\
\\
\fc\boxt\fd &\rTTo^{\Delta^{(l)}} &(\fc\boxt\fd)^{\tens l} 
&\rTTo^{\text{unshuffle}}_\cong &\fc^{\tens l}\boxt\fd^{\tens l}
\\
\\
\dTTo<{\Bigl(
\begin{smallmatrix}
\pi\boxt\pi\\
\pi\boxt\eps\\
\eps\boxt\pi
\end{smallmatrix}
\Bigr)} &&\dTTo>{\Bigl(
\begin{smallmatrix}
\pi\boxt\pi\\
\pi\boxt\eps\\
\eps\boxt\pi
\end{smallmatrix}
\Bigr)^{\tens l}} 
&&\dTTo~{(\nu(i_1)\tdt\nu(i_l)\boxt\nu(j_1)\tdt\nu(j_l))_{i,j}}
\\
&= &&= &
\\
\begin{matrix}
\bar\fc\boxt\bar\fd \\
\oplus \\
\bar\fc\boxt\1\fd \\
\oplus \\
\1\fc\boxt\bar\fd
\end{matrix}
&\rTTo^{\bar\Delta^{(l)}} &
\begin{pmatrix}
\bar\fc\boxt\bar\fd \\
\oplus \\
\bar\fc\boxt\1\fd \\
\oplus \\
\1\fc\boxt\bar\fd
\end{pmatrix}
^{\tens l}
&\rTTo^{\text{unshuffle}}_\cong
&\bigoplus^{i_k,j_k\in\{0,1\}}_{i_k+j_k>0} 
\bar\fc^{\tens i_1}\tdt\bar\fc^{\tens i_l}
\boxt\bar\fd^{\tens j_1}\tdt\bar\fd^{\tens j_l}
\end{diagram}
\]
where \(\nu(0)=\eps\) and \(\nu(1)=\pi\).
Denote by \(p=\#\{k\mid i_k=1\}\) and \(q=\#\{k\mid j_k=1\}\) 
certain cardinalities.
The canonical isomorphism of the summand in the bottom right corner with
\(\bar\fc^{\tens p}\boxt\bar\fd^{\tens q}\) satisfies
\begin{align*}
\Delta^{(l)}\cdot(\nu(i_1)\tdt\nu(i_l))\cdot\text{iso} 
&= \Delta^{(p)}\cdot\pi^{\tens p} =\pi\cdot\bar\Delta^{(p)}: 
\fc \to \bar\fc^{\tens p},
\\
\Delta^{(l)}\cdot(\nu(j_1)\tdt\nu(j_l))\cdot\text{iso} 
&= \Delta^{(q)}\cdot\pi^{\tens q} =\pi\cdot\bar\Delta^{(q)}: 
\fd \to \bar\fd^{\tens q}.
\end{align*}
Clearly, \(p+q\ge l\).
Therefore, if \(c\in\bar\fc\), \(\bar\Delta^{(n)}c=0\) and 
\(d\in\bar\fd\), \(\bar\Delta^{(m)}d=0\), then the lower row applied 
to \(c\boxt d\) vanishes for \(l=n+m-1\).
If \(c\in\bar\fc\), \(\bar\Delta^{(n)}c=0\) and \(U\in\Ob\fd\), then the
lower row applied to \(c\boxt\eta_U\) ends up only in the summand with 
\(i_1=\dots=i_l=1\), \(j_1=\dots=j_l=0\).
Hence, \(\bar\Delta^{(l)}(c\boxt\eta_U)=0\) for $l=n$.
Similarly, if \(X\in\Ob\fc\) and \(d\in\bar\fd\), 
\(\bar\Delta^{(m)}d=0\), then \(\bar\Delta^{(m)}(\eta_X\boxt d)=0\).
\end{proof}

Being a full monoidal subcategory of $\acCat$ the category $\ncCat$ 
is symmetric.

\begin{definition}
Let \(f,g:\fa\to\fb\in\coCat\).
An \((f,g)\)\n-coderivation \(r:f\to g:\fa\to\fb\) of degree $d$ is a 
collection of morphisms \(r:\fa(X,Y)\to\fb(fX,gY)\) of degree $d$, which
satisfies the equation \(r\cdot\Delta=\Delta\cdot(f\tens r+r\tens g)\).
\end{definition}

The maps \(f\tens r,r\tens g:\fa\tens\fa\to\fb\tens\fb\) are defined 
similarly to \eqref{eq-f1tdtfn}.

\begin{proposition}\label{pro-coderivations-components-T}
Let \(f,g:\fa\to T\fb\in\ncCat\).
\((f,g)\)\n-coderivations \(r:f\to g:\fa\to T\fb\) of degree $d$ are in 
bijection with the collections of morphisms 
\(\check r=r\cdot \pr_1:\fa(X,Y)\to\fb(fX,gY)\) of degree $d$.
\end{proposition}

\begin{proof}
The commutative diagram
\begin{diagram}[width=6em]
&&T\fb &&
\\
&\ruTTo^r &\dTTo<{\Delta^{(k)}} &\rdTTo^{\pr_k} &
\\
\fa &&(T\fb)^{\tens k} &\rTTo^{\hspace*{-1em}\pr_1^{\tens k}} 
&\fb^{\tens k}
\\
& \rdTTo<{\Delta^{(k)}} 
&\uTTo~{\sum_{q+1+t=k}f^{\tens q}\tens r\tens g^{\tens t}} 
&\ruTTo_{\sum_{q+1+t=k}\check{f}^{\tens q}\tens\check{r}\tens\check{g}^{\tens t}} &
\\
&&\fa^{\tens k} &&
\end{diagram}
shows that $r$ is given by the formula
\begin{equation}
r =\sum_{k\ge1} \Delta^{(k)} \cdot 
\sum_{q+1+t=k}\check{f}^{\tens q}\tens\check{r}\tens\check{g}^{\tens t} 
=\Delta^{(3)}\cdot(f\tens\check{r}\tens g).
\label{eq-r=sum}
\end{equation}
The first expression makes sense, since in each term of decomposition of
\(\Delta^{(k)}\) there are factors of \(\bar\Delta^{(n)}\) and $k-n$ 
unit morphisms (elements \(\eta(1)\)).
The maps \(\check f\) and \(\check g\) vanish on the latter, hence, if 
\(\bar\Delta^{(k-1)}(x)=0\), then $k$\n-th term of \eqref{eq-r=sum} 
vanishes on $x$.
The second expression obviously makes sense.
\end{proof}

Let \(\fa,\fb\in\acCat\).
The coderivation quiver \(\Coder(\fa,\fb)\) has augmentation preserving 
cofunctors \(f:\fa\to\fb\) as objects and the $d$\n-th component of the 
graded $\Lambda$\n-module \(\Coder(\fa,\fb)(f,g)\) consists of 
coderivations \(r:f\to g:\fa\to\fb\) of degree $d$.
Notice that in \cite{Lyu-AinfCat} the notation \(\Coder(A,B)\) was used 
as a shorthand for \(\Coder(TA,TB)\) for $\cv$\n-quivers $A$ and $B$.

Let \(\phi:\fa\boxt\fc\to\fb\) be a cocategory homomorphism of degree 0.
By definition the homomorphism $\phi$ satisfies the equation
\begin{diagram}
\fa\boxt\fc &\rTTo^\phi &\fb &\rTTo^\Delta &\fb\tens\fb 
\\
\dTTo<{\Delta\boxt\Delta} &&= &&\uTTo>{\phi\tens\phi} 
\\
(\fa\tens\fa)\boxt(\fc\tens\fc) &&\rTTo^{1\tens c\tens1} 
&&(\fa\boxt\fc)\tens(\fa\boxt\fc)
\end{diagram}
Let \(c\in\fc^n\) (in the next several paragraphs $c$ does not mean the 
symmetry).
Introduce \(c\chi:\fa\to\fb\in\und{\cv\Quiv}^n\) by the formula 
\(a(c\chi)=(a\boxt c)\phi\).
Then the above equation is equivalent to
\begin{equation}
a(c\chi)\Delta =a\Delta(c_{(1)}\chi\tens c_{(2)}\chi).
\label{eq-a-chi-Delta}
\end{equation}
Another equation satisfied by $\phi$ is counitality: 
\((a\boxt c)\phi\eps\equiv a(c\chi)\eps=(a\eps)(c\eps)\).

Assume that \(\fa,\fb\in\coCat\) and \(\fc\in\acCat\).
Given a cofunctor \(\phi:\fa\boxt\fc\to\fb\) and an object \(C\in\fc\) 
there is a cofunctor \((\_\boxt C)\phi:\fa\to\fb\), which acts on 
objects as \((\_\boxt C)\phi:\Ob\fa\to\Ob\fb\), 
\(A\mapsto(A\boxt C)\phi\) and on morphisms as 
\((\_\boxt1_C\eta)\phi:\fa(A',A'')\to\fb((A'\boxt C)\phi,(A''\boxt C)\phi)\), 
\(a\mapsto(a\boxt1_C\eta)\phi\), where \((1_C)\eta\in\fc(C,C)\), 
\(1_C=1\in(\1\fc)(C,C)=\Lambda\).
If, furthermore, \(\fa\) and \(\fb\) are augmented and $\phi$ preserves 
augmentation, so does \((\_\boxt C)\phi\) for any \(C\in\Ob\fc\).

Assume that \(c\in\cc(f,g)^d\) satisfies 
\(c\Delta=1_f\eta\tens c+c\tens1_g\eta\).
Then the collection 
\(\xi:\fa(A',A'')\to\fb((A'\boxt f)\phi,(A''\boxt g)\phi)\), 
\(a\mapsto(a\boxt c)\phi\), is a 
\(((\_\boxt f)\phi,(\_\boxt g)\phi)\)-coderivation of degree $d$.

\subsection{Evaluation}\label{sec-Evaluation}
Let $\fa$ be a conilpotent cocategory and let $\fb$ be a $\cv$\n-quiver.
Define the evaluation cofunctor 
\(\phi=\ev:\fa\boxt T\Coder(\fa,T\fb)\to T\fb\) on objects as 
\(\ev(A\boxt f)=fA\), and on morphisms by the corresponding $\chi$.
Let \(f^0,f^1,\dots,f^n:\fa\to T\fb\) be cofunctors, and let 
\(r^1,\dots,r^n\) be coderivations of certain degrees as in 
\(f^0\rto{r^1} f^1\rto{r^2} \dots f^{n-1}\rto{r^n} f^n:\fa\to T\fb\), 
$n\ge0$.
Then \(c=r^1\tdt r^n\in T^n\Coder(\fa,T\fb)(f^0,f^n)\).
Define \((a\boxt c)\ev=a.(c\chi)\) as
\begin{equation}
(a\boxt(r^1\tdt r^n))\ev 
=(a\Delta^{(2n+1)})(f^0\tens\check r^1\tens f^1\tens\check r^2\tdt 
f^{n-1}\tens\check r^n\tens f^n)\mu_{T\fb}^{(2n+1)}.
\label{eq-ar1rnev}
\end{equation}
The right hand side belongs to 
\((T\fb)^{\tens(2n+1)}\mu_{T\fb}^{(2n+1)}\) and is mapped by 
multiplication \(\mu_{T\fb}^{(2n+1)}\) into \(T\fb\).
In particular, for $n=1$ we have \((a\boxt r^1)\ev=(a)r^1\) due 
to \eqref{eq-r=sum}.
In order to see that $\ev$ is a cofunctor we verify 
\eqref{eq-a-chi-Delta}:
\begin{align*}
&a.(c\chi)\Delta 
=a\Delta^{(2n+1)}(f^0\tens\check r^1\tens f^1\tens\check r^2\tdt 
f^{n-1}\tens\check r^n\tens f^n)\mu_{T\fb}^{(2n+1)}\Delta
\\
&= a\Delta^{(2n+1)}\sum_{m=0}^n(f^0\tens\check r^1\tdt 
f^{m-1}\tens\check r^m\tens f^m\Delta\tens\check r^{m+1}\tens 
f^{m+1}\tdt\check r^n\tens f^n)
\\
&\hspace{25em} (\mu_{T\fb}^{(2m+1)}\bigotimes\mu_{T\fb}^{(2n-2m+1)})
\\
&= a\Delta^{(2n+1)}\sum_{m=0}^n\bigl(f^0\tens\check r^1\tdt 
f^{m-1}\tens\check r^m\tens\Delta(f^m\bigotimes f^m)
\\
&\hspace{11em} \tens\check r^{m+1}\tens f^{m+1}\tdt\check r^n\tens 
f^n\bigr) (\mu_{T\fb}^{(2m+1)}\bigotimes\mu_{T\fb}^{(2n-2m+1)})
\\
&= a\Delta\sum_{m=0}^n(\Delta^{(2m+1)}\bigotimes\Delta^{(2n-2m+1)}) 
\bigl[(f^0\tens\check r^1\tdt f^{m-1}\tens\check r^m\tens f^m)
\mu_{T\fb}^{(2m+1)}
\\
&\hspace{13em} \bigotimes(f^m\tens\check r^{m+1}\tens 
f^{m+1}\tdt\check r^n\tens f^n)\mu_{T\fb}^{(2n-2m+1)}\bigr]
\\
&= a\Delta(c_{(1)}\chi\bigotimes c_{(2)}\chi).
\end{align*}
Here we have used $\tens$ for product in the tensor quiver $T\fb$ 
and $\bigotimes$ for \(T\fb\bigotimes T\fb\).
The counitality equation for $\ev$ has to be proven for $n=0$, 
where it reduces to counitality of $f^0$.
The cofunctor $\ev$ preserves the augmentation, since all 
\(f\in\Ob\Coder(\fa,T\fb)\) do.

The following is a version of Proposition~3.4 of \cite{Lyu-AinfCat}.

\begin{proposition}\label{prop-psi-ev}
For \(\fa\in\ncCat\), \(\fb,\fc^1,\dots\fc^q\in\cv\Quiv\) with notation 
\(\fc=T\fc^1\bdb T\fc^q\) the map
\begin{align*}
\ncCat(\fc,T\Coder(\fa,T\fb)) &\longrightarrow \ncCat(\fa\boxt\fc,T\fb),
\\
\psi &\longmapsto \bigl( \fa\boxt\fc \rTTo^{\fa\boxt\psi} 
\fa\boxt T\Coder(\fa,T\fb) \rto{\ev} T\fb \bigr) \notag
\end{align*}
is a bijection.
\end{proposition}

We give a new

\begin{proof}
An augmentation preserving cofunctor \(\psi:\fc\to T\Coder(\fa,T\fb)\) 
is described by an arbitrary quiver map
\(\check{\psi}=\psi\cdot\pr_1:\fc\to\Coder(\fa,T\fb)\in\Lambda\modulc_\LL\Quiv\)
such that \(\eta\cdot\check{\psi}=0\).
Let \(\phi:\fa\boxt\fc\to T\fb\) be an augmentation preserving cofunctor.
We have to satisfy the equation
\[ \sum_{k\ge0}(a\boxt c\Delta^{(k)}\check{\psi}^{\tens k})\ev 
=(a\boxt c)\phi, \qquad a\in\fa^\bull, \ c\in\fc^\bull.
\]
It suffices to consider two cases.
In the first one \(c=\eta(1_g)\) for some \(g\in\Ob\fc\).
Then the equation takes the form \((a)(g\psi)=(a\boxt c)\phi\) which 
defines the cofunctor \(g\psi\in\ncCat(\fa,T\fb)\) in the left hand side.

In the second case \(c\in\bar{\fc}^d\), \(d\in\ZZ\), the equation takes 
the form
\begin{equation*}
(a)(c)\check{\psi} 
+\sum_{k\ge2}(a\boxt c\Delta^{(k)}\check{\psi}^{\tens k})\ev
=(a\boxt c)\phi, \qquad a\in\fa^\bull, \ c\in\bar{\fc}^\bull.
\end{equation*}
Since \(\eta\cdot\check{\psi}=0\) the comultiplication $\Delta$ can be 
replaced with $\bar{\Delta}$.
The structure of \(\fc=T\fc^1\bdb T\fc^q\) is such that the component 
\(\psi_{i_1,\dots,i_q}\) in the left hand side of 
\begin{equation}
(a)(c)\check{\psi} =(a\boxt c)\phi 
-\sum_{k\ge2}(a\boxt c\bar{\Delta}^{(k)}\check{\psi}^{\tens k})\ev, 
\qquad a\in\fa^\bull, \ c\in\bar{\fc}^\bull.
\label{eq-acpsi-acphi}
\end{equation}
is expressed via the components \(\psi_{j_1,\dots,j_q}\) with smaller 
indices \((j_1,\dots,j_q)\) in the product poset \(\NN^q\).
For \(c\in\bar{\fc}(X,Y)^d\), \(X=(X_1,\dots,X_q)\), 
\(Y=(Y_1,\dots,Y_q)\), \(X_i,Y_i\in\Ob\fc^i\), find $n\ge0$ such that 
\(c\bar{\Delta}^{(n+1)}=0\).
Equation \eqref{eq-acpsi-acphi} determines a unique collection of maps
\(c\check{\psi}\in\und{\Lambda\modul}\bigl(\fa(U,V),\wh{T\fb}((U,X)\phi,(V,Y)\phi)\bigr)^d\).
It remains to verify that it is a coderivation.
We have to prove that
\[ (a)(c\check{\psi})\Delta_\fb 
=(a)\Delta_\fa[(\_\boxt X)\phi\tens(\_)(c\check{\psi}) 
+(\_)(c\check{\psi})\tens(\_\boxt Y)\phi].
\]
The case $n=0$ being obvious, assume that $n\ge1$.
The sum in \eqref{eq-acpsi-acphi} goes from $k=2$ to $n$.
Correspondingly,
\[ (a)(c\check{\psi})\Delta 
=(a\Delta)[(\_\boxt c_{(1)})\phi\tens(\_\boxt c_{(2)})\phi]
-\sum_{k=2}^n [(a\Delta)\boxt(c_{\bar1}\check{\psi}\tdt 
c_{\bar k}\check{\psi})\Delta_{T\Coder}]\tau_{(23)}(\ev\tens\ev).
\]
Here according to Sweedler's notation \(c_{(1)}\tens c_{(2)}=c\Delta\).
Similarly, \(c_{\bar1}\tdt c_{\bar k}=c\bar{\Delta}^{(k)}\).
Recall the middle four interchange 
\([(a\tens b)\boxt(c\tens d)]\tau_{(23)}=(-1)^{bc}(a\boxt c)\tens(b\boxt d)\).
The above expression has to be equal to
\begin{multline*}
(a\Delta)\Bigl\{(\_\boxt1_{X})\phi\tens\bigl[(\_\boxt c)\phi -\sum_{k=2}^n\bigl(\_\boxt(c_{\bar1}\check{\psi}\tdt 
c_{\bar k}\check{\psi})\bigr)\ev \bigr]\Bigr\}
\\
+(a\Delta)\Bigl\{\bigl[(\_\boxt c)\phi 
-\sum_{k=2}^n\bigl(\_\boxt(c_{\bar1}\check{\psi}\tdt 
c_{\bar k}\check{\psi})\bigr)\ev\bigr] \tens(\_\boxt1_{Y})\phi\Bigr\}.
\end{multline*}
Canceling the above terms we come to identity to be checked
\begin{equation}
(a\Delta)[(\_\boxt c_{\bar1})\phi\tens(\_\boxt c_{\bar2})\phi] 
=\sum_{k=2}^n \bigl[(a\Delta)\boxt(c_{\bar1}\check{\psi}\tdt 
c_{\bar k}\check{\psi})\bar\Delta_{T\Coder}\bigr]\tau_{(23)}(\ev\tens\ev).
\label{eq-c1phi-c2phi}
\end{equation}
The right hand side equals
\begin{align*}
&\sum_{k=2}^n\sum_{i=1}^{k-1} \Bigl\{(a\Delta)\boxt
\bigl[(c_{\bar1}\check{\psi}\tdt c_{\bar i}\check{\psi})
\bigotimes(c_{\overline{i+1}}\check{\psi}\tdt c_{\bar k}\check{\psi}) 
\bigr]\Bigr\}\tau_{(23)}(\ev\tens\ev) \notag
\\
&=\sum_{k=2}^n\sum_{i=1}^{k-1} (a\Delta) \Bigl\{
\bigl[\_\boxt(c_{\bar1}\check{\psi}\tdt c_{\bar i}\check{\psi})\bigl]\ev
\tens \bigl[\_\boxt(c_{\overline{i+1}}\check{\psi}\tdt 
c_{\bar k}\check{\psi})\bigr]\ev \Bigr\} \notag
\\
&=\sum_{i=1}^n\sum_{j=1}^n (a\Delta) \Bigl\{
\bigl[\_\boxt(c_{\bar1}\check{\psi}\tdt c_{\bar i}\check{\psi})\bigl]\ev
\tens \bigl[\_\boxt(c_{\overline{i+1}}\check{\psi}\tdt 
c_{\overline{i+j}}\check{\psi})\bigr]\ev \Bigr\} \notag
\\
&=(a\Delta)[(\_)(c_{\bar1}F)\tens(\_)(c_{\bar2}F)],
\end{align*}
where
\[ (a)(cF) =\sum_{i=1}^n\bigl[a\boxt(c_{\bar1}\check{\psi}\tdt 
c_{\bar i}\check{\psi})\bigl]\ev =(a)(c\check{\psi}) 
+\sum_{i=2}^n\bigl[a\boxt(c_{\bar1}\check{\psi}\tdt 
c_{\bar i}\check{\psi})\bigl]\ev =(a\boxt c)\phi
\]
due to \eqref{eq-acpsi-acphi}.
Hence the right hand side of \eqref{eq-c1phi-c2phi} equals 
\((a\Delta)[(\_\boxt c_{\bar1})\phi\tens(\_\boxt c_{\bar2})\phi]\), 
which is the left hand side of \eqref{eq-c1phi-c2phi}.
\end{proof}

Let $\fa$ be a conilpotent cocategory and let $\fb$, $\fc$ be quivers.
Consider the cofunctor given by the upper right path in the diagram
\begin{diagram}
	\fa\boxt T\Coder(\fa,T\fb)\boxt T\Coder(T\fb,T\fc) & 
	\rTTo^{\ev\boxt1} & T\fb\boxt T\Coder(T\fb,T\fc)
	\\
	\dTTo<{1\boxt M} & = & \dTTo>\ev
	\\
	\fa\boxt T\Coder(\fa,T\fc) & \rTTo^\ev & T\fc
\end{diagram}
By \propref{prop-psi-ev} there is a unique augmentation preserving 
cofunctor
\[ M: T\Coder(\fa,T\fb)\boxt T\Coder(T\fb,T\fc) \to T\Coder(\fa,T\fc).
\]

Denote by $\1$ the unit object $\boxt^0$ of the monoidal category of 
nilpotent cocategories, that is, $\Ob\1=\{*\}$, $\1(*,*)=\Lambda$.
Denote by $\rightunit:\fa\boxt\1\to\fa$ and 
$\leftunit:\1\boxt\fa\to\fa$ the corresponding natural isomorphisms.
By \propref{prop-psi-ev} there exists a unique augmentation preserving 
cofunctor $\eta_{T\fb}:\1\to T\Coder(T\fb,T\fb)$, such that
\[ \rightunit = \bigl(T\fb\boxt\1 \rTTo^{1\boxt\eta_{T\fb}} 
T\fb\boxt T\Coder(T\fb,T\fb) \rTTo^\ev T\fb\bigr).
\]
Namely, the object $*\in\Ob\1$ goes to the identity homomorphism 
$\id_{T\fb}:T\fb\to T\fb$.

The following statement (published as \cite[Proposition~4.1]{Lyu-AinfCat})
%, see also Kontsevich and Soibelman \cite{KonSoi-book}) 
follows from \propref{prop-psi-ev}.

\begin{proposition}%\label{pro-M-assoc-eta-unit}
The multiplication $M$ is associative and $\eta$ is its two-sided unit:
	\begin{diagram}[nobalance]
T\!\Coder(\fa,T\fb)\!\boxt\! T\!\Coder(T\fb,T\fc)\!\boxt\! T\!\Coder(T\fc,T\fd)
& \rTTo^{M\boxt1} & T\!\Coder(\fa,T\fc)\!\boxt\! T\!\Coder(T\fc,T\fd) 
		\\
		\dTTo<{1\boxt M} && \dTTo>M 
		\\
T\!\Coder(\fa,T\fb)\!\boxt\! T\!\Coder(T\fb,T\fd) &\rTTo^M	
&T\!\Coder(\fa,T\fd)
	\end{diagram}
\end{proposition}

The multiplication $M$ is computed explicitly in 
\cite[\S4]{Lyu-AinfCat}, see, in particular, Examples~4.2 there.

\section{Filtered cocategories}
Let $\LL$ be a partially ordered commutative monoid with the operation 
$+$ and neutral element 0.
Of course, we assume that \(a\le b\), \(c\le d\) imply \(a+c\le b+d\).
The subsets
\begin{align*}
\LL_+ &=\{l\in\LL\mid l\ge0\},
\\
\LL_- &=\{l\in\LL\mid l\le0\}
\end{align*}
are submonoids.
Clearly, \(\LL_+\cap\LL_-=\{0\}\).
We require that
\[ \LL_{++} =\{l\in\LL\mid l>0\} =\LL_+-{0}
\]
were non-empty.
We assume that $\LL$ satisfies the following conditions:
\begin{enumerate}
\renewcommand{\labelenumi}{(\roman{enumi})}
\item for all $a,b\in\LL$ there is $c\in\LL$ such that $a\le c$, 
$b\le c$ (that is, \((\LL,\le)\) is directed);

\item for all $a,b\in\LL$ there is $c\in\LL$ such that $c\le a$, 
$c\le b$ (that is, \(\LL^\op\) is directed);
		
\item for all $a,b\in\LL$ there is $c\in\LL$ such that $a+c\ge b$.	
\end{enumerate}
This generalizes the assumptions of \cite{MR3856926}.
If $\LL$ is a directed group (satisfies (i)), then $\LL$ satisfies (ii) 
and (iii) as well for obvious reasons.

The symmetric monoidal category of $\ZZ$\n-graded abelian groups 
(with the usual signed symmetry) is denoted $\grAb$.
$\LL$\n-filtered graded abelian group is a $\ZZ$\n-graded abelian group 
$M$ together with, for every \(l\in\LL\), a graded subgroup $\cf^lM$ 
such that \(a\le b\in\LL\) implies that \(\cf^aM\supset\cf^bM\) and 
\(\cup_{l\in\LL}\cf^lM=M\).
The symmetric multicategory \(\wh{\grAb_\LL}\) of $\LL$\n-filtered 
graded abelian groups is formed by polylinear maps of certain degree 
preserving the filtration:
\begin{align*}
\wh{\grAb_\LL}(M_1,\dots,M_n;N)^d = \{(\text{polylinear maps }f: 
M_1^{k_1}\times\dots\times{}&M_n^{k_n}\to 
N^{k_1+\dots+k_n+d})_{k_i\in\ZZ} \mid 
\\
\mid (\cf^{l_1}M_1^{k_1}\times\dots\times\cf^{l_n}M_n^{k_n})f 
&\subset\cf^{l_1+\dots+l_n}N^{k_1+\dots+k_n+d} \},
\end{align*}
$n\ge1$.
The sign for composition is the same as in 
\cite[Example~3.17]{BesLyuMan-book}.
This multicategory is representable 
\cite[Definition~3.23]{BesLyuMan-book} (see also 
\cite[Definition~8.3]{Hermida:MultiRep}) by a symmetric monoidal 
category which we denote \(\grAb_\LL\).
This follows from a similar statement for the closed multicategory 
$\wh{\Ab}$ of abelian groups.
One deduces the tensor product of a family $M_1$, \dots, $M_n$, $n\ge1$,
as the tensor product of $\ZZ$\n-graded abelian groups $M_i$, equipped 
with the filtration \cite[(2)]{MR3856926}
\begin{equation}
\cf^l(\tens_{i=1}^nM_i)
=\im(\oplus_{l_1+\dots+l_n=l}\tens_{i=1}^n\cf^{l_i}M_i 
\to \tens_{i=1}^nM_i).
\label{eq-Fl(xi1nMi)}
\end{equation}
Thus, \(\wh{\grAb_\LL}(M_1,\dots,M_n;N)\) is naturally isomorphic to 
\(\grAb_\LL(M_1\tdt M_n,N)\) for $n\ge1$ (more in 
\cite[Theorem~3.24]{BesLyuMan-book}).
The unit object is $\ZZ$, concentrated in degree 0, equipped with 
the filtration
\[ \cf^l\ZZ =
\begin{cases}
\ZZ, &\text{ if }l\le0,
\\
0, &\text{ otherwise}.
\end{cases}
\]
We define \(\wh{\grAb_\LL}(;N)\) as \(\grAb_\LL(\ZZ,N)\) in order 
to keep representability.

The monoidal category \(\grAb_\LL\) is symmetric with the signed 
symmetry of $\ZZ$\n-graded abelian groups.
Furthermore, it is closed.
In fact, let \(M,N\in \grAb_\LL\).
Associate with them a new graded $\LL$\n-filtered abelian group 
\(\und{\grAb_\LL}(M,N)\) with
\[ \cf^l\und{\grAb_\LL}(M,N)^d =\{f\in\und\grAb(M,N)^d \mid 
\forall\lambda\in\LL \; \forall k\in\ZZ \; 
(\cf^\lambda M^k)f \subset \cf^{\lambda+l}N^{k+d} \},
\]
the inner hom.
The evaluation
\[ \ev: M\tens\und{\grAb_\LL}(M,N) \to N, \quad m\tens f \mapsto (m)f,
\]
is a morphism of \(\grAb_\LL\), and it turns \(\grAb_\LL\) into 
a closed symmetric monoidal category.
Indeed, let \(\phi:M\tens P\to N\in\grAb_\LL\).
To any \(p\in P^d\), $d\in\ZZ$, assign a degree $d$ map 
\(\psi(p):M\to N\), \(m\mapsto\phi(m\tens p)\).
If \(p\in\cf^lP^d\), then \(\psi(p)\in\cf^l\und{\grAb_\LL}(M,N)^d\).
Hence a map \(\psi:P\to\und{\grAb_\LL}(M,N)\in\grAb_\LL\) such that
\begin{diagram}
M\tens P &&\rTTo^\phi &&N
\\
\dTTo<{1\tens\psi} &\!\!= &&\ruTTo(4,2)>\ev
\\
M\tens\und{\grAb_\LL}(M,N) &&&&.
\end{diagram}
Vice versa, given $\psi$ one obtains $\phi$ as the composition 
\((1\tens\psi)\cdot\ev\).
The two maps \(\phi\longleftrightarrow\psi\) are inverse to each other, 
and \(\grAb_\LL\) is closed.

According to \cite[Proposition~4.8]{BesLyuMan-book} the symmetric 
multicategory \(\wh{\grAb_\LL}\) is closed as well.
It is easy to describe the inner hom 
\(\wh{\und{\grAb_\LL}}(M_1,\dots,M_n;N)\in\Ob\wh{\grAb_\LL}\) via
\begin{align*}
\cf^l\wh{\und{\grAb_\LL}}(M_1,\dots,M_n;N)^d 
=\{(\text{polylinear maps }f: M_1^{k_1}\times\dots\times{}&M_n^{k_n}\to 
N^{k_1+\dots+k_n+d})_{k_i\in\ZZ} \mid 
\\
\mid (\cf^{l_1}M_1^{k_1}\times\dots\times\cf^{l_n}M_n^{k_n})f
\subset{}&\cf^{l_1+\dots+l_n+l}N^{k_1+\dots+k_n+d} \}.
\end{align*}
The corresponding evaluation is
\[ \ev: M_1,\dots,M_n,\wh{\und{\grAb_\LL}}(M_1,\dots,M_n;N) \to N, 
\quad (m_1,\dots,m_n,f) \mapsto (m_1,\dots,m_n)f.
\]

A commutative $\LL$\n-filtered graded ring $\Lambda$ is a commutative 
monoid (commutative algebra) in \(\grAb_\LL\).
Modules over $\Lambda$ in \(\grAb_\LL\) are called $\LL$\n-filtered 
$\ZZ$\n-graded $\Lambda$\n-modules and are identified with commutative 
$\Lambda$\n-bimodules (for short, $\Lambda$\n-modules).
In examples of interest (see \exaref{exa-Novikov-ring}) $\Lambda$ is 
$2\ZZ$\n-graded, so the commutativity issues for it are the same as in 
non-graded case.

Due to condition (i) a filtration \((\cf^\lambda M^k)_{\lambda\in\LL}\) 
on the graded $k$\n-th component $M^k$ viewed as a basis of 
neighborhoods of the origin defines a uniform structure on $M^k$ with 
the entourages \(\{(x,y)\in M\times M\mid x-y\in\cf^\lambda M^k\}\).
Standard properties of uniform structures are listed in 
\cite[Chap.~II, \S1, \S2]{zbMATH03395017}.

\begin{proposition}\label{pro-uniformly-continuous}
With the above uniform structure

(a) An element of \(\und{\grAb_\LL}(M,N)^d\) is a family of uniformly 
continuous maps $M^k\to N^{k+d}$.

(b) Each \(f\in\und{\grAb_\LL}(M_1,\dots,M_n;N)^d\) is a family of 
continuous maps 
\(f:M_1^{k_1}\times\dots\times M_n^{k_n}\to N^{k_1+\dots+k_n+d}\), where
\(M_i^{k_i}\), $N^k$ are given the topology, associated with the uniform
structure \cite[Chap.~II, \S1, n. 2, Definition~3]{zbMATH03395017}.

(c) If $\LL=\LL_+$, then each \(f\in\und{\grAb_\LL}(M_1,\dots,M_n;N)^d\)
is a family of uniformly continuous maps 
\(f:M_1^{k_1}\times\dots\times M_n^{k_n}\to N^{k_1+\dots+k_n+d}\).
\end{proposition}

\begin{proof}
(a) Let \(f\in\cf^l\und{\grAb_\LL}(M,N)^d\).
For any \(h\in\LL\) there exists \(\lambda\in\LL\) such that 
\(l+\lambda\ge h\) by condition (iii).
Then for arbitrary points $x,y\in M^k$ such that 
\(x-y\in\cf^\lambda M^k\) we have 
\(f(x)-f(y)=f(x-y)\in\cf^{l+\lambda}N^{k+d}\subset\cf^hN^{k+d}\).

(b) Fix a point 
\((y_1,\dots,y_n)\in M_1^{k_1}\times\dots\times M_n^{k_n}\).
There are $c_i\in\LL$ such that \(y_i\in\cf^{c_i}M_i^{k_i}\).
For an arbitrary \(\lambda\in\LL\) take \(\lambda_i\in\LL\) such that
\[ \lambda_i\ge c_i, \qquad \lambda_i +\sum_{j\neq i}c_j \ge\lambda.
\]
Consider the neighborhood of \(y_i\)
\[ \{x_i\in M_i^{k_i}\mid x_i-y_i\in\cf^{\lambda_i}M_i^{k_i}\} 
\subset \cf^{c_i}M_i^{k_i}.
\]
For $x_i$ from this neighborhood the element \(f(x_1,\dots,x_n)\) is in 
neighborhood of \(f(y_1,\dots,y_n)\), namely,
\begin{multline}
\begin{aligned}
f(x_1,\dots,x_n) -f(y_1,\dots,y_n) = &f(x_1-y_1,x_2,\dots,x_n) 
+f(y_1,x_2-y_2,x_3,\dots,x_n)
\\
&+\dots +f(y_1,\dots,y_{n-1},x_n-y_n) \in
\end{aligned}
\\
\in\cf^{\lambda_1+c_2+\dots+c_n}N^k+\cf^{c_1+\lambda_2+c_2+\dots+c_n}N^k
+\dots +\cf^{c_1+\dots+c_{n-1}+\lambda_n}N^k \subset\cf^\lambda N^k,
\label{eq-f(x1xn)-f(y1yn)}
\end{multline}
where \(k=k_1+\dots+k_n+d\).

(c) For any \(\lambda\in\LL\) and any points 
\((x_1,\dots,x_n),(y_1,\dots,y_n)\in M_1^{k_1}\times\dots\times M_n^{k_n}\)
if \(x_i-y_i\in\cf^\lambda M_i^{k_i}\), \(1\le i\le n\), then 
\(f(x_1,\dots,x_n) -f(y_1,\dots,y_n)\in\cf^\lambda N^{k_1+\dots+k_n+d}\)
similarly to \eqref{eq-f(x1xn)-f(y1yn)}.
\end{proof}

\subsection{Completion of a filtered graded abelian group.}
The notion of complete filtered abelian group is a particular case of a 
complete uniform space 
\cite[Chap.~II, \S3, n.3, Def.~3]{zbMATH03395017}.
There is the notion of separated completion (from now on 
\emph{completion}) $\hat{M}=(\hat M^k)$ of a uniform space $M=(M^k)$ 
\cite[Chap.~II, \S3, n.7, Th.~3]{zbMATH03395017}.
It consists of minimal Cauchy filters on $M=(M^k)$ 
\cite[Chap.~II, \S3, n.2]{zbMATH03395017}.
It is known that any Cauchy filter $F$ contains a unique minimal Cauchy 
filter $\und F$ \cite[Chap.~II, \S3, n.2, Prop.~5]{zbMATH03395017}.
A base of the filter $\und F$ can be obtained as a family 
\(\{A+\cf^\lambda M\mid A\in B,\,\lambda\in\LL\}\), where $B$ is a base 
of the filter $F$.

Consider now the graded abelian group 
\(\wh M=\lim_{\lambda\in\LL^\op}(M/\cf^\lambda M)\) equipped 
with filtration
\begin{equation}
\cf^l\wh M
=\lim_{\lambda\in\LL^\op}((\cf^l M+\cf^\lambda M)/\cf^\lambda M).
\label{eq-FlwhM}
\end{equation}
We understand the first limit as terminal cone on the functor 
\(\LL^\op\to\grAb\), \(\lambda\mapsto M/\cf^\lambda M\).
In our assumptions the non-empty subsemigroup \(\LL_{++}\) is a final 
subset of poset $\LL$.
Hence, 
\(\lim_{\lambda\in\LL^\op}(M/\cf^\lambda M)
=\lim_{\lambda\in\LL_{++}^\op}(M/\cf^\lambda M)\).
We are going to prove that $\wh{M}$ coincides with $\hat{M}$.
Until done we distinguish the two notations.

\begin{proposition}
When $\LL$ satisfies condition (i), the filtered graded abelian group 
$\wh{M}$ is complete.
\end{proposition}

\begin{proof}[(Seems known)]
It suffices to look at a graded component of $M$ which we still denote 
$M$.
The definition of completeness can be given also via Cauchy nets, 
namely, we have to prove that any Cauchy net in $\wh{M}$ converges.
A net is a mapping 
\(x:D\to\wh M=\lim_{\lambda\in\LL^\op}(M/\cf^\lambda M)\), 
\(d\mapsto x^d=([x^d_\lambda])_{\lambda\in\LL}\), where $D$ is 
a preordered directed set.
Classes \([x^d_\lambda]\in M/\cf^\lambda M\) lift to elements 
\(x^d_\lambda\in M\) such that for any $d\in D$ and for all 
\(a\le b\in\LL\) we have \(x^d_a-x^d_b\in\cf^aM\).
The net $x$ is Cauchy iff for every $l\in\LL$ there is \(N=N(l)\in D\) 
such that for all \(n,m\ge N\in D\) we have \(x^n-x^m\in\cf^l\wh M\).
The last condition reads: for every \(\lambda\in\LL\) we have 
\(x^n_\lambda-x^m_\lambda\in\cf^lM+\cf^\lambda M\).

Let $x$ as above be a Cauchy net.
Consider the collection 
\(y=(y_\lambda)\overset{\text{def}}=(x^{N(\lambda)}_\lambda)_{\lambda\in\LL}\).
Let us show that \(y\in\wh M\).
Recall that for $a\le b\in\LL$ there is $N\in D$ such that 
\(N\ge N(a)\), \(N\ge N(b)\).
Then \(y_a\equiv x^N_a\mod\cf^aM\), \(y_b\equiv x^N_b\mod\cf^bM\), 
\(x^N_a\equiv x^N_b\mod\cf^aM\).
Hence \(y_a\equiv y_b\mod\cf^aM\) and \(y\in\wh M\).
It follows from the condition in the first paragraph that $x$ converges 
to $y$, thus $\wh{M}$ is complete.
\end{proof}

\begin{proposition}
When $\LL$ satisfies condition (i), the separated completion \(\hat M\) 
of a filtered graded abelian group $M$ coincides with $\wh{M}$.
The filtration
\[ \cf^l\hat{M} 
=\{F\in\hat{M} \mid \exists\,0\in A\in F\ \ A\subset \cf^lM\}, 
\quad l\in\LL,
\]
on \(\hat M\) identifies with filtration \eqref{eq-FlwhM} on $\wh M$.
\end{proposition}

\begin{proof}
The canonical mapping \(\yii:M\to\wh M\) is uniformly continuous since 
for \(z\in\cf^lM\) we have  \(\yii(z)\in\cf^l\wh M\).
Moreover, the filtration on $M$ is a preimage of the filtration on 
$\wh{M}$ hence preimage of the filtration on its subset $\yii(M)$.
In fact, the image $\yii(z)$ of $z\in M$ is in \(\cf^l\wh M\) iff for 
all \(\lambda\in\LL\) we have \(z\in\cf^lM+\cf^\lambda M\).
For \(\lambda=l\) we get \(z\in\cf^lM\).
Furthermore, \(\wh M\) is separated.
In fact, if \(x=([x_\lambda])_{\lambda\in\LL}\) belongs to all 
\(\cf^l\wh M\) then \(x_\lambda\in\cf^lM+\cf^\lambda M\) for all 
\(l,\lambda\in\LL\), hence, \(x_\lambda\in\cf^\lambda M\) for all 
\(\lambda\in\LL\), that is, $x=0$.

Let us prove that the image \(\yii(M)\) is everywhere dense in \(\wh M\).
For any \(x=([x_\lambda])\in\wh M\) and any \(l\in\LL\) we have to 
provide an element \(y\in M\) such that \(x-\yii(y)\in\cf^l\wh M\).
The last condition reads: for all \(\lambda\in\LL\) we have 
\(x_\lambda-y\in\cf^lM+\cf^\lambda M\).
Take \(y=x_l\).
By assumption for all \(l,\lambda\in\LL\) there is \(c\in\LL\) 
such that \(l\le c\), \(\lambda\le c\).
Hence, \(x_\lambda-y=x_\lambda-x_c+x_c-x_l\in\cf^\lambda M+\cf^lM\) 
as required.

Now we can construct the following commutative diagram
\begin{diagram}[h=0.85em]
\HmeetV &&\rLine_\yi &&\HmeetV
\\
\vLine &&&&\dTTo
\\
\\
M &\rEpi^{\overline{\yi}} &\yi(M) &\rMono^{\underline{\yi}} &\hat{M} 
&\rTTo^\cong &\yi(M)\sphat
\\
&\rdEpi(2,4)_{\overline{\yii}} &&&&\ldTTo(2,4)<\cong>h
\\
\dLine &&\dTTo<\cong>g &&\dTTo<\cong>f
\\
\\
&&\yii(M) &\rMono^{\underline\yii} &\wh M
\\
&&&&\uTTo
\\
\HmeetV &&\rLine^\yii &&\HmeetV
\end{diagram}
where \(\yi:M\to\hat{M}\) is the canonical mapping, sending a point 
to the filter of neighborhoods of this point.
It is denoted by $i$ in \cite[Chap.~II, \S3, n.7, 2)]{zbMATH03395017}.
The isomorphism \(\hat{M}\to\yi(M)\sphat\) follows by Propositions 
12.$1^\circ$ and 13 of \cite[Chap.~II, \S3, n.7]{zbMATH03395017}.
The morphism $g$ exists by Proposition~16 and it is invertible by 
Proposition~17 of \cite[Chap.~II, \S3, n.8]{zbMATH03395017}.
The isomorphism $h$ exists by 
\cite[Chap.~II, \S3, n.7, Prop.~13]{zbMATH03395017}.
Hence an isomorphism $f$.

The filter \(\co=\yi(0)\in\hat{M}\) of neighborhoods of $0\in M$ is 
a minimal Cauchy filter.
It has the base \(\{\cf^\lambda M\mid \lambda\in\LL\}\).
Filtration on \(\hat{M}\) consists of
\begin{align*}
\cf^l\hat{M} 
&=\{F\in\hat{M} \mid \exists\,A\in F\cap\co\ \ A-A\subset\cf^lM\}
\\
&=\{F\in\hat{M} \mid \exists\,A\in F \; \exists\lambda\in\LL\ \ A\supset
\cf^\lambda M,\ A-A\subset\cf^lM\}
\\
&=\{F\in\hat{M} \mid \exists\,0\in A\in F\ \ A\subset\cf^lM\}
\end{align*}
\cite[Chap.~II, \S3, n.7, 1)]{zbMATH03395017} (since Cauchy filter $F$ 
is minimal, its base consists of sets invariant under addition of 
\(\cf^\lambda M\) for some $\lambda\in\LL$, which we may assume $\ge l$).

Both \(\cf^l\hat{M}\) and \(\cf^l\wh{M}\) induce on $M$ the same 
subspace \(\cf^lM\) via pull-back with bases \(\yi:M\to\hat{M}\) and 
\(\yii:M\to\wh M\) \cite[Chap.~II, \S3, n.7, 2)]{zbMATH03395017}.
The images \(\yi(M)\subset\hat{M}\), \(\yii(M)\subset\wh{M}\) are dense 
\cite[Chap.~II, \S3, n.7, 3)]{zbMATH03395017}, hence, filtrations 
\(\cf^l\hat{M}\) and \(\cf^l\wh{M}\) are taken to each other under the 
isomorphism \(f:\hat{M}\to\wh{M}\) and its inverse.
\end{proof}

From now on we do not distinguish $\hat{M}$ and $\wh M$.

The mapping \(M\mapsto\hat{M}\) extends to the completion functor 
\(\hat{\text{-}}:\grAb_\LL\to\grAb_\LL\) 
\cite[Chap.~II, \S3, n.7, Prop.~15]{zbMATH03395017} in a unique way so 
that the maps \(\yii_M:M\to\hat{M}\) form a natural transformation.
A filtered graded abelian group $M$ is complete, when the canonical map 
\(\yii:M\to\hat M\) is an isomorphism.
The $\grAb_\LL$\n-category \(\cgrAb_\LL\) of complete $\LL$\n-filtered 
graded abelian groups is a reflective subcategory of \(\grAb_\LL\).
This follows by the remark that complete topological abelian groups form
a reflective subcategory of the category of topological abelian groups.
Thus by \cite[Corollary~4.2.4]{Borceux:CatHandbook2} (see enriched 
version at the end of Chapter~1 of \cite{KellyGM:bascec}) 
the completion functor is an idempotent monad.
In particular, for the unit of this monad \(\yii_M:M\to\hat{M}\), the 
morphisms \(\yii_{\hat{M}}=\wh{\yii_M}:\hat{M}\to\hat{\hat{M}}\) are 
inverse to the multiplication 
\(\mu_M:\hat{\hat{M}}\to\hat{M}\) (cf. \cite[Lemma~2.24]{MR3856926}).
It follows from \apref{ap-A-Reflective} that the reflective subcategory 
\(\cgrAb_\LL\) is symmetric monoidal with the monoidal product 
\(M\hat{\tens}N\overset{\text{def}}=\wh{M\tens N}\).
The unit object is still $\ZZ$.
We extend the functor \(\hat{\tens}^n:\cgrAb_\LL^n\to\cgrAb_\LL\) to 
\(\hat{\tens}^n:\grAb_\LL^n\to\cgrAb_\LL\) via the same recipe
\(\hat{\tens}_{i=1}^nM_i\overset{\text{def}}=\wh{\tens_{i=1}^nM_i}\).

\begin{proposition}\label{pro-nat-trans-phi}
For $n\ge1$ there is a natural transformation
\(\phi^n:\tens_{i=1}^n\wh{M_i}\to\wh{\tens_{i=1}^nM_i}:
\grAb_\LL^n\to\grAb_\LL\).
\end{proposition}

\begin{proof}
Let $M_i$, $1\le i\le n$,  be filtered abelian groups.
Let $F_i\in\hat M_i$ be minimal Cauchy filters in $M_i$, $1\le i\le n$.
Denote $M=\tens_{i=1}^nM_i$.
Define a basis $B$ of a filter $F$ in $M$ as
\[ B =\{A_1\tens A_2\tdt A_n \mid \forall i\, A_i\in F_i\},
\]
where 
$A_1\tens A_2\tdt A_n
=\{x_1\tens x_2\tdt x_n \mid \forall i\,x_i \in A_i\}$.
Let us prove that $F$ is a Cauchy filter.
Given $\lambda\in\LL$, take for $1\le i\le n$ arbitrary elements 
$a_i\in\LL$, take $A_i\in F_i$ such that 
$A_i-A_i\equiv\{x-y\mid x,y\in A_i\}\subset\cf^{a_i}M_i$, take 
arbitrary elements $y_i\in A_i$.
Let $b_i\in \LL$ be such that $y_i\in\cf^{b_i}M_i$.
Let $c_i\in\LL$ be such that $c_i\le a_i$ and $c_i\le b_i$.
Then $A_i\subset\cf^{c_i}\LL$ since any $x\in A_i$ can be presented as $x=x-y_i+y_i\in\cf^{c_i}M_i+\cf^{b_i}M_i\subset\cf^{c_i}M_i$. 
Let $\lambda_i\in\LL$, $1\le i\le n$, be such that
\[ \lambda_i +\sum_{j\neq i}c_j \ge\lambda.
\]
Let $B_i\in F_i$ satisfy $B_i-B_i\subset\cf^{\lambda_i}M_i$.
Then the set $S=(A_1\cap B_1)\tdt(A_n\cap B_n)\in F$ satisfies 
$S-S\subset\cf^\lambda M$.
In fact, for $x_i,y_i\in A_i\cap B_i$ we have
\begin{align}
x_1\tdt x_n -y_1\tdt y_n &=(x_1-y_1)\tens x_2\tdt x_n 
+y_1\tens(x_2-y_2)\tens x_3\tdt x_n \notag
\\
&+\dots +y_1\tdt y_{n-1}\tens(x_n-y_n) \in \notag
\\
\in\cf^{\lambda_1}M_1\tens\cf^{c_2}M_2\tens\cdots&\tens\cf^{c_n}M_n 
+\cf^{c_1}M_1\tens\cf^{\lambda_2}M_2\tens\cf^{c_3}M_3\tdt\cf^{c_n}M_n 
+\dots \notag
\\
&+\cf^{c_1}M_1\tdt\cf^{c_{n-1}}M_{n-1}\tens\cf^{\lambda_n}M_n 
\subset\cf^\lambda M.
\label{eq-x1tdtxn-y1tdtyn}
\end{align}
The Cauchy filter $F$ contains a unique minimal Cauchy filter $\und F$ 
\cite[Chap.~II, \S3, n.2, Prop.~5]{zbMATH03395017} and we define 
$\phi^n$ as a map sending \(F_1\tdt F_n\) to $\und F$.
The outcome does not depend on the choices made during the construction.
In fact, axioms on $\LL$ and on filters ensure that two different 
choices $F'$ and $F''$ for $F$ are contained in a third choice $F'''$ of
Cauchy filter $F$, hence, for minimal Cauchy filters we have 
\(\und F'=\und F'''=\und F''\).

This is a well-defined mapping $\tilde\phi^n$ from the free graded 
abelian group generated by $n$\n-tuples \((F_1,\dots,F_n)\) to $\hat M$.
One has
\[ \tilde\phi^n(F_1,\dots,F_i'+F_i'',\dots,F_n) 
=\tilde\phi^n(F_1,\dots,F_i',\dots,F_n) 
+\tilde\phi^n(F_1,\dots,F_i'',\dots,F_n)
\]
if \(F_i',F_i''\in\hat{M}_i^k\), $1\le i\le n$, $k\in\ZZ$.
In fact, sum of Cauchy filters $G'$, $G''$ is defined as the filter $G$,
generated by $A'+A''$, $A'\in G'$, $A''\in G''$.
Clearly, $G$ is a Cauchy filter.
Thus, the map $\tilde\phi^n$ factors through a well-defined map 
\(\phi^n:\tens_{i=1}^n\wh{M_i}\to\wh{\tens_{i=1}^nM_i}\).

Notice that
\[ \phi^n(\tens_{i=1}^n\cf^{\lambda_i}\hat{M_i}) \subset 
\cf^{\lambda_1+\dots+\lambda_n} \wh{\tens_{i=1}^nM_i}.
\]
In fact, let \(F_i\in\cf^{\lambda_i}\hat{M_i}\), \(1\le i\le n\).
Thus, $F_i$ is a minimal Cauchy filter such that there is 
\(0\in A_i\in F_i\), \(A_i\subset\cf^{\lambda_i}M_i\).
Then
\[ \phi^n(F_1\tdt F_n) \ni A_1\tdt A_n +\cf^\nu(M_1\tdt M_n)
\]
for any \(\nu\in\LL\).
However, \(0\in A_1\tdt A_n\) and
\[ A_1\tdt A_n \subset \cf^{\lambda_1}M_1\tdt\cf^{\lambda_n}M_n \subset
\cf^{\lambda_1+\dots+\lambda_n}(\tens_{i=1}^nM_i).
\]
Hence, for \(\nu\ge\lambda_1+\dots+\lambda_n\)
\[ 0\in A_1\tdt A_n +\cf^\nu(M_1\tdt M_n) \subset 
\cf^{\lambda_1+\dots+\lambda_n}(\tens_{i=1}^nM_i).
\]
Therefore, 
\(\phi^n(F_1\tdt F_n)\in
\cf^{\lambda_1+\dots+\lambda_n}\wh{\tens_{i=1}^nM_i}\).

Reasonings, similar to independence of choices show that $\phi^n$ 
form a natural transformation.
\end{proof}

For $n=0$ the version of $\phi^n$ is the isomorphism 
\(\phi^0=\yii:\ZZ\to\wh\ZZ\).
The filtered abelian group $\ZZ$ is complete due to non-emptiness of 
\(\LL_{++}\).
In fact, \(\wh\ZZ=\lim_{\lambda\in\LL_{++}^\op}\ZZ=\ZZ\).

\begin{proposition}
The pair \((\hat{\text-},\phi^\bull):\grAb_\LL\to\grAb_\LL\) is a lax 
symmetric monoidal functor.
\end{proposition}

\begin{proof}
Naturality (in the ordinary everyday usage sense) of the construction of
$\phi^n$ leads to required condition from \cite{DayStreet-subst}, see also 
diagram~(2.17.2) of \cite{BesLyuMan-book}.
\end{proof}

According to \cite[Proposition~3.28]{BesLyuMan-book} \(\phi^\bull\) make
completion \(\hat{\text-}\) also into a symmetric multifunctor 
\(\wh{\grAb_\LL}\to\wh{\grAb_\LL}\).

\begin{proposition}\label{pro-yiyi-phi=yi}
The canonical mapping \(\yi:M \to\hat{M}\) satisfies for $n\ge0$ the 
equation
\begin{equation}
\bigl(M_1\tdt M_n \rTTo^{\yi_1\tdt\yi_n} \hat{M}_1\tdt\hat{M}_n 
\rTTo^{\phi^n} \wh{M_1\!\!\tens\!\cdot\!\!\cdot\!\!\cdot\!\tens\!\!M_n}
\bigr) =\yi_{M_1\tdt M_n}.
\label{eq-yiyi-phi=yi}
\end{equation}
\end{proposition}

\begin{proof}
For $n\ge1$ we may consider a graded component of $M_i$, a filtered 
abelian group, which we denote again by $M_i$.
Take elements \(y_i\in M_i\).
For some \(c_i\in\LL\) we have \(y_i\in\cf^{c_i}M_i\).
The filter \(\yi(y_i)\in\hat{M}_i\) has the base formed by 
\(y_i+\cf^{\lambda_i}M_i\), \(\lambda_i\in\LL\).
Thus, the Cauchy filter $F$ with the basis 
\((y_1+\cf^{\lambda_1}M_1)\tdt(y_n+\cf^{\lambda_n}M_n)\) contains the 
minimal Cauchy filter 
\(\phi^n(\yi(y_1)\tdt\yi(y_n))\in\wh{M_1\tdt M_n}\).
For any \(\lambda\in\LL\) there are \(\lambda_i\in\LL\) such that
\[ \lambda_i\ge c_i, \qquad \lambda_i +\sum_{j\neq i}c_j \ge\lambda.
\]
Using \eqref{eq-x1tdtxn-y1tdtyn} we see that the last set is contained 
in \(y_1\tdt y_n+\cf^\lambda(M_1\tdt M_n)\).
Hence, $F$ contains the minimal Cauchy filter of neighborhoods of 
\(y_1\tdt y_n\).
Therefore, \(\phi^n(\yi(y_1)\tdt\yi(y_n))=\yi(y_1\tdt y_n)\) 
by uniqueness of the minimal Cauchy subfilter.

The case of $n=0$ is straightforward.
\end{proof}

\subsection{Complete \texorpdfstring{$\Lambda$}{Λ}-modules.}\label{sec-Complete-Lambda-modules}
From now on, \emph{the graded commutative filtered ring $\Lambda$} will 
be \textbf{complete}.

For the moment $\cv=\Lambda\modul_\LL$ means the category of 
$\LL$\n-filtered graded $\Lambda$\n-modules for a graded commutative 
$\LL$\n-filtered ring $\Lambda$.
Morphisms are grading and filtration preserving $\Lambda$\n-module maps.
It is symmetric monoidal with the tensor product \(M\tens_\Lambda N\) 
equipped with filtration \eqref{eq-Fl(xi1nMi)}, where, of course, 
$\tens$ has to be interpreted as $\tens_\Lambda$, not as $\tens_\ZZ$.
The unit object $\1$ is $\Lambda$ with its filtration.
The category \(\Lambda\modul_\LL\) is closed.
In fact, let \(M,N\in\Lambda\modul_\LL\).
Associate with them a new graded $\LL$\n-filtered $\Lambda$\n-module 
\(\und{\Lambda\modul_\LL}(M,N)\) with
\[ \cf^l\und{\Lambda\modul_\LL}(M,N)^d =\{f\in\und{\Lambda\modul}(M,N)^d
\mid \forall\lambda\in\LL \; \forall k\in\ZZ \; (\cf^\lambda M^k)f 
\subset \cf^{\lambda+l}N^{k+d} \},
\]
the inner hom.
The evaluation
\[ \ev: M\tens_\Lambda\und{\Lambda\modul_\LL}(M,N) \to N, 
\quad m\tens f \mapsto (m)f,
\]
is a morphism of \(\Lambda\modul_\LL\), and it turns this category into 
a closed symmetric monoidal one.
Proof is the same as in \(\grAb_\LL\) case.
All definitions and notions of Section~\ref{sec-intro} apply for this 
$\cv$.
Note that the uniform space associated with the product 
\(\prod_{i\in I}M_i\) in \(\Lambda\modul_\LL\) (over an infinite 
set $I$) differs from product of uniform spaces $M_i$.

A $\Lambda$\n-module $M$ is complete, when the canonical map 
\(\yi:M\to\hat M\) is an isomorphism.
The category of complete $\Lambda$\n-modules \(\Lambda\modulc_\LL\) is 
a full \(\Lambda\modul_\LL\)-subcategory of \(\Lambda\modul_\LL\).

\begin{remark}
Let \((\fc,\Delta,\eps)\) be a cocategory.
Then the completion \(\hat\fc\) equipped with the comultiplication 
\(\big(\hat\fc\rTTo^{\wh\Delta} \wh{\fc\tens\fc}\rTTo^{\wh{\yi\tens\yi}} (\hat\fc\tens\hat\fc)\sphat\equiv\hat\fc\hat\tens\hat\fc\big)\) 
and the counit \(\hat\fc\rTTo^{\wh\eps} \wh{\1\fc}=\hat\1\hat\fc\) is 
a cocategory over $\Lambda$ with respect to monoidal structure 
$\hat\tens$, see \apref{sec-Completion-coalgebras}.
\end{remark}

\begin{example}\label{exa-Novikov-ring}
The universal Novikov ring
\[ \Lambda_{nov}(R) =\bigg\{\sum_{i=0}^{\infty}a_iT^{\lambda_i}e^{n_i} 
\mid \forall i\;a_i\in R,\;\lambda_i\in\RR,\;n_i\in\ZZ,\;
\lim_{i\to\infty}\lambda_i=\infty\bigg\}
\]
contains a subring, the Novikov ring,
\[ \Lambda_{0,nov}(R) =\bigg\{\sum_{i=0}^{\infty}a_iT^{\lambda_i}e^{n_i}
\in\Lambda_{nov}(R) \mid \forall i\;\lambda_i\ge0 \bigg\}
\]
The grading is determined by \(\deg R=0\), \(\deg T=0\), \(\deg e=2\) 
\cite[\S1.7 (Conv. 4)]{FukayaOhOhtaOno:Anomaly}.
\(\Lambda_{nov}(R)\) is $\RR$\n-filtered according to 
\cite[\S1.7 (Conv. 6)]{FukayaOhOhtaOno:Anomaly}.
The filtration is
\[ \cf^\lambda\Lambda_{nov}(R) =T^\lambda\Lambda_{0,nov}(R), 
\qquad \lambda\in\RR.
\]
Similarly, \(\Lambda_{0,nov}(R)\) is $\RR_{\ge0}$\n-filtered by
\[ \cf^\lambda\Lambda_{0,nov}(R) =T^\lambda\Lambda_{0,nov}(R), 
\qquad \lambda\in\RR_{\ge0}.
\]
These rings are complete [\textit{ibid}].
\end{example}

\subsection{Complete cocategories.}\label{sec-Complete-cocategories}
The set-up of \apref{ap-A-Reflective} applies well to symmetric monoidal
$\Lambda\modul_\LL$\n-category \(\cd=(\Lambda\modul_\LL\Quiv,\boxt^I)\) 
and its reflective subcategory of complete quivers 
\(\cc=\Lambda\modulc_\LL\Quiv\).
According to \propref{prop-C-lax-monoidal} the category $\cc$ is lax 
symmetric monoidal with the product 
\(\wh\boxt^{i\in I}\fa_i=\wh{\boxt^{i\in I}\fa_i}\).
The completion of a filtered quiver \(\fa\in\Ob\cd\) is given by the 
quiver \(\hat{\fa}\in\Ob\cc\) with \(\Ob\hat{\fa}=\Ob\fa\) and 
\(\hat{\fa}(X,Y)=\wh{\fa(X,Y)}\) for \(X,Y\in\Ob\fa\).
Hence, \(\yi:\fa\to\hat{\fa}\) is given by the morphisms 
\(\yi:\fa(X,Y)\to\wh{\fa(X,Y)}\), \(X,Y\in\Ob\fa\).
\propref{pro-nat-trans-phi} implies the existence of a natural 
transformation 
\(\phi^n:\boxt_{i=1}^n\wh{\fa_i}\to\wh{\boxt_{i=1}^n\fa_i}\).
According to \propref{pro-yiyi-phi=yi} the equation
\[ \bigl(\fa_1\bdb\fa_n \rTTo^{\yi_1\bdb\yi_n} 
\hat{\fa}_1\bdb\hat{\fa}_n \rTTo^{\phi^n} 
\wh{\fa_1\!\!\boxt\!\cdot\!\!\cdot\!\!\cdot\!\boxt\!\!\fa_n} \bigr) 
=\yi_{\fa_1\bdb\fa_n}
\]
holds.
Therefore, the conclusion of \propref{pro-1yi1yi1yi1-invertible} holds 
true and by \corref{cor-D-monoidal-C-monoidal} we find that 
\(\cc=(\Lambda\modulc_\LL\Quiv,\wh\boxt^I)\) is a symmetric monoidal 
$\Lambda\modul_\LL$\n-category.

Fix a (large) set $S$.
Consider $\Lambda\modul_\LL$\n-category \(\cd=\Lambda\modul_\LL\Quiv_S\)
and its reflective subcategory of complete quivers 
\(\cc=\Lambda\modulc_\LL\Quiv_S\).
We apply the results of \apref{ap-A-Reflective} to this situation 
as well.
Again, the conclusions of Propositions \ref{pro-nat-trans-phi} and 
\ref{pro-yiyi-phi=yi} hold for \(M_i\in\Lambda\modul_\LL\Quiv_S\), 
hence, \(\cd=(\Lambda\modul_\LL\Quiv_S,\tens^I)\) and 
\(\cc=(\Lambda\modulc_\LL\Quiv_S,\hat\tens^I)\) are monoidal 
$\Lambda\modul_\LL$\n-categories.
In particular, the construction of \apref{sec-Completion-coalgebras} 
applies.

It will be shown later that the following simple definition is 
equivalent to Definition~2.40 of \cite{MR3856926}.

\begin{definition}
Let $\fa$ be a complete filtered quiver and $S$ a set.
A morphism \(\phi:\1S\to\fa\in\Lambda\modulc_\LL\Quiv\) is called 
\emph{tensor convergent} if for every $l\in\LL$ and every $X\in S$ there
exists $N\in\NN$ such that for every $n\ge N$
\[ [\phi(1_X)]^{\hat{\tens}n} \in 
\cf^l[\fa(\phi X,\phi X)^{\hat{\tens}n}].
\]
\end{definition}

\begin{lemma}\label{lem-phipsi-tensor_convergent-phi+psi}
If \(\phi,\psi:\1S\to\fc\) are tensor convergent and 
\(\Ob\phi=\Ob\psi:S\to\Ob\fc\), then \(\phi+\psi\) is tensor convergent 
as well.
\end{lemma}

\begin{proof}
Fix an element \(x\in S\) and an element $l\in\LL$.
Consider \(\lambda\in\LL_+\) such that \(\lambda\ge l\).
Denote by $Y$ the $\Lambda$\n-module \(\fc(\phi X,\phi X)\).
Let $K\in\NN$ (resp. $M\in\NN$) be such that for every $k\in\NN$, 
$k\ge K$ (resp. $m\in\NN$, $m\ge M$) we have 
\(\phi(1_X)^{\hat{\tens}k}\in\cf^\lambda(Y^{\hat{\tens}k})\) (resp. \(\psi(1_X)^{\hat{\tens}m}\in\cf^\lambda(Y^{\hat{\tens}m})\)).
Let \(\phi(1_X)^{\hat{\tens}k}\in\cf^{c_k}(Y^{\hat{\tens}k})\) for 
\(0\le k<K\) and let $c\in\LL_+$ be such that \(c+c_k\ge0\) for 
\(0\le k<K\).
Let \(\psi(1_X)^{\hat{\tens}m}\in\cf^{b_m}(Y^{\hat{\tens}m})\) for 
\(0\le m<M\) and let $b\in\LL_+$ be such that \(b+b_m\ge0\) for 
\(0\le m<M\).
Let $N\in\NN$ (resp. $P\in\NN$) be such that for every $k\in\NN$, 
$k\ge N$ (resp. $m\in\NN$, $m\ge P$) we have 
\(\phi(1_X)^{\hat{\tens}k}\in\cf^{\lambda+b}(Y^{\hat{\tens}k})\) (resp. \(\psi(1_X)^{\hat{\tens}m}\in\cf^{c+\lambda}(Y^{\hat{\tens}m})\)).
Set \(Q=1+\max\{K+P,N+M\}\).
For any $n\ge Q$ the $2^n$ summands of 
\((\phi(1_X)+\psi(1_X))^{\hat{\tens}n}\) are identified with one of the 
summands $\phi(1_X)^{\hat{\tens}a}\hat{\tens}\psi(1_X)^{\hat{\tens}d}$, 
\(a+d=n\), \(a,d\in\NN\), with the use of symmetry, preserving 
the filtration.
If $a<K$, then $d\ge P$ and
\[ \phi(1_X)^{\hat{\tens}a}\hat{\tens}\psi(1_X)^{\hat{\tens}d} \in
\cf^{c_a}(Y^{\hat{\tens}a})\hat{\tens}\cf^{c+\lambda}(Y^{\hat{\tens}d}) 
\subset \cf^\lambda(Y^{\hat{\tens}n}).
\]
If $d<M$, then $a\ge N$ and
\[ \phi(1_X)^{\hat{\tens}a}\hat{\tens}\psi(1_X)^{\hat{\tens}d} \in
\cf^{\lambda+b}(Y^{\hat{\tens}a})\hat{\tens}\cf^{b_d}(Y^{\hat{\tens}d}) 
\subset \cf^\lambda(Y^{\hat{\tens}n}).
\]
It remains to consider the case $a\ge K$, $d\ge M$.
Then
\[ \phi(1_X)^{\hat{\tens}a}\hat{\tens}\psi(1_X)^{\hat{\tens}d} \in
\cf^\lambda(Y^{\hat{\tens}a})\hat{\tens}\cf^\lambda(Y^{\hat{\tens}d}) 
\subset \cf^{2\lambda}(Y^{\hat{\tens}n}).
\]
Hence, 
\((\phi(1_X)+\psi(1_X))^{\hat{\tens}n}\in
\cf^\lambda(Y^{\hat{\tens}n})\subset\cf^l(Y^{\hat{\tens}n})\).
\end{proof}

\begin{remark}\label{rem-tensor-convergent}
For any map $f:Q\to S$ and tensor convergent \(\phi:\1S\to\fa\) the map 
$\1f\cdot\phi$ is tensor convergent as well.
For any morphism \(g:\fa\to\fb\in\Lambda\modulc_\LL\Quiv\) and any 
$X\in S$ the map 
\(g^{\hat{\tens}n}:\fa(\phi X,\phi X)^{\hat{\tens}n}\to
\fb(g\phi X,g\phi X)^{\hat{\tens}n}\) 
is in \(\Lambda\modulc_\LL\), hence, \(\phi\cdot g\) is 
tensor convergent as well.
\end{remark}

\begin{definition}
Inspired by Definitions \ref{def-cocategory}, 
\ref{def-augmented-cocategory}, \ref{def-conilpotent-cocategory} we say 
that a \emph{completed conilpotent cocategory} $\fC=\hat{\fc}$ is a 
completion (as a filtered quiver) of a conilpotent cocategory $\fc$.
It is itself a cocategory (with respect to \(\hat{\tens}\)) equipped 
with the comultiplication 
\(\Delta_{\hat{\fc}}=(\hat{\fc}\rto{\wh{\Delta_\fc}} \wh{\fc\tens \fc}
\rTTo^{\wh{\yi\tens\yi}} \wh{\hat{\fc}\tens\hat{\fc}}
=\hat{\fc}\hat{\tens}\hat{\fc})\), 
the counit 
\(\eps_{\hat{\fc}}=\wh{\eps_\fc}:
\hat{\fc}\to\hat{\Lambda}\hat{\fc}=\Lambda\fc\) 
and the augmentation 
\(\eta_{\hat{\fc}}=\wh{\eta_\fc}:
\Lambda\fc=\hat{\Lambda}\hat{\fc}\to\hat{\fc}\), 
see \apref{sec-Completion-coalgebras}.
Morphisms of completed conilpotent cocategories (\emph{cofunctors}) 
\(f:\fB\to\fC\) are morphisms of \(\Lambda\modulc_\LL\Quiv\) compatible 
with the comultiplication and the counit in the sense of 
\eqref{eq-dia-Delta-epsilon} (with $\tens$ replaced with $\hat{\tens}$) 
such that \(\eta_\fB\cdot f-\1 f\cdot\eta_\fC:\1\fB\to\fC\) is 
tensor convergent.
In their set-up De~Deken and Lowen introduce another notion -- 
$qA_\infty$-functors \cite{MR3856926} (which turns out equivalent to 
cofunctors, cf. \propref{prop-cofunctor-qA8}) in analogy with 
$qdg$-functors of \cite{1010.0982}.
The category of cofunctors between completed conilpotent cocategories 
is denoted $\cncCat$.
The category with augmentation preserving cofunctors between completed 
conilpotent cocategories (\(\eta_\fB\cdot f=\1 f\cdot\eta_\fC\), see 
\eqref{dia-augmentation-preserving}) is denoted $\acncCat$.
\end{definition}

\begin{remark}
The composition of cofunctors \(f:\fa\to\fb\) and \(g:\fb\to\fc\) is a 
cofunctor itself.
Indeed,
\begin{multline*}
\eta_\fa\cdot f\cdot g -\1(f\cdot g)\cdot\eta_\fc 
=\eta_\fa\cdot f\cdot g -\1 f\cdot\eta_\fb\cdot g 
+\1 f\cdot\eta_\fb\cdot g -\1f\cdot\1g\cdot\eta_\fc
\\
=(\eta_\fa\cdot f-\1 f\cdot\eta_\fb)\cdot g +\1 f\cdot(\eta_\fb\cdot g 
-\1g\cdot\eta_\fc): \1\fa \to \fc.
\end{multline*}
Both summands are tensor convergent by \remref{rem-tensor-convergent} 
and induce the mapping \(\Ob f\cdot\Ob g\) on objects.
By \lemref{lem-phipsi-tensor_convergent-phi+psi} the map 
\(\eta_\fa\cdot f\cdot g-\1(f\cdot g)\cdot\eta_\fc\) is 
tensor convergent.
Identity morphism is a cofunctor.
Thus, there is a category of completed conilpotent cocategories 
$\cncCat$ with cofunctors as morphisms.
\end{remark}

\begin{lemma}\label{lem-phiboxtpsi-tensor-convergent}
Let \(\phi:\1S\to\fb\), \(\psi:\1Q\to\fd\) be morphisms in 
\(\Lambda\modulc_\LL\Quiv\) and let $\phi$ be tensor convergent.
Then \(\phi\wh\boxt\psi:\1(S\times Q)=\1S\wh\boxt\1Q\to\fb\wh\boxt\fd\) 
is tensor convergent.
\end{lemma}

\begin{proof}
Let $X\in S$, $Y\in Q$.
We have \(1_Y\in\cf^0(\1Q(Y,Y))\), hence, 
\(\psi(1_Y)\in\cf^0(\fd(\psi Y,\psi Y))\).
Therefore, 
\(\psi(1_Y)^{\hat\tens n}\in\cf^0[\fd(\psi Y,\psi Y)^{\hat\tens n}]\).
Thus, for $n$ large enough
\begin{multline*}
(\phi\wh\boxt\psi)(1_X\wh\boxt1_Y)]^{\hat\tens n} 
=\phi(1_X)^{\hat\tens n}\wh\boxt\psi(1_Y)^{\hat\tens n} \in
\cf^l[\fb(\phi X,\phi X)^{\hat\tens n}]\wh\boxt\cf^0
[\fd(\psi Y,\psi Y)^{\hat\tens n}]
\\
\subset\cf^l\{[\fb(\phi X,\phi X)\wh\boxt
\fd(\psi Y,\psi Y)]^{\hat\tens n}\},
\end{multline*}
that is, \(\phi\wh\boxt\psi\) is tensor convergent.
\end{proof}

By \propref{pro-nilpotent-full-monoidal-subcategory-augmented} the 
category $\acncCat$ is monoidal with respect to $\wh\boxt$.
Furthermore:

\begin{proposition}
The category $\cncCat$ is monoidal with respect to $\wh\boxt$.
\end{proposition}

\begin{proof}
First of all, the $\wh\boxt$\n-product of morphisms of complete quivers,
which preserve the grading and the filtration, preserves them as well.
Secondly, let \(f:\fA\to\fB\), \(g:\fC\to\fD\) be morphisms from 
$\cncCat$.
Then
\begin{multline*}
(\eta_\fA\wh\boxt\eta_\fC)\cdot(f\wh\boxt g) 
-(\1f\wh\boxt\1g)\cdot(\eta_\fB\wh\boxt\eta_\fD) 
=\eta_\fA f\wh\boxt\eta_\fC g -\1f\eta_\fB\wh\boxt\1g\eta_\fD
\\
=(\eta_\fA f-\1f\eta_\fB)\wh\boxt\eta_\fC g 
+\1f\eta_\fB\wh\boxt(\eta_\fC g -\1g\eta_\fD): 
\1\fA\wh\boxt\1\fC \to \fB\wh\boxt\fD.
\end{multline*}
By \lemref{lem-phiboxtpsi-tensor-convergent} both summands are tensor 
convergent.
Both include the mapping \(\Ob f\times\Ob g\) on objects.
By \lemref{lem-phipsi-tensor_convergent-phi+psi} their sum is tensor 
convergent, hence, \(f\wh\boxt g\) is a cofunctor.
\end{proof}

If cofunctor \(f:\fa\to\fb\in\ncCat\), then 
\(\hat f:\hat\fa\to\hat\fb\in\acncCat\).
Thus, completion induces a functor \(\ncCat\to\acncCat\).

\begin{example}
Consider \(\LL=\{0,\infty\}\) with the neutral element 0 and the rules 
$0<\infty$, \(\infty+\infty=\infty\).
Any abelian group $M$ is equipped with the $\LL$\n-filtration 
$\cf^0M=M$, $\cf^\infty M=0$.
Such a filtration is called \emph{discrete} by 
\cite[Examples 2.4, 2.17, Remark~2.8]{MR3856926}.
The canonical mapping of $\LL$\n-filtered abelian groups 
\(\yi:M\to\hat{M}\) is an isomorphism.
The same for graded abelian groups $M$.
We have \(\Lambda\modulc_\LL=\Lambda\modul\) for an arbitrary graded 
commutative ring $\Lambda$.
A morphism \(\phi:\1S\to\fa\in\Lambda\modul\Quiv\) is tensor convergent 
if for every $X\in S$ there exists $N\in\NN$ such that 
\([\phi(1_X)]^{\hat{\tens}N}=0\).
The category $\cncCat$ has the same objects as $\ncCat$, but larger 
sets of morphisms.
For \(f:\fa\to\fb\in\cncCat\) the map 
\(\eta_\fa\cdot f-\1 f\cdot\eta_\fb\) is tensor convergent, while for 
$f\in\ncCat$ it is 0.
\end{example}

Completion of $\Lambda$\n-modules commutes with direct sums in the 
following sense.
Let $M$, $N$ be filtered $\Lambda$\n-modules.
Their direct sum is determined by the diagram 
\(M\pile{\lTTo^{\pr_1}\\ \rTTo_{\inj_1}}M\oplus N
\pile{\rTTo^{\pr_2}\\ \lTTo_{\inj_2}}N\) 
with the standard relations between $\pr_i$ and $\inj_j$.
The same relations hold in the completion 
\(\wh M\pile{\lTTo^{\wh{\pr_1}}\\ \rTTo_{\wh{\inj_1}}}\wh{M\oplus N}
\pile{\rTTo^{\wh{\pr_2}}\\ \lTTo_{\wh{\inj_2}}}\wh N\).
Therefore, \(\wh{M\oplus N}\cong\wh M\oplus\wh N\).

Applying to the conilpotent cocategory $T\fa$ from 
\exaref{exa-T-conilpotent-cocategory} the completion construction of 
\secref{sec-Completion-coalgebras}, we get a functor
\(\wh{T\text-\,}:\widetilde{\Lambda\modul_\LL\Quiv}\to
\acncCat\to\Coalg_{\widetilde{\Lambda\modulc_\LL\Quiv}}\).
The decomposition \(T\fa=T^{<n}\fa\oplus T^{\ge n}\fa\), $n\ge1$, 
implies the decomposition 
\(\wh{T\fa}=\wh{T^{<n}\fa}\oplus\wh{T^{\ge n}\fa}\).

\begin{remark}
Recall that $T\fa$ is also a (free) category with the composition $\mu$,
and \(T^{\ge n}\fa\) is its ideal.
This can be expressed as the existence of top arrow $\mu'$ in 
the commutative diagram
\begin{diagram}
T^{\ge n}\fa\tens T\fa &\rTTo^{\mu'} &T^{\ge n}\fa
\\
\dMono<{i\tens1} &&\dMono>i
\\
T\fa\tens T\fa &\rTTo^\mu &T\fa
\end{diagram}
where $i$ is the split inclusion, and by another similar diagram.
Completing this square to the right square in
\begin{diagram}
\wh{T^{\ge n}\fa}\hat\tens\wh{T\fa} &\rTTo^{\wh{\yi\tens\yi}^{-1}} 
&\wh{T^{\ge n}\fa\tens T\fa} &\rTTo^{\wh{\mu'}} &\wh{T^{\ge n}\fa}
\\
\dMono<{\wh i\wh\tens1} &&\dMono<{\wh{i\tens1}} &&\dMono>{\wh i}
\\
\wh{T\fa}\hat\tens\wh{T\fa} &\rTTo^{\wh{\yi\tens\yi}^{-1}} 
&\wh{T\fa\tens T\fa} &\rTTo^{\wh\mu} &\wh{T\fa}
\end{diagram}
we get a commutative diagram.
Thus, \(\wh{T^{\ge n}\fa}\) is a two-sided ideal of 
\((\wh{T\fa},\mu_{\wh{T\fa}})\).
\end{remark}

\begin{remark}\label{rem-series-converges}
Consider $\Lambda\modul_\LL$\n-category \(\cd=\Lambda\modul_\LL\Quiv\) 
and its reflective subcategory of complete quivers 
\(\cc=\Lambda\modulc_\LL\Quiv\).
Let $\fA$ be a completed conilpotent cocategory and let $\fb$ be 
a filtered quiver.
The obvious embedding of uniform spaces 
\((\pr_k)_{k\ge0}:T\fb\rMono \prod_{k\ge0}\fb^{\tens k}\), where the 
product is taken in $\tilde\cd$, leads to embedding of completions 
\(\wh{T\fb}\subset\wh{\prod_{k\ge0}\fb^{\tens k}}
=\prod_{k\ge0}\wh{\fb^{\tens k}}\),
see \cite[Chap.~II, \S3, n.9, Cor.~1]{zbMATH03395017}.
For the last equation just notice that limits commute with limits.
Therefore, we have injections
\begin{equation}
(\text-\cdot\wh{\pr_k})_{k\ge0}: \cncCat(\fA,\wh{T\fb}) \rMono 
\tilde{\cc}(\fA,\wh{T\fb}) \rMono 
\tilde{\cc}(\fA,\prod_{k\ge0}\wh{\fb^{\tens k}}) 
=\prod_{k\ge0}\tilde{\cc}(\fA,\wh{\fb^{\tens k}}),
\label{eq-injections-cncCat-prod}
\end{equation}
where the first product is in $\tilde{\cc}$ and the second in $\Set$.
The completion \(\wh{T\fb}\) of 
\(\oplus_{k\ge0}\wh{\fb^{\tens k}}
\equiv\coprod_{k\ge0}\wh{\fb^{\tens k}}\) 
coincides with the closure of this subspace of the complete space 
\(\prod_{k\ge0}\wh{\fb^{\tens k}}\).
Hence, \(\wh{T\fb}\) consists of elements of certain degree
\(x=(x_0,x_1,x_2,\dots)\in\prod_{k\ge0}\wh{\fb^{\tens k}}\) such that 
for every \(l\in\LL\) there is $n\in\NN$ with the property that for all 
$k\ge n$ we have \(x_k\in\cf^l\wh{\fb^{\tens k}}\).
We may say also that the series \(\sum_{k=0}^\infty x_k\) converges.
This is equivalent to the previous condition since $x_k$ belong to 
different direct summands.
\end{remark}

\begin{lemma}\label{lem-abcd-fg-convergent-series}
Let $\fa$, $\fb$, $\fc$, $\fd$ be complete quivers.
Let \(f_k:\fa\to\fb\), $k\in\NN$, and \(g_m:\fc\to\fd\), $m\in\NN$, be 
morphisms of filtered quivers.
Assume that series \(f=\sum_{k\in\NN}f_k\) (resp. 
\(g=\sum_{m\in\NN}g_m\)) pointwise converges, that is, for each 
$x\in\fa^d$, $d\in\ZZ$, (resp. $y\in\fc^p$, $p\in\ZZ$) and for every 
\(l\in\LL\) we have \(f_k(x)\in\cf^l\fb\) (resp. \(g_m(y)\in\cf^l\fd\)) 
except for finite number of terms.
Then the tensor product \(f\wh\boxt g:\fa\wh\boxt\fc\to\fb\wh\boxt\fd\) 
is the sum of pointwise convergent series 
\(\sum_{k,m\in\NN}f_k\wh\boxt g_m\), that is, for any $x\in\fa^d$, 
$y\in\fc^p$, $d,p\in\ZZ$, and any \(l\in\LL\) we have 
\(f_k(x)\wh\boxt g_m(y)\in\cf^l(\fb\wh\boxt\fd)\) except for finite 
number of terms.
\end{lemma}

\begin{proof}
For any $l\in\LL$ consider any decomposition $l=l'+l''$, $l',l''\in\LL$.
Let $x\in\fa^d$, $y\in\fc^p$, $d,p\in\ZZ$.
There exists $K\in\NN$ such that \(f_k(x)\in\cf^{l'}(\fb)\) for all 
$k\ge K$.
Consider $c_k\in\LL$ for $k<K$, $k\in\NN$, such that 
\(f_k(x)\in\cf^{c_k}(\fb)\).
There is \(\lambda''\in\LL\) such that \(\lambda''\ge l''\) and 
\(c_k+\lambda''\ge l\) for all $k<K$.
There exists $M\in\NN$ such that \(g_m(y)\in\cf^{\lambda''}(\fd)\) 
for all $m\ge M$.
Consider $b_m\in\LL$ for $m<M$, $m\in\NN$, such that 
\(g_m(y)\in\cf^{b_m}(\fd)\).
There is \(\lambda'\in\LL\) such that \(\lambda'\ge l'\) and 
\(\lambda'+b_m\ge l\) for all $m<M$.
There exists $N\in\NN$ such that $N\ge K$ and 
\(f_k(x)\in\cf^{\lambda'}(\fb)\) for all $k\ge N$.
For all pairs \((k,m)\in\NN^2\) except such that $k<N$ and $m<M$ we 
deduce from the above that 
\(f_k(x)\wh\boxt g_m(y)\in\cf^l(\fb\wh\boxt\fd)\).
\end{proof}

\begin{theorem}\label{thm-Cofunctors-1st-projection}
(i) Let \(\fA=\hat{\fa}\) be a completed conilpotent cocategory and let 
$\fb$ be a filtered quiver.
Cofunctors to the completed tensor cocategory \(\wh{T\fb}\) are in a 
natural bijection with the subset of quiver morphisms
\[ \cncCat(\fA,\wh{T\fb}) \rto\Phi \{\phi:\fA\to\hat{\fb}\in\tilde{\cc}
\mid \eta\cdot\phi:\Lambda\Ob\fA\to\hat{\fb}
\text{ is tensor convergent}\}, \quad f \mapsto f\cdot\wh{\pr_1}.
\]

(ii) There is a natural bijection
\[ \acncCat(\fA,\wh{T\fb}) \cong \{\phi:\fA\to\hat{\fb}\in\tilde{\cc}
\mid \eta\cdot\phi=0\}, \quad f \mapsto f\cdot\wh{\pr_1}.
\]
\end{theorem}

\begin{proof}
(i) First we remark that the corestriction \(f\cdot\wh{\pr_k}\) of any 
morphism \(f:\fA\to\wh{T\fb}\in\cncCat\) is uniquely determined by the 
composition 
\(\check{f}=f\cdot\wh{\pr_1}:\fA\to\wh\fb\in\Lambda\modulc_\LL\Quiv\) 
as the following commutative diagram shows
\begin{diagram}[width=4em,h=2.3em,LaTeXeqno]
&&\wh{T\fb} &\rTTo^{\wh{\pr_k}} &\wh{\fb^{\tens k}}
\\
&\ruTTo^f &\dTTo<{\Delta^{(k)}} &= &\dTTo>{\wh{\yi^{\tens k}}}
\\
\fA &= &\wh{T\fb}^{\hat\tens k} &\rTTo^{\wh{\pr_1}^{\hat\tens k}} 
&\wh{\fb}^{\hat\tens k}
\\
&\rdTTo<{\Delta^{(k)}} &\uTTo<{f^{\hat\tens k}} 
&\ruTTo^=_{\check{f}^{\hat\tens k}} &
\\
&&\fA^{\hat\tens k} &&
\label{dia-f:a-whTb}
\end{diagram}
In fact, the obvious equation
\[ \pr_k =\bigl(T\fb \rTTo^{\Delta^{(k)}} T\fb^{\tens k} 
\rTTo^{\pr_1^{\tens k}} \fb^{\tens k}\bigr)
\]
implies commutativity of the exterior of the diagram
\begin{diagram}[h=2.25em]
T\fb &\rTTo^\yi &\wh{T\fb} &&\rTTo^{\wh{\pr_k}} &&\wh{\fb^{\tens k}}
\\
\dTTo<{\Delta^{(k)}} &= &\dTTo<{\wh{\Delta^{(k)}}} 
&\rdTTo<=>{\Delta^{(k)}} &&&\dTTo>{\wh{\yi^{\tens k}}}
\\
T\fb^{\tens k} &\rTTo^\yi &\wh{T\fb^{\tens k}} 
&\rTTo^{\wh{\yi^{\tens k}}} &\wh{\wh{T\fb}^{\tens k}} 
&\rTTo^{\wh{\wh{\pr_1}^{\tens k}}} &\wh{\wh{\fb}^{\tens k}}
\end{diagram}
Therefore the trapezium commute which is the upper right square 
in diagram~\eqref{dia-f:a-whTb}.

So given \(\phi:\fA\to\hat{\fb}\in\tilde{\cc}\) such that 
\(\eta\cdot\phi:\Lambda\Ob\fA\to\hat{\fb}\) is tensor convergent, let us
prove that \(f:\fA\to\wh{T\fb}\) with \(\Ob f=\Ob\phi\) and for any 
element \(x\in\fA^d\), \(d\in\ZZ\), the value of $f(x)$ given by 
the (convergent) series
\begin{multline}
\sum_{k\ge0} 
\wh{\yi^{\tens k}}^{-1}\phi^{\hat{\tens}k}(\Delta_\fA^{(k)}x) 
=\sum_{k\ge0} \hat\yi^{-1}\phi(x_{(1)})\htdht\hat\yi^{-1}\phi(x_{(k)})
\\
=(x)\eps\cdot(\1\phi)\cdot\inj_0 
+\sum_{k\ge1} \hat\yi^{-1}\phi(x_{(1)})\htdht\hat\yi^{-1}\phi(x_{(k)}),
\label{eq-phix1-phixk}
\end{multline}
is a cofunctor.
Here \(x_{(1)}\htdht x_{(k)}\equiv\Delta^{(k)}(x)\).
Convergence means that for every $l\in\LL$ the $k$\n-th term except for 
a finite number of terms belongs to \(\cf^l\wh{\fb^{\tens k}}\) and 
will be proven now.

Assume that \(\fA=\hat{\fa}\) where cocategory $\fa$ is conilpotent and 
\(\phi:\fA\to\hat{\fb}\in\tilde{\cc}\) is such that 
\(\eta_\fA\cdot\phi\) tensor converges.
Replacing $\fa$ with the conilpotent cocategory \(\yi(\fa)\) we may 
assume that \(\yi:\fa\hookrightarrow\hat{\fa}\) is an embedding.
We have to prove that \eqref{eq-phix1-phixk} converges for all 
\(x\in\fA^d\).
It suffices to assume that \(x\in\bar{\fA}(X,Y)^d\).
Use the notation 
\(x_{(\bar{1})}\htdht x_{(\bar{k})}\equiv\bar{\Delta}_\fA^{(k)}(x)\).
The counital comultiplication $\Delta$ is recovered from the reduced 
comultiplication $\bar\Delta$ via the formula
\begin{equation}
\Delta_\fA^{(k)}(x) =\sum_{j:S\hookrightarrow\mb k}^{S\ne\emptyset} 
\hat{\tens}^{i\in\mb k} (x_{(\overline{j^{-1}i})})^{\chi(i\in jS)},
\label{eq-Delta-bar-Delta}
\end{equation}
where \(y^0=\eta_\fA(1)\), \(y^1=y\), and the summation extends over all
non-empty subsets $S$ of \(\mb k=\{1,2,\dots,k\}\).
By convention, \(\bar{\Delta}_\fA^{(1)}x=x\).

Assume for a moment that \(x\in\bar{\fa}^d\), hence, for some $n>0$ 
\(\bar{\Delta}_\fa^{(n)}(x)=0\).
So the summation in \eqref{eq-Delta-bar-Delta} goes over $S$ with 
\(|S|<n\).
The list of tuples of objects \(X=Z_0,Z_1,\dots,Z_{m-1},Z_m=Y\), which 
occur in
\[ \bar{\Delta}_\fa^{(m)}(x) \in 
\oplus_{Z_1,\dots,Z_{m-1}} \fa(X,Z_1)\tens\fa(Z_1,Z_2)\tdt\fa(Z_{m-1},Y)
\]
for $1\le m<n$, is finite.
Form a single list \(Z_1,\dots,Z_q\) of \(Z_i\) occurring in all such 
decompositions and denote 
\(Y_q=\hat\fb(\phi Z_q,\phi Z_q)\in\Lambda\modulc_\LL\).
Let $P\in\NN$ such that for every $p\in\NN$, $p\ge P$, we have 
\(\phi\eta(1_{Z_q})^{\tens p}\in\cf^0(Y_q^{\tens p})\).
Let \(b_p\in\LL\) be such that 
\(\phi\eta(1_{Z_q})^{\tens p}\in\cf^{b_p}(Y_q^{\tens p})\) for any 
\(1\le q\le Q\) and any $p<P$.
Let \(b\le b_p\) for all $p<P$.
Let $c\in\LL$ be such that all factors \(x_{(\bar i)}\) of all summands 
of all \(\Delta^{(m)}(x)\), $0\le m<n$, belonging to 
\(\fa(Z_{i-1},Z_i)\), were, in fact, in \(\cf^c\fa(Z_{i-1},Z_i)\).
There is \(\lambda\in\LL\) such that \((n-1)c+nb+\lambda\ge l\).
There is $N\in\NN$ such that $k\ge N$ implies 
\(\phi\eta(1_{Z_q})^{\tens k}\in\cf^\lambda(Y_q^{\tens k})\).
When \(k>nN\) at least one factor of this type occurs in 
\(\Delta^{(k)}(x)\).
Hence, for \(k>nN\) we have 
\(\phi^{\tens k}\Delta^{(k)}(x)\in\cf^\lambda(\hat\fb(\phi X,\phi Y))\).

Consider now an arbitrary \(x\in\bar{\fA}(X,Y)^d\).
Given $l\in\LL$ there is an element \(x'\in\bar{\fa}(X,Y)^d\) such that 
\(x-x'\in\cf^l\bar{\fA}(X,Y)^d\).
Then for all $k\ge1$ we have 
\(\phi^{\hat{\tens}k}(\Delta_\fA^{(k)}x)
-\yi[\phi^{\tens k}(\Delta_\fa^{(k)}x')]
\in\cf^l(\hat{\fb}^{\hat{\tens}k})\).
Since 
\(\phi^{\tens k}(\Delta_\fa^{(k)}x')\in\cf^l(\hat{\fb}^{\tens k})\) 
for large $k$ we deduce that 
\(\phi^{\hat{\tens}k}(\Delta_\fA^{(k)}x)
\in\cf^l(\hat{\fb}^{\hat{\tens}k})\) 
which proves the convergence of \eqref{eq-phix1-phixk} and gives a 
well-defined map of filtered quivers \(f:\fA\to\wh{T\fb}\) 
with \(\Ob f=\Ob\phi\).

Let us prove that $f$ is a morphism of cocategories.
Due to coassociativity of $\Delta_\fA$ we may write a convergent series
\begin{equation}
(x)f\cdot\Delta =\sum_{k\in\NN}\sum_{m+n=k}^{m,n\in\NN} 
\phi^{\hat\tens m}(\Delta^{(m)}x_{(1)})
\hat\tens\phi^{\hat\tens n}(\Delta^{(n)}x_{(2)}).
\label{eq-xf-Delta}
\end{equation}
Variant of \lemref{lem-abcd-fg-convergent-series} for $\hat\tens$ 
gives another convergent series
\[ (x)\Delta\cdot(f\hat\tens f) =\sum_{m,n\in\NN} 
\phi^{\hat\tens m}(\Delta^{(m)}x_{(1)})
\hat\tens\phi^{\hat\tens n}(\Delta^{(n)}x_{(2)})
\]
with the same terms as in \eqref{eq-xf-Delta} but with different 
summation order.
Since the sum of a series convergent in our sense does not depend on the
order of summation, we conclude that 
\(f\cdot\Delta=\Delta\cdot(f\hat\tens f)\).

The morphism $f$ preserves the counit due to \eqref{eq-phix1-phixk}:
\[ f\cdot\eps =f\cdot\wh{\pr_0} =\eps\cdot(\1\phi) =\eps\cdot(\1f).
\]

The map $f$ is a cofunctor since
\[ (1_X)\eta_\fA\cdot f -(1_X)\1f\cdot\inj_0\cdot\yi 
=\sum_{k\ge1} [\phi\eta(1_X)]^{\hat{\tens}k}.
\]
By definition, for any $l\in\LL_+$ there exists $N\in\NN$ such that for 
every $n\ge N$ we have 
\([\phi\eta(1_X)]^{\hat{\tens}n}\in\cf^l[\wh{\fb^{\tens n}}]\).
Clearly, all terms of
\[ [(1_X)\eta_\fA\cdot f -(1_X)\1f\cdot\inj_0\cdot\yi]^{\hat{\tens}n} 
=\Bigl[\sum_{k\ge1} [\phi\eta(1_X)]^{\hat{\tens}k}\Bigr]^{\hat{\tens}n}
\]
are in \(\cf^l[\wh{T\fb}]\).
Summing up, a map
\[ \Psi: \{\phi:\fA\to\hat{\fb}\in\tilde{\cc}\mid 
\eta\cdot\phi:\Lambda\Ob\fA\to\hat{\fb}\text{ is tensor convergent}\} 
\to \cncCat(\fA,\wh{T\fb})
\]
is constructed.

Let us prove that for any cofunctor \(f:\fA\to\wh{T\fb}\) the map 
\(\eta\cdot\check{f}=\eta\cdot f\cdot\wh{\pr_1}:\1\fA\to\wh\fb\) 
is tensor convergent.
We know that 
\(\eta_\fA\cdot f-\1f\cdot\inj_0\cdot\yi:\1\fA\to\wh{T\fb}\) is tensor 
convergent.
Therefore, 
\((\eta_\fA\cdot f-\1f\cdot\inj_0\cdot\yi)\cdot
\wh{\pr_1}=\eta\cdot\check{f}\)
is tensor convergent by \remref{rem-tensor-convergent}.
Thus, the claimed map \(\Phi:f\mapsto\check{f}\) is constructed.

Clearly, \(\Phi\Psi(\phi)=\phi\).
In particular, $\Phi$ is surjective.
As the reasoning at the beginning of the proof shows, 
injection~\eqref{eq-injections-cncCat-prod} factorizes through $\Phi$, 
namely, \((\text-\cdot\wh{\pr_k})_{k\ge0}=\Phi\cdot\Xi\), where
\[ \Xi: \tilde{\cc}(\fA,\wh{\fb}) \rMono 
\prod_{k\ge0}\tilde{\cc}(\fA,\wh{\fb^{\tens k}}), \qquad 
\phi \mapsto (\Delta_\fA^{(k)}\cdot\phi^{\hat\tens k})_{k\ge0}.
\]
Therefore, $\Phi$ is an injection as well.
We conclude that $\Phi$ is bijective and \(\Psi=\Phi^{-1}\).

(ii) follows from (i).
\end{proof}

\begin{corollary}\label{cor-Cofunctors-aTb}
Let $\fa$ be a conilpotent cocategory and let $\fb$ be a filtered quiver.

(i) Cofunctors to the completed tensor cocategory \(\wh{T\fb}\) are in 
a natural bijection with the subset of quiver morphisms
\[ \cncCat(\hat\fa,\wh{T\fb}) \cong \{\phi:\fa\to\hat{\fb}\in\tilde{\cc}
\mid \eta\cdot\phi:\Lambda\Ob\fa\to\hat{\fb}
\text{ is tensor convergent}\}, \quad f \mapsto f\cdot\wh{\pr_1}.
\]

(ii) There is a natural bijection
\[ \acncCat(\hat\fa,\wh{T\fb}) \cong\{\phi:\fa\to\hat{\fb}\in\tilde{\cc}
\mid \eta\cdot\phi=0\}, \quad f \mapsto f\cdot\wh{\pr_1}.
\]
\end{corollary}

\begin{proof}
Follows from \propref{pro-uniformly-continuous}(i) and 
\thmref{thm-Cofunctors-1st-projection} by universality property of 
the completion.
\end{proof}

\begin{definition}\label{def-cofunctor-aB}
Let \(\fa\in\ncCat\), \(\fB\in\cncCat\).
A \emph{cofunctor} \(f:\fa\to\fB\) is a morphism from 
\(\Lambda\modul_\LL\Quiv\) compatible with the comultiplication and 
the counit in the sense that
\begin{equation}
\begin{diagram}[inline]
\fa &&\rTTo^f &&\fB
\\
\dTTo<\Delta &&= &&\dTTo>\Delta
\\
\fa\tens\fa &\rTTo^{f\tens f} &\fB\tens\fB &\rTTo^\yi &\wh{\fB\tens\fB}
\end{diagram}
\qquad,\qquad
\begin{diagram}[inline]
\fa &\rTTo^f &\fB
\\
\dTTo<\eps &= &\dTTo>\eps
\\
\Lambda\fa &\rTTo^{\Lambda f} &\Lambda\fB
\end{diagram}
\label{eq-dia-aB-Delta-epsilon}
\end{equation}
and such that \(\eta_\fa\cdot f-\1f\cdot\eta_\fB\) is tensor convergent.
\end{definition}

An explanation of the above is given by

\begin{proposition}\label{pro-restriction-universality-ff'}
Let \(\fa\in\ncCat\), \(\fB\in\cncCat\).
The restriction and universality of the completion give mutually 
inverse bijections
\[ \cncCat(\hat{\fa},\fB) \longleftrightarrow 
\{\text{ cofunctors } \fa\to\fB\}.
\]
\end{proposition}

\begin{proof}
Diagrams~\eqref{eq-dia-aB-Delta-epsilon} for \(f=\yi:\fa\to\hat{\fa}\) 
take the form of left rectangles below
\begin{equation}
\begin{diagram}[inline,h=2.1em]
\fa &&\rTTo^\yi &&\hat{\fa} &\rTTo^g &\fB
\\
\dTTo<\Delta &&= &\ldTTo^{\hat{\Delta}}_= &\dTTo>{\Delta_{\hat{\fa}}} 
&&\dTTo>{\Delta_\fB}
\\
\fa\tens\fa &\rTTo^\yi &\wh{\fa\tens\fa} &\rTTo^{\wh{\yi\tens\yi}} 
&\wh{\hat\fa\tens\hat\fa} &\rTTo^{\wh{g\tens g}} &\wh{\fB\tens\fB}
\end{diagram}
\quad,\quad
\begin{diagram}[inline,h=2.1em]
\fa &\rTTo^\yi &\hat{\fa} &\rTTo^g &\fB
\\
\dTTo<\eps &= &\dTTo>{\hat\eps} &&\dTTo>\eps
\\
\Lambda\fa &\rEq^{\Lambda\yi}_\yi &\Lambda\fa &\rTTo^{\Lambda g} 
&\Lambda\fB
\end{diagram}
\quad.
\label{eq-dia-aahatB-Delta-epsilon}
\end{equation}
and they obviously commute, proving that \(\yi:\fa\to\hat{\fa}\) is 
a cofunctor.
The restriction map is 
\(\cncCat(\hat{\fa},\fB)\to\{\text{ cofunctors } \fa\to\fB\}\), 
\(g\mapsto\yi\cdot g\).

The inverse map is constructed as follows.
Let \(f:\fa\to\fB\in\Lambda\modul_\LL\Quiv\) satisfy 
\eqref{eq-dia-aB-Delta-epsilon}.
Then there is a unique 
\(g=\tilde f:\hat{\fa}\to\fB\in\Lambda\modul_\LL\Quiv\) such that 
\(f=\bigl(\fa\rto\yi \hat{\fa}\rto g \fB\bigr)\).
If $f$ is a cofunctor, the exterior rectangles of 
\eqref{eq-dia-aahatB-Delta-epsilon} commute.
By universality property of \(\yi:\fa\to\hat{\fa}\), the right squares 
of \eqref{eq-dia-aahatB-Delta-epsilon} commute as well.
\end{proof}

Introduce the notation \(\hat T\fa=\wh{T\fa}\).
We have shown in the above proof that any cofunctor 
\(f:T\fa\to\hat T\fb\) factorizes as 
\(f=\bigl(T\fa\rto\yi \hat T\fa\rto{\tilde f} \hat T\fb\bigr)\) for 
a unique \(\tilde f\in\cncCat(\hat T\fa,\hat T\fb)\).
The components of $f$ and $\tilde f$,
\begin{align*}
f_k &=\bigl(T^k\fa \rTTo^{\inj_k} T\fa \rTTo^f \hat T\fb 
\rTTo^{\wh{\pr_1}} \hat\fb\bigr),
\\
\tilde f_k &=\bigl(\wh{T^k\fa} \rTTo^{\wh{\inj_k}} \hat T\fa
\rTTo^{\tilde f} \hat T\fb \rTTo^{\wh{\pr_1}} \hat\fb\bigr),
\end{align*}
are related by \(f_k=\yi\cdot\tilde f_k\) as well.

\begin{remark}\label{rem-sumprin-converges-to-Id}
It follows from \remref{rem-series-converges} that the series
\[ \sum_{k=0}^\infty \wh{\pr_k}\cdot\wh{\inj_k} 
=\sum_{k=0}^\infty \bigl(\hat T\fb \rTTo^{\wh{\pr_k}} 
\wh{T^k\fb} \rTTo^{\wh{\inj_k}} \hat T\fb\bigr)
\]
converges to \(\Id_{\hat T\fb}\).
\end{remark}

We use this remark in order to write down components of the composition 
\(h=\bigl(T\fa\rto f \hat T\fb\rto{\tilde g} \hat T\fc\bigr)\).
We have by \eqref{dia-f:a-whTb}
\begin{align}
h_l &=\sum_{k=0}^\infty \bigl(T^l\fa \rTTo^{\inj_l} T\fa \rTTo^f 
\hat T\fb \rTTo^{\wh{\pr_k}} \wh{T^k\fb} \rTTo^{\wh{\inj_k}} \hat T\fb 
\rTTo^{\tilde g} \hat T\fc \rTTo^{\wh{\pr_1}} \hat\fc\bigr) \notag
\\
&=\sum_{k=0}^\infty \bigl(T^l\fa \rTTo^{\inj_l} T\fa\rTTo^{\Delta^{(k)}}
(T\fa)^{\tens k} \rTTo^{\check{f}^{\tens k}} \hat\fb^{\tens k} \rTTo^\yi
\hat\fb^{\wh\tens k} \rTTo^{\wh{\yi^{\tens k}}^{-1}} \fb^{\wh\tens k} 
\rTTo^{\tilde g_k} \hat\fc\bigr) \notag
\\
&=\sum_{i_1+\dots+i_k=l}^{k\ge0} \bigl(T^l\fa\rTTo^{f_{i_1}\tdt f_{i_k}}
\hat\fb^{\tens k} \rTTo^\yi \hat\fb^{\wh\tens k} 
\rTTo^{\wh{\yi^{\tens k}}^{-1}} \fb^{\wh\tens k} \rTTo^{\tilde g_k} 
\hat\fc\bigr).
\label{eq-hl-fi1fik-gk}
\end{align}

\begin{proposition}\label{prop-cofunctor-qA8}
Let $\fa$, $\fb$ be filtered quivers and let 
\(f:\hat T\fa\to\hat T\fb\in\Lambda\modulc_\LL\Quiv\) be compatible 
with the comultiplication and the counit.
Then $f$ is a cofunctor iff 
\(f_0=\wh{\inj_0}\cdot f\cdot\wh{\pr_1}:\1\fa\to\hat\fb\) 
is tensor convergent.
\end{proposition}

\begin{proof}
Assuming that \(f_0:\1\fa\to\hat\fb\) is tensor convergent, we find due 
to diagram~\eqref{dia-f:a-whTb} that 
\(\eta_{\hat{T}\fa}\cdot f-\1 f\cdot\eta_{\hat{T}\fb}
=\sum_{n>0}f_0^{\hat\tens n}:\1\fa\to\hat{T}\fb\)
and the series in the right hand side is convergent.
Furthermore, the series in the right hand side is tensor convergent.
Therefore, if $f_0$ is tensor convergent, then $f$ is cofunctor.

Assuming that $f$ is a cofunctor, we see that by definition 
\(y=\eta_{\hat{T}\fa}\cdot f-\1 f\cdot\eta_{\hat{T}\fb}:
\1\fa\to\wh{T^{\ge1}\fb}\) 
is tensor convergent.
From \eqref{eq-phix1-phixk} we obtain that 
\(y=\sum_{n>0}f_0^{\hat\tens n}\).
The series 
\(f_0'=\sum_{m>0}(-1)^{m-1}y^{\hat\tens m}:\1\fa\to\wh{T^{\ge1}\fb}\) 
converges.
Let us compute the sum:
\[ f_0' =\sum_{m>0}(-1)^{m-1}
\Bigl(\sum_{n>0}f_0^{\hat\tens n}\Bigr)^{\hat\tens m} 
=\sum_{k>0}f_0^{\hat\tens k} \sum_{i_1+\dots+i_m=k}^{m,i_j>0}(-1)^{m-1}.
\]
Notice that the coefficient near $t^k$ in expansion
\((\frac t{1-t})^m=t^m\sum_{a=0}^\infty(-1)^at^a\binom{-m}a\) equals 
\(\binom{k-1}{k-m}\) for $k\ge m$ and vanishes if $k<m$.
Therefore,
\[ f_0' =\sum_{k>0}f_0^{\hat\tens k} 
\sum_{m=1}^k(-1)^{m-1}\binom{k-1}{k-m} 
=f_0+\sum_{k>1}f_0^{\hat\tens k} (1-1)^{k-1} =f_0.
\]
We conclude that \(f_0=f_0'\) is tensor convergent.
By the way, one can show that the both compositions of the maps 
\(f_0\mapsto y\), \(y\mapsto f_0\) are identities.
\end{proof}

This proposition shows that in the set-up of De~Deken and Lowen 
$qA_\infty$-functors \cite{MR3856926} are the same as cofunctors.

\begin{corollary}
Let $\fa$, $\fb$ be filtered quivers and let
\(f:T\fa\to\hat T\fb\in\Lambda\modul_\LL\Quiv\) be compatible with the 
comultiplication and the counit in the sense of 
diagrams~\eqref{eq-dia-aB-Delta-epsilon}.
Then $f$ is a cofunctor iff 
\(f_0=\inj_0\cdot f\cdot\wh{\pr_1}:\1\fa\to\hat\fb\) is tensor 
convergent.
\end{corollary}

\subsection{Coderivations}\label{sec-Coderivations}
\begin{definition}
	Let \(f,g:\fA\to\fB\in\cncCat\).
An \((f,g)\)\n-coderivation \(r:f\to g:\fA\to\fB\) of degree $d$ and of 
level $\lambda$ is a collection of elements 
\(r\in\cf^\lambda\und{\Lambda\modul_\LL}(\fA(X,Y),\fB(fX,gY))^d\), which
satisfies the equation 
\(r\cdot\Delta=\Delta\cdot(f\hat\tens r+r\hat\tens g)\).
\end{definition}

Let \(\fA,\fB\in\cncCat\).
The coderivation quiver \(\Coder(\fA,\fB)\) has cofunctors 
\(f:\fA\to\fB\) as objects and the component 
\(\cf^\lambda\Coder(\fA,\fB)(f,g)^d\) of the filtered graded 
$\Lambda$\n-module \(\Coder(\fA,\fB)(f,g)\) consists of coderivations 
\(r:f\to g:\fA\to\fB\) of degree $d$ and of level $\lambda$.

\begin{proposition}\label{pro-coderivations-bijection}
Let $\fb$ be a filtered quiver and let $f,g:\fA\to\wh{T\fb}\in\cncCat$.
\((f,g)\)\n-coderivations \(r:f\to g:\fA\to\wh{T\fb}\) of degree $d$ and
of level $\lambda$ are in bijection with the collections of morphisms 
\(\check r=r\cdot\wh{\pr_1}\in
\cf^\lambda\und{\Lambda\modul_\LL}(\fA(X,Y),\hat\fb(fX,gY))^d\), 
$X,Y\in\Ob\fA$.
\end{proposition}

\begin{proof}
	The commutative diagram
	\begin{diagram}[width=6em,h=2.3em]
		&&\wh{T\fb} &\rTTo^{\wh{\pr_k}} &\wh{\fb^{\tens k}}
		\\
		&\ruTTo^r &\dTTo<{\Delta^{(k)}} &= &\dTTo>{\wh{\yi^{\tens k}}}
		\\
\fA &= &\wh{T\fb}^{\hat\tens k} &\rTTo^{\wh{\pr_1}^{\hat\tens k}} 
&\wh{\fb}^{\hat\tens k}
\\
& \rdTTo<{\Delta^{(k)}} 
&\uTTo~{\sum_{q+1+t=k}f^{\hat\tens q}\hat\tens r\hat\tens g^{\hat\tens t}}
&\ruTTo^=_{\sum_{q+1+t=k}\check{f}^{\hat\tens q}\hat\tens\check{r}\hat\tens\check{g}^{\hat\tens t}}
&
		\\
		&&\fA^{\hat\tens k} &&
	\end{diagram}
shows that the composition 
\(\fA\rto r\wh{T\fb}\hookrightarrow\prod_{k\ge0}\wh{\fb^{\tens k}}\) is 
given by the family
\[ \bigl(\Delta_\fA^{(k)} \cdot 
\sum_{q+1+t=k}\check{f}^{\hat\tens q}\hat\tens\check{r}
\hat\tens\check{g}^{\hat\tens t}\bigr)_{k=0}^\infty.
\]
Due to coassociativity of \(\Delta_\fA\) this equals 
\(\Delta_\fA^{(3)}\cdot(f\hat\tens\check{r}
\hat\tens g)\cdot\mu_{\wh{T\fb}}^{(3)}\), 
which clearly lies in $\wh{T\fb}$.
Thus,
\begin{equation}
r =\Delta_\fA^{(3)}\cdot(f\hat\tens\check{r}\hat\tens g)
\cdot\mu_{\wh{T\fb}}^{(3)}
\label{eq-r-Delta3-frg}
\end{equation}
is unambiguously determined by the collection 
\(\check{r}=r\cdot\wh{\pr_1}:\fA(X,Y)\to\hat{\fb}(fX,gY)\).

On the other hand, the right hand side of \eqref{eq-r-Delta3-frg} is an 
\((f,g)\)\n-coderivation as the following computation shows
\begin{align*}
x(r\cdot\Delta) &=(x_{(1)}\hat\tens x_{(2)}\hat\tens x_{(3)})
[(f\cdot\Delta)\hat\tens\check{r}\hat\tens g] 
+(x_{(1)}\hat\tens x_{(2)}\hat\tens x_{(3)})
[f\hat\tens\check{r}\hat\tens(g\cdot\Delta)]
\\
&=\bigl(x_{(1)}\wh{\bigotimes}x_{(2)}\hat\tens x_{(3)}\hat\tens x_{(4)}\bigr)
\bigl(f\wh{\bigotimes}f\hat\tens\check{r}\hat\tens g\bigr) 
+\bigl(x_{(1)}\hat\tens x_{(2)}\hat\tens x_{(3)}\wh{\bigotimes}x_{(4)}\bigr)
\bigl(f\hat\tens\check{r}\hat\tens g\wh{\bigotimes}g\bigr)
\\
&=\bigl(x_{(1)}\wh{\bigotimes}x_{(2)}\bigr) \bigl(f\wh{\bigotimes}r\bigr)
+\bigl(x_{(1)}\wh{\bigotimes}x_{(2)}\bigr) \bigl(r\wh{\bigotimes}g\bigr)
=(x\Delta)(f\hat\tens r+r\hat\tens g).
\end{align*}
It remains to note that $f$ and $g$ preserve the filtration and the 
grading.
\end{proof}

\begin{corollary}\label{cor-a-hatb-coderivations}
Let \(\fa\in\ncCat\), let $\fb$ be a filtered quiver and let 
\(f,g:\hat\fa\to\wh{T\fb}\in\cncCat\).
\((f,g)\)\n-coderivations \(r:f\to g:\hat\fa\to\wh{T\fb}\) of degree $d$
and of level $\lambda$ are in bijection with the collections of morphisms
\(\check r=r\cdot\wh{\pr_1}\in
\cf^\lambda\und{\Lambda\modul_\LL}(\fa(X,Y),\hat\fb(fX,gY))^d\), 
$X,Y\in\Ob\fa$.
\end{corollary}

\begin{proof}
Follows from Propositions \ref{pro-uniformly-continuous}(i) and 
\ref{pro-coderivations-bijection} by universality property of 
the completion.
\end{proof}

\begin{definition}\label{def-coderivation-aB}
Let \(\fa\in\ncCat\), \(\fB\in\cncCat\), and let \(f,g:\fa\to\fB\) be 
cofunctors in the sense of \defref{def-cofunctor-aB}.
An \((f,g)\)\n-coderivation \(r:f\to g:\fa\to\fB\) of degree $d$ and of 
level $\lambda$ is a collection of elements 
\(r\in\cf^\lambda\und{\Lambda\modul_\LL}(\fa(X,Y),\fB(fX,gY))^d\), 
which satisfies the equation
\begin{diagram}[LaTeXeqno]
\fa &&\rTTo^r &&\fB
\\
\dTTo<\Delta &&= &&\dTTo>\Delta
\\
\fa\tens\fa &\rTTo^{f\tens r+r\tens g} &\fB\tens\fB &\rTTo^\yi 
&\wh{\fB\tens\fB}
\label{dia-arB}
\end{diagram}
The filtered $\Lambda$\n-module of \((f,g)\)\n-coderivations is denoted 
\(\Coder(\fa,\fB)(f,g)\).
\end{definition}

The reason for introducing this definition is given by the following

\begin{proposition}\label{pro-coderivations-aB}
Let \(\fa\in\ncCat\), \(\fB\in\cncCat\), and let \(f,g:\fa\to\fB\) be 
cofunctors.
They can be represented as \(f=\yi\cdot f'\), \(g=\yi\cdot g'\) by 
\propref{pro-restriction-universality-ff'}.
Then the map
\begin{equation}
\Coder(\hat{\fa},\fB)(f',g') \to \Coder(\fa,\fB)(f,g), 
\qquad r' \mapsto \yi\cdot r'=r,
\label{eq-r'-ir'-r}
\end{equation}
is a bijection.
\end{proposition}

\begin{proof}
Take \(r'\in\Coder(\hat{\fa},\fB)(f',g')\).
Then the rightmost quadrilateral (trapezium) in the following diagram 
commutes:
\begin{diagram}[LaTeXeqno]
\fa &\rTTo^\yi &\hat{\fa} &&\rTTo^{r'} &&\fB
\\
\dTTo<\Delta &= &\dTTo<{\hat\Delta} &\rdTTo<=>{\Delta_{\hat{\fa}}} 
&&&\dTTo>\Delta
\\
\fa\tens\fa &\rTTo^\yi &\wh{\fa\tens\fa} &\rTTo^{\wh{\yi\tens\yi}} 
&\wh{\hat\fa\tens\hat\fa} &\rTTo^{\wh{f'\tens r'}+\wh{r'\tens g'}} 
&\wh{\fB\tens\fB}
\\
&\rdTTo(4,2)[hug]_{f\tens r+r\tens g} &&&= &\ruTTo_\yi &
\\
&&&&\fB\tens\fB &&
\label{dia-air'B}
\end{diagram}
Therefore, the whole diagram commutes and map \eqref{eq-r'-ir'-r} is 
well-defined.

On the other hand, any map 
\(r\in\cf^\lambda\und{\Lambda\modul_\LL}(M,N)^d\) takes \(\cf^lM^k\) 
to \(\cf^{l+\lambda}N^{k+d}\).
We are interested in \(M=\fa(X,Y)\), \(N=\fB(fX,gY)\).
By \propref{pro-uniformly-continuous}(i) \(r:M^k\to N^{k+d}\) are 
uniformly continuous for all $k\in\ZZ$.
Therefore, these maps factorize as 
\(r=\bigl(M^k\rto\yi \wh{M^k} \rto{r'} N^{k+d}\bigl)\).
Clearly, \(r'\in\cf^\lambda\und{\Lambda\modul_\LL}(\hat M,N)^d\) and 
\(r=\yi\cdot r'\).
The exterior of \eqref{dia-air'B} commutes.
Thus, the biggest rectangle in \eqref{dia-air'B} commutes.
Hence, the right rectangle commutes.
Equivalently, the trapezium with vertices $\hat{\fa}$, \dots, 
$\wh{\fB\tens\fB}$ commutes, that is, $r'$ is an 
\((f',g')\)-coderivation.
\end{proof}

\begin{corollary}\label{cor-coderivations-aB}
Let \(\fa\in\ncCat\), \(\fB\in\cncCat\).
Then the filtered quivers \(\Coder(\hat{\fa},\fB)\) and 
\(\Coder(\fa,\fB)\) are isomorphic.
\end{corollary}

When we write \(r:f\to g:\ca\to\cb\) we mean 
\(r\in\cf^\lambda\Coder(Ts\ca,\hat Ts\cb)(f,g)^d\) for some $d\in\ZZ$ 
and \(\lambda\in\LL\).
Suppose that \(\tilde h:\hat Ts\cb\to \hat Ts\cc\) is a cofunctor.
Then for $r$ as above there is a coderivation 
\(r\tilde h\in
\cf^\lambda\Coder(Ts\ca,\hat Ts\cc)(f\tilde h,g\tilde h)^d\), 
whose components are found as
\begin{multline*}
(r\tilde h)_l =\sum_{i_1+\dots+i_k+t+j_1+\dots+j_m=l}^{k,m\ge0} 
\Bigl(T^ls\ca 
\rTTo^{f_{i_1}\tdt f_{i_k}\tens r_t\tens g_{j_1}\tdt g_{j_m}} 
(s\hat\cb)^{\tens(k+1+m)}
\\
\rTTo^\yi (s\hat\cb)^{\wh\tens(k+1+m)}\rTTo^{\wh{\yi^{\tens k+1+m}}^{-1}}
(s\cb)^{\wh\tens(k+1+m)} \rTTo^{\tilde h_{k+1+m}} \hat\fc\Bigr),
\end{multline*}
due to \propref{pro-coderivations-bijection}, 
\remref{rem-sumprin-converges-to-Id} similarly to 
\eqref{eq-hl-fi1fik-gk}.

Suppose now that besides 
\(\tilde r\in
\cf^\lambda\Coder(\hat Ts\ca,\hat Ts\cb)(\tilde f,\tilde g)^d\) 
we have a cofunctor \(e:Ts\cc\to\hat Ts\ca\).
Then we have also a coderivation 
\(e\tilde r:e\tilde f\to e\tilde g:\cc\to\cb\), 
\(e\tilde r\in
\cf^\lambda\Coder(Ts\cc,\hat Ts\cb)(e\tilde f,e\tilde g)^d\), 
whose components are given by
\[ (e\tilde r)_l =\sum_{i_1+\dots+i_k=l}^{k\ge0} \bigl(T^ls\cc 
\rTTo^{e_{i_1}\tdt e_{i_k}} (s\hat\ca)^{\tens k} \rTTo^\yi 
(s\hat\ca)^{\wh\tens k} \rTTo^{\wh{\yi^{\tens k}}^{-1}} 
(s\ca)^{\wh\tens k} \rTTo^{\tilde r_k} s\hat\cb\bigr)
\]
due to \thmref{thm-Cofunctors-1st-projection}, 
\remref{rem-sumprin-converges-to-Id} similarly to 
\eqref{eq-hl-fi1fik-gk}.

\subsection{Evaluation.}\label{sec-Evaluation-filtered}
Let $\fA$ be a completed conilpotent cocategory and let $\fb$ be a 
$\cv$\n-quiver.
Define the evaluation cofunctor 
\(\ev:\fA\boxt T\Coder(\fA,\wh{T\fb})\to\wh{T\fb}\) on objects as 
\(\ev(A\boxt f)=fA\), and on morphisms as follows.
Let \(f^0,f^1,\dots,f^n:\fA\to\wh{T\fb}\) be cofunctors, and let 
\(r^1,\dots,r^n\) be coderivations of certain degrees and of some level 
as in 
$f^0\rto{r^1} f^1\rto{r^2} \dots f^{n-1}\rto{r^n} f^n:\fA\to\wh{T\fb}$,
$n\ge0$.
Then \(c=r^1\tdt r^n\in T^n\Coder(\fA,\wh{T\fb})(f^0,f^n)\).
Define
\begin{equation*}
[a\boxt(r^1\tdt r^n)]\ev =(a\Delta^{(2n+1)})(f^0\hat\tens\check r^1
\hat\tens f^1\hat\tens\check r^2\htdht f^{n-1}\hat\tens\check r^n
\hat\tens f^n) \mu_{\wh{T\fb}}^{(2n+1)}.
\end{equation*}
The right hand side belongs to 
\((\wh{T\fb})^{\hat\tens(2n+1)}\mu_{\wh{T\fb}}^{(2n+1)}\) and is mapped 
by multiplication \(\mu_{\wh{T\fb}}^{(2n+1)}\) from 
\eqref{eq-Multiplication-algebra-hat-TX} into \(\wh{T\fb}\).
So defined $\ev$ is a cofunctor.
Indeed, \(\eta\cdot\ev-\1\ev\cdot\eta\) applied to \(1_A\boxt1_f\), 
\(A\in\Ob\fA\), \(f\in\Ob\Coder(\fA,\wh{T\fb})=\cncCat(\fA,\wh{T\fb})\),
gives a tensor convergent expression in \(\wh{T\fb}\)
\[ [\eta(1_A)\boxt\inj_0(1_f)]\ev -(1_{fA})\eta 
=f[\eta(1_A)] -\eta(1_{fA}),
\]
since $f$ is a cofunctor.
%{\sf The evaluation for completed conilpotent cocategories is the cofunctor}
%\[ \mathsf{ \Ev =\bigl[ \fA\wh\boxt\hat T\Coder(\fA,\wh{T\fb}) 
%\rTTo^{\wh{1\boxt\yi}^{-1}} (\fA\boxt T\Coder(\fA,\wh{T\fb}))\sphat 
%\rTTo^{\wh\ev} \wh{\wh{T\fb}} \rTTo^{\hat{\yi}^{-1}} \wh{T\fb} \bigr].}
%\]

The following statement generalizes Proposition~3.4 of \cite{Lyu-AinfCat}.

\begin{theorem}\label{thm-Ev-psi}
For \(\fa\in\ncCat\), \(\fb,\fc^1,\dots\fc^q\in\Lambda\modul_\LL\Quiv\) 
with notation \(\fc=T\fc^1\boxt\dots\boxt T\fc^q\) the map
\begin{align*}
\ncCat(\fc,T\Coder(\fa,\wh{T\fb})) 
&\longrightarrow \cncCat(\fa\boxt\fc,\wh{T\fb}),
\\
\psi &\longmapsto \bigl( \fa\boxt\fc \rTTo^{\fa\boxt\psi} 
\fa\boxt T\Coder(\fa,\wh{T\fb}) \rto{\ev} \wh{T\fb} \bigr)
\end{align*}
is a bijection.
\end{theorem}

\begin{proof}
An augmentation preserving cofunctor 
\(\psi:\fc\to T\Coder(\hat\fa,\wh{T\fb})\cong T\Coder(\fa,\wh{T\fb})\) 
(see \corref{cor-coderivations-aB}) is described by an arbitrary quiver 
map
\(\check{\psi}
=\psi\cdot\pr_1:\fc\to\Coder(\fa,\wh{T\fb})\in\Lambda\modul_\LL\Quiv\) 
such that \(\eta\cdot\check{\psi}=0\) by \propref{pro-tensor-ncCat}.
Let \(\phi:\fa\boxt\fc\to\wh{T\fb}\) be a cofunctor.
It equals the cofunctor 
\((\fa\boxt\psi)\cdot\ev:\fa\boxt\fc\to\wh{T\fb}\) if the equation
\[ \sum_{k\ge0}(a\boxt c\Delta^{(k)}\check{\psi}^{\tens k})\ev 
=(a\boxt c)\phi, \qquad a\in\fa^\bull, \ c\in\fc^\bull,
\]
holds (by \thmref{thm-Cofunctors-1st-projection}(i) and 
\propref{pro-restriction-universality-ff'}).
It suffices to consider two cases.
In the first one \(c=\eta(1_g)\) for some \(g\in\Ob\fc\).
Then the equation takes the form \((a)(g\psi)=(a\boxt c)\phi\) which 
defines the cofunctor \(g\psi\in\cncCat(\fa,\wh{T\fb})\) in the left 
hand side.

In the second case \(c\in\cf^l\bar{\fc}^d\) the equation takes the form
\begin{equation*}
(a)(c)\check{\psi} 
+\sum_{k\ge2}(a\boxt c\Delta^{(k)}\check{\psi}^{\tens k})\ev
=(a\boxt c)\phi, \qquad a\in\fa^\bull, \ c\in\bar{\fc}^\bull.
\end{equation*}
Since \(\eta\cdot\check{\psi}=0\) the comultiplication $\Delta$ can be 
replaced with $\bar{\Delta}$.
The structure of \(\fc=T\fc^1\bdb T\fc^q\) is such that the component 
\(\psi_{i_1,\dots,i_q}\) in the left hand side of 
\begin{equation}
(a)(c)\check{\psi} =(a\boxt c)\phi 
-\sum_{k\ge2}(a\boxt c\bar{\Delta}^{(k)}\check{\psi}^{\tens k})\ev, 
\qquad a\in\fa^\bull, \ c\in\bar{\fc}^\bull,
\label{eq-acpsi-acphi-filtered}
\end{equation}
is expressed via the components \(\psi_{j_1,\dots,j_q}\) with smaller 
indices \((j_1,\dots,j_q)\) in the product poset \(\NN^q\).
For \(c\in\cf^l\bar{\fc}(X,Y)^d\), \(X=(X_1,\dots,X_q)\), 
\(Y=(Y_1,\dots,Y_q)\), \(X_i,Y_i\in\Ob\fc^i\), find $n\ge0$ such that 
\(c\bar{\Delta}^{(n+1)}=0\).
Equation \eqref{eq-acpsi-acphi-filtered} determines a unique collection 
of maps 
\(c\check{\psi}\in\cf^l\und{\Lambda\modul_\LL}
\bigl(\fa(U,V),\wh{T\fb}((U,X)\phi,(V,Y)\phi)\bigr)^d\).
It remains to verify that it is a coderivation.
We have to prove that
\[ (a)(c\check{\psi})\Delta_\fb 
=(a)\Delta_\fa[(\_\boxt X)\phi\tens(\_)(c\check{\psi}) 
+(\_)(c\check{\psi})\tens(\_\boxt Y)\phi].
\]
The case $n=0$ being obvious, assume that $n\ge1$.
The sum in \eqref{eq-acpsi-acphi-filtered} goes from $k=2$ to $n$.
Correspondingly,
\[ (a)(c\check{\psi})\Delta 
=(a\Delta)[(\_\boxt c_{(1)})\phi\tens(\_\boxt c_{(2)})\phi]
-\sum_{k=2}^n [(a\Delta)\boxt(c_{\bar1}\check{\psi}\tdt 
c_{\bar k}\check{\psi})\Delta_{T\Coder}]\tau_{(23)}(\ev\tens\ev).
\]
Here according to Sweedler's notation \(c_{(1)}\tens c_{(2)}=c\Delta\).
Similarly, \(c_{\bar1}\tdt c_{\bar k}=c\bar{\Delta}^{(k)}\).
Recall the middle four interchange 
\([(a\tens b)\boxt(c\tens d)]\tau_{(23)}
=(-1)^{bc}(a\boxt c)\tens(b\boxt d)\).
The above expression has to be equal to
\begin{multline*}
(a\Delta)\Bigl\{(\_\boxt1_{X})\phi\tens\bigl[(\_\boxt c)\phi 
-\sum_{k=2}^n\bigl(\_\boxt(c_{\bar1}\check{\psi}\tdt 
c_{\bar k}\check{\psi})\bigr)\ev \bigr]\Bigr\}
\\
+(a\Delta)\Bigl\{\bigl[(\_\boxt c)\phi 
-\sum_{k=2}^n\bigl(\_\boxt(c_{\bar1}\check{\psi}\tdt 
c_{\bar k}\check{\psi})\bigr)\ev\bigr] \tens(\_\boxt1_{Y})\phi\Bigr\}.
\end{multline*}
Canceling the above terms we come to identity to be checked
\begin{equation}
(a\Delta)[(\_\boxt c_{\bar1})\phi\tens(\_\boxt c_{\bar2})\phi] 
=\sum_{k=2}^n \bigl[(a\Delta)\boxt(c_{\bar1}\check{\psi}\tdt c_{\bar k}
\check{\psi})\bar\Delta_{T\Coder}\bigr]\tau_{(23)}(\ev\tens\ev).
\label{eq-c1phi-c2phi-filtered}
\end{equation}
The right hand side equals
\begin{align*}
&\sum_{k=2}^n\sum_{i=1}^{k-1} \Bigl\{(a\Delta)\boxt\bigl[(c_{\bar1}
\check{\psi}\tdt c_{\bar i}\check{\psi})\bigotimes(c_{\overline{i+1}}
\check{\psi}\tdt c_{\bar k}\check{\psi}) \bigr]\Bigr\}\tau_{(23)}
(\ev\tens\ev) \notag
\\
&=\sum_{k=2}^n\sum_{i=1}^{k-1} (a\Delta) \Bigl\{\bigl[\_\boxt(c_{\bar1}
\check{\psi}\tdt c_{\bar i}\check{\psi})\bigl]\ev \tens \bigl[\_\boxt
(c_{\overline{i+1}}\check{\psi}\tdt c_{\bar k}\check{\psi})\bigr]\ev 
\Bigr\} \notag
\\
&=\sum_{i=1}^n\sum_{j=1}^n (a\Delta) \Bigl\{\bigl[\_\boxt(c_{\bar1}
\check{\psi}\tdt c_{\bar i}\check{\psi})\bigl]\ev \tens \bigl[\_\boxt
(c_{\overline{i+1}}\check{\psi}\tdt c_{\overline{i+j}}\check{\psi})
\bigr]\ev \Bigr\} \notag
\\
&=(a\Delta)[(\_)(c_{\bar1}F)\tens(\_)(c_{\bar2}F)],
\end{align*}
where
\[ (a)(cF) =\sum_{i=1}^n\bigl[a\boxt(c_{\bar1}\check{\psi}\tdt 
c_{\bar i}\check{\psi})\bigl]\ev =(a)(c\check{\psi}) 
+\sum_{i=2}^n\bigl[a\boxt(c_{\bar1}\check{\psi}\tdt 
c_{\bar i}\check{\psi})\bigl]\ev =(a\boxt c)\phi
\]
due to \eqref{eq-acpsi-acphi-filtered}.
Hence the right hand side of \eqref{eq-c1phi-c2phi-filtered} equals 
\((a\Delta)[(\_\boxt c_{\bar1})\phi\tens(\_\boxt c_{\bar2})\phi]\), 
which is the left hand side of \eqref{eq-c1phi-c2phi-filtered}.
\end{proof}

Let $\fa$ be a conilpotent cocategory and let $\fb$, $\fc$ be quivers.
Consider the cofunctor given by the upper right path in the diagram
\begin{diagram}
	\fa\boxt T\Coder(\fa,\wh{T\fb})\boxt T\Coder(\wh{T\fb},\wh{T\fc}) &
	\rTTo^{\ev\boxt1} & \wh{T\fb}\boxt T\Coder(\wh{T\fb},\wh{T\fc})
	\\
	\dTTo<{1\boxt M} & = & \dTTo>\ev
	\\
	\fa\boxt T\Coder(\fa,\wh{T\fc}) & \rTTo^\ev & \wh{T\fc}
\end{diagram}
By \thmref{thm-Ev-psi} there is a unique augmentation preserving 
cofunctor
\[ M:  T\Coder(\fa,\wh{T\fb})\boxt T\Coder(\wh{T\fb},\wh{T\fc})
\to  T\Coder(\fa,\wh{T\fc}).
\]

Denote by $\1$ the unit object $\boxt^0$ of the monoidal category of 
cocategories, that is, $\Ob\1=\{*\}$, $\1(*,*)=\Lambda$.
Denote by $\rightunit:\fa\boxt\1\to\fa$ and 
$\leftunit:\1\boxt\fa\to \fa$ the corresponding natural isomorphisms.
By \thmref{thm-Ev-psi} there exists a unique augmentation preserving 
cofunctor $\eta_{\wh{T\fb}}:\1\to T\Coder(\wh{T\fb},\wh{T\fb})$, 
such that
\[ \rightunit = \bigl(\wh{T\fb}\boxt\1 \rTTo^{1\boxt\eta_{\wh{T\fb}}} 
\wh{T\fb}\boxt T\Coder(\wh{T\fb},\wh{T\fb}) \rTTo^\ev \wh{T\fb}\bigr).
\]
Namely, the object $*\in\Ob\1$ goes to the identity homomorphism 
$\id_{\wh{T\fb}}:\wh{T\fb}\to \wh{T\fb}$.

The following statement follows from \thmref{thm-Ev-psi}.

\begin{proposition}%\label{pro-M-assoc-eta-unit}
The multiplication $M$ is associative and $\eta$ is its two-sided unit:
	\begin{diagram}[nobalance]
T\!\Coder(\fa,\wh{T\fb})\!\boxt\! T\!\Coder(\wh{T\fb},\wh{T\fc})
\!\boxt\! T\!\Coder(\wh{T\fc},\wh{T\fd}) &\rTTo^{M\boxt1} 
&T\!\Coder(\fa,\wh{T\fc})\!\boxt\! T\!\Coder(\wh{T\fc},\wh{T\fd}) 
\\
		\dTTo<{1\boxt M} && \dTTo>M 
		\\
T\!\Coder(\fa,\wh{T\fb})\!\boxt\! T\!\Coder(\wh{T\fb},\wh{T\fd}) 
&\rTTo^M &T\!\Coder(\fa,\wh{T\fd})
	\end{diagram}
\end{proposition}

The multiplication $M$ is computed explicitly in 
\cite[\S4]{Lyu-AinfCat}, see, in particular, Examples~4.2 there.

\section{Filtered \texorpdfstring{$A_\infty$}{A8}-categories}
\label{sec-Filtered-A8-categories}
For a filtered graded quiver $\ca$ denote by \(s\ca=\ca[1]\) the same 
quiver with the shifted grading, \(\ca[1]^n=\ca^{n+1}\).
The shift commutes with the completion.
By $s$ we denote also the ``identity'' map $s:\ca\to\ca[1]$, 
\(\ca^n\ni x\mapsto x\in\ca[1]^{n-1}\), of degree $-1$.

\begin{definition}
A filtered \ainf-category $\ca$ is an $\LL$\n-filtered $\ZZ$\n-graded 
quiver $\ca$, equipped with a coderivation 
$b:\Id\to\Id:\hat Ts\ca\to\hat Ts\ca$ of degree 1 and of level 0, such 
that\footnote{I am grateful to Kaoru Ono for explaining the reasons why 
the differential preserves the grading in Fukaya categories.}
the collection $b:\hat Ts\ca(X,Y)\to\hat Ts\ca(X,Y)$ satisfies $b^2=0$.
Another name -- curved \ainf-category.
De~Deken and Lowen use the name of filtered $cA_{\infty}$-category 
\cite{MR3856926}.
\end{definition}

The codifferential $b$ is determined in a unique fashion by the 
collection of morphisms 
\(\check b=b\cdot\wh{\pr_1}\in
\cf^0\und{\Lambda\modul_\LL}(\hat Ts\ca(X,Y),s\hat\ca(X,Y))^1\), 
\(X,Y\in\Ob\ca\), equivalently, by the collection of morphisms 
\(\check b=b\cdot\wh{\pr_1}\in
\cf^0\und{\Lambda\modul_\LL}(Ts\ca(X,Y),s\hat\ca(X,Y))^1\), 
\(X,Y\in\Ob\ca\), due to \corref{cor-a-hatb-coderivations}, 
equivalently, by the components 
\(b_n\in\cf^0\und{\Lambda\modul_\LL}(T^ns\ca(X,Y),s\hat\ca(X,Y))^1\), 
\(X,Y\in\Ob\ca\), $n\ge0$.
The codifferential $b$ is recovered from its components 
\(b_j:T^js\ca\to s\hat\ca\) due to Propositions 
\ref{pro-coderivations-bijection}, \ref{pro-coderivations-aB}:
\[ b =\sum_{i+j+k=n} (\yi1^{\wh\tens i})\hat\tens b_j\hat\tens
(\yi1^{\wh\tens k}): T^ns\ca\to (T^{\le n+1}s\ca)\sphat\,.
\]
The square $b^2$ is a (1,1)-coderivation of level 0 and of degree 2:
\[ b^2\Delta =b\Delta(1\hat\tens b+b\hat\tens1) 
=\Delta(1\hat\tens b+b\hat\tens1)^2 
=\Delta(1\hat\tens b^2+b^2\hat\tens1): Ts\ca \to \hat Ts\ca.
\]
Thus, the equation $b^2=0$ is equivalent to the system $(b^2)_n=0$, 
$n\ge0$.
The components of $b^2$ can be found via 
\remref{rem-sumprin-converges-to-Id} by insertion of $\id$ between $b$ 
and $b$.
Therefore, the equation $b^2=0$ can be written as
\[ \sum_{i+j+k=n} [(\yi1^{\wh\tens i})\hat\tens b_j\hat\tens
(\yi1^{\wh\tens k})] b_{i+1+k}=0: T^ns\ca\to s\hat\ca.
\]

Let $\ca$, $\cb$ be filtered \ainf-categories, let 
\(f^0,f^1,\dots,f^n:Ts\ca\to\hat Ts\cb\) be cofunctors (see 
\defref{def-cofunctor-aB}), and let \(r^1,\dots,r^n\) be coderivations 
of certain degrees and of some level as in 
\(f^0\rto{r^1} f^1\rto{r^2} \dots f^{n-1}\rto{r^n} f^n:
Ts\ca\to\hat Ts\cb\), 
$n\ge0$ (see \defref{def-coderivation-aB}).
Then \(r^1\tdt r^n\in T^n\Coder(Ts\ca,\hat Ts\cb)(f^0,f^n)\).
Let \(a\in(T^\bull s\ca)^\bull\).

\begin{proposition}\label{pro-B-F-KS-LH}
In the above assumptions there is a unique (1,1)-coderivation of degree 
1 and level 0 
$B:T\Coder(Ts\ca,\hat Ts\cb)\to T\Coder(Ts\ca,\hat Ts\cb)$, such that
	\begin{equation}
[a\boxt(r^1\tdt r^n)]\ev b =[a\boxt(r^1\tdt r^n)B]\ev 
+(-)^{r^1+\dots+r^n}[ab\boxt(r^1\tdt r^n)]\ev
	\label{eq-theta-b-B-theta}
	\end{equation}
for all \(a\in Ts\ca\), $n\ge0$, 
$r^1\tdt r^n\in T^n\Coder(Ts\ca,\hat Ts\cb)(f^0,f^n)$.
It satisfies $B^2=0$, thus, it gives an \ainf-structure to 
$s^{-1}\Coder(Ts\ca,\hat Ts\cb)\cong 
s^{-1}\Coder(\hat Ts\ca,\hat Ts\cb)$.
\end{proposition}

\begin{proof}
$B$ is determined by its components 
$B_j:T^j\Coder(Ts\ca,\hat Ts\cb)\to \Coder(Ts\ca,\hat Ts\cb)$ of degree 
1 and level 0 due to \propref{pro-coderivations-components-T}:
\[ B =\sum_{i+j+k=n} 1^{\tens i}\tens B_j\tens1^{\tens k}: 
T^n\Coder(Ts\ca,\hat Ts\cb)\to T^{\le n+1}\Coder(Ts\ca,\hat Ts\cb).
\]
For $n=0$ the equation reads \(f^0b=1_{f^0}B+bf^0\), where 
\(1_{f^0}=1\in T^0\Coder(Ts\ca,\hat Ts\cb)(f^0,f^0)=\Lambda\).
Hence, since both $f^0b$ and $bf^0$ are \((f^0,f^0)\)-coderivations, 
$B_0$ is found in a unique way as
\begin{equation}
1_{f^0}B_0 =f^0b-bf^0 \in \Coder(Ts\ca,\hat Ts\cb)(f^0,f^0).
\label{eq-1f0B0}
\end{equation}

Assume that the coderivation components $B_j$ for $j<n$ are already 
found so that \eqref{eq-theta-b-B-theta} is satisfied up to $n-1$ 
arguments.
Let us determine a $\Lambda$\n-linear map 
\((r^1\tdt r^n)B_n:Ts\ca\to\hat Ts\cb\) from 
\eqref{eq-theta-b-B-theta} rewritten in the form
\begin{multline*}
a.(r^1\tdt r^n)B_n =[a\boxt(r^1\tdt r^n)]\ev b 
-(-)^{r^1+\dots+r^n}[ab\boxt(r^1\tdt r^n)]\ev
\\
-\sum_{q+j+t=n}^{j<n} [a\boxt\{(r^1\tdt r^n)
(1^{\tens q}\tens B_j\tens1^{\tens t})\}]\ev.
%\label{eq-Bn-bb-sum}
\end{multline*}
Let us show that $(r^1\tdt r^n)B_n$ is a $(f^0,f^n)$-coderivation.
Indeed,
\begin{equation}
(r^1\tdt r^n)B_n\Delta =\Delta[f^0\hat\tens(r^1\tdt r^n)B_n
+(r^1\tdt r^n)B_n\hat\tens f^n]
\label{eq-r1tdtrnBn-Delta}
\end{equation}
due to computation
\begin{align*}
&a.(r^1\tdt r^n)B_n\Delta =[a\boxt(r^1\tdt r^n)]\Delta(\ev\hat\tens\ev)
(1\hat\tens b+b\hat\tens1)
\\
&-(-)^{r^1+\dots+r^n}[ab\boxt(r^1\tdt r^n)]\Delta(\ev\hat\tens\ev)
\\
&-\sum_{q+j+t=n}^{j<n} [a\boxt\{(r^1\tdt r^n)(1^{\tens q}\tens B_j\tens
1^{\tens t})\}]\Delta(\ev\hat\tens\ev).
\\
&=\sum_{k+l=n}(a\Delta)\{[\text-\boxt(r^1\tdt r^k)]\ev\hat\tens
[\text-\boxt(r^{k+1}\tdt r^n)]\ev\}(1\hat\tens b+b\hat\tens1)
\\
&-(-)^{r^1+\dots+r^n} \!\!\sum_{k+l=n}\!(a\Delta)(1\tens b+b\tens1)
\{[\text-\boxt(r^1\tdt r^k)]\ev\hat\tens[\text-\boxt(r^{k+1}\tdt r^n)]\ev\}
\\
&-\sum_{q+j+t=n}^{j<n}(a\Delta)\sum_{k+v=q}
[\text-\boxt(r^1\tdt r^k)]\ev\hat\tens[\text-\boxt\{(r^{k+1}\tdt r^n)
(1^{\tens v}\tens B_j\tens1^{\tens t})\}]\ev
\\
&-\!\!\!\sum_{k=0}^n(-)^{r^{k+1}+\dots+r^n}(a\Delta)\!\!\!\!\!\!
\sum_{q+j+w=k}^{j<n}\!\!\!\!\![\text-\boxt\{(r^1\!\tens...\tens\!
r^k)(1^{\tens q}\!\tens\!B_j\!\tens\!1^{\tens w})\}]\!\ev\!\hat\tens
[\text-\boxt(r^{k+1}\!\tens...\tens\!r^n)]\!\ev
\\
&=(a\Delta)\Bigl\langle \sum_{k+l=n}[\text-\boxt(r^1\tdt r^k)]\ev
\hat\tens[\text-\boxt(r^{k+1}\tdt r^n)]\ev b
\\
&-\sum_{k+l=n}(-)^{r^{k+1}+\dots+r^n} [\text-\boxt(r^1\tdt r^k)]
\ev\hat\tens[\text-b\boxt(r^{k+1}\tdt r^n)]\ev
\\
&-\sum_{k+v+j+t=n}^{j<n} [\text-\boxt(r^1\tdt r^k)]\ev\hat\tens[\text-
\boxt\{(r^{k+1}\tdt r^n)(1^{\tens v}\tens B_j\tens1^{\tens t})\}]\ev
\\
&+\sum_{k+l=n}(-)^{r^{k+1}+\dots+r^n} [\text-\boxt(r^1\tdt r^k)]\ev 
b\hat\tens[\text-\boxt(r^{k+1}\tdt r^n)]\ev
\\
&-\sum_{k+l=n}(-)^{r^1+\dots+r^n} [\text-b\boxt(r^1\tdt r^k)]\ev\hat
\tens[\text-\boxt(r^{k+1}\tdt r^n)]\ev
\\
&-\!\!\sum_{k=0}^n(-)^{r^{k+1}+\dots+r^n}\!\!\!\!\!
\sum_{q+j+w=k}^{j<n}\!\!\!\![\text-\boxt\{(r^1\!\tens...\tens\!
r^k)(1^{\tens q}\!\tens\!B_j\!\tens\!1^{\tens w})\}]
\ev\hat\tens[\text-\boxt(r^{k+1}\!\tens...\tens\!r^n)]\ev \Bigr\rangle.
\end{align*}
The sum of the first three expressions in angle brackets equals its 
restriction to $k=0$:
\begin{multline*}
f^0\hat\tens[\text-\boxt(r^1\tdt r^n)]\ev b
-(-)^{r^1+\dots+r^n}f^0\hat\tens[\text-b\boxt(r^1\tdt r^n)]\ev
\\
-\sum_{v+j+t=n}^{j<n}f^0\hat\tens[\text-\boxt\{(r^1\tdt r^n)(1^{\tens v}
\tens B_j\tens1^{\tens t})\}]\ev =f^0\hat\tens(r^1\tdt r^n)B_n.
\end{multline*}
The sum of the last three expressions in angle brackets equals its 
restriction to $k=n$:
\begin{multline*}
[\text-\boxt(r^1\tdt r^n)]\ev b\hat\tens f^n
-(-)^{r^1+\dots+r^n}[\text-b\boxt(r^1\tdt r^n)]\ev\hat\tens f^n
\\
-\sum_{q+j+w=n}^{j<n}[\text-\boxt\{(r^1\tdt r^n)(1^{\tens q}\tens B_j
\tens1^{\tens w})\}]\ev\hat\tens f^n =(r^1\tdt r^n)B_n\hat\tens f^n.
\end{multline*}
This proves \eqref{eq-r1tdtrnBn-Delta}.

Notice that $B^2$ is a (1,1)-coderivation of level 0 and of degree 2:
\[ B^2\Delta =B\Delta(1\tens B+B\tens1) =\Delta(1\tens B+B\tens1)^2 
=\Delta(1\tens B^2+B^2\tens1).
\]
Since $b^2=0$, we have for \(a\in(T^\bull s\ca)^\bull\), $n\ge0$, 
\(r_i\in\Coder(Ts\ca,\hat Ts\cb)\),
\begin{multline*}
[a\boxt(r^1\tdt r^n)B^2]\ev =[a\boxt(r^1\tdt r^n)B]\ev b 
-(-)^{r^1+\dots+r^n+1}[ab\boxt(r^1\tdt r^n)B]\ev
\\
=-(-)^{r^1+\dots+r^n}[ab\boxt(r^1\tdt r^n)]\ev b 
-(-)^{r^1+\dots+r^n+1}[ab\boxt(r^1\tdt r^n)]\ev b =0.
\end{multline*}
Composing this equality with \(\wh{\pr_1}:\hat Ts\cb\to s\hat\cb\) we 
get \(0=[a\boxt(r^1\tdt r^n)B^2]\ev\wh{\pr_1}\).
Substituting into this expression the expansion of $B^2$ into components
\(B^2=\sum_{i+j+k=n}1^{\tens i}\tens(B^2)_j\tens1^{\tens k}\), we find 
that all summands (except one), composed with $\ev$, map 
\(a\boxt(r^1\tdt r^n)\) into the ideal \(\wh{T^{\ge2}s\cb}\), and 
composed furthermore with \(\wh{\pr_1}\) vanish.
The only surviving summand satisfies
\[ 0 =[a\boxt(r^1\tdt r^n)(B^2)_n]\ev\wh{\pr_1} 
=(a)[(r^1\tdt r^n)(B^2)_n]\wh{\pr_1}.
\]
By \corref{cor-a-hatb-coderivations} the coderivation 
\((r^1\tdt r^n)(B^2)_n\in\Coder(Ts\ca,\hat Ts\cb)(f^0,f^n)\) vanishes.
Since this holds for all $n\ge0$, the coderivation $B^2$ vanishes.
\end{proof}

\begin{definition}
Let $\ca$, $\cb$ be filtered \ainf-categories.
A cofunctor \(f:Ts\ca\to\hat Ts\cb\) is called a 
\emph{filtered \ainf-functor} if $bf=fb$.
\end{definition}

Both sides of this equation are $(f,f)$\n-coderivations.
In components:
\begin{alignat*}2
(b\cdot f)_n &=\sum_{i+j+k=n}^{i,j,k\ge0} (1^{\hat\tens i}\hat\tens 
b_j\hat\tens1^{\hat\tens k}) \cdot f_{i+1+k} &&: T^ns\ca \to s\hat\cb,
\\
(f\cdot b)_n &=\sum_{i_1+\dots+i_k=n}^{k,i_j\ge0} (f_{i_1}\htdht
f_{i_k}) \cdot b_k &&: T^ns\ca \to s\hat\cb.
\end{alignat*}
Equality of these expressions for all $n\ge0$ is equivalent to condition
$bf=fb$ and, as we have seen, equivalent to $1_fB_0=0$.

Composing \eqref{eq-theta-b-B-theta} with 
\(\wh{\pr_1}:\hat Ts\cb\to s\hat\cb\) we find the components of 
coderivation $B:T\Coder(Ts\ca,\hat Ts\cb)\to T\Coder(Ts\ca,\hat Ts\cb)$.
Recall that $B_0$ is given by \eqref{eq-1f0B0}, components of $rB_1$ for
\(r:f\to g:Ts\ca\to\hat Ts\cb\) are found from
\begin{multline*}
(a)(rB_1)\wh{\pr_1} =(a)[rb-(-)^rbr]\wh{\pr_1}
\\
=\sum_{i,k\ge0} [a\Delta^{(i+1+k)}] \bigl[
(\check f^{\tens i}\tens\check r\tens\check g^{\tens k})b_{i+1+k} 
-(-)^r(\pr_1^{\tens i}\tens\check b\tens\pr_1^{\tens k})r_{i+1+k}\bigr].
\end{multline*}
Notice that, in general, \([r,b]\equiv rb-(-)^rbr\) is not a 
coderivation, unless the source and the target of $r$ are \ainf-functors
(cf. \remref{rem-[rb]-coderivation}).
In detail, denote by \((a)(rB_1)\spcheck\) the coderivation value 
\((a)(rB_1)\wh{\pr_1}\).
Then by \remref{rem-sumprin-converges-to-Id}
\begin{multline*}
(rB_1)\spcheck_0 =\sum_{i,k\ge0} 
(f_0^{\tens i}\tens r_0\tens g_0^{\tens k})b_{i+1+k} -(-)^rb_0r_1,
\\
\hskip\multlinegap (rB_1)\spcheck_1 =\sum_{i,k\ge0} 
(f_0^{\tens i}\tens r_1\tens g_0^{\tens k})b_{i+1+k}
+\sum_{m,n,k\ge0} (f_0^{\tens m}\tens f_1\tens f_0^{\tens n}
\tens r_0\tens g_0^{\tens k})b_{m+n+2+k} \hfill
\\
+\sum_{i,m,n\ge0} (f_0^{\tens i}\tens r_0\tens g_0^{\tens m}\tens 
g_1\tens g_0^{\tens n})b_{i+2+m+n}
-(-)^r[b_1r_1 +(1\tens b_0)r_2 +(b_0\tens1)r_2],
\end{multline*}
etc.
For $n\ge2$ we have
\begin{multline*}
(a)[(r^1\tdt r^n)B_n]\wh{\pr_1} =[a\boxt(r^1\tdt r^n)]\ev b\,\wh{\pr_1}=
\\
\sum_{i^0,i^1,\dots,i^n\ge0} \hspace*{-1em}[a\Delta^{(i^0+\dots+i^n+n)}]
\bigl((\check f^0)^{\tens i^0}\tens\check r^1\tens
(\check f^1)^{\tens i^1}\tens\check r^2\tdt
(\check f^{n-1})^{\tens i^{n-1}}\tens\check r^n\tens
(\check f^n)^{\tens i^n} \bigr)b_{i^0+\dots+i^n+n}.
\end{multline*}

\begin{remark}\label{rem-[rb]-coderivation}
Let \(f,g:Ts\ca\to\hat Ts\cb\) be filtered \ainf-functors and let $r$ 
be an $(f,g)$\n-coderivation of degree $d$ and of level $l$.
Then \([r,b]=rb-(-)^rbr\) is an $(f,g)$\n-coderivation of degree $d+1$ 
and of level $l$, in particular, $rB_1=[r,b]$.
Indeed,
\begin{align*}
(rb &-(-)^rbr)\Delta =r\Delta(1\hat\tens b+b\hat\tens1) 
-(-)^rb\Delta(f\hat\tens r+r\hat\tens g)
\\
&=\Delta[(f\hat\tens r+r\hat\tens g)(1\hat\tens b+b\hat\tens1) 
-(-)^r(1\hat\tens b+b\hat\tens1)(f\hat\tens r+r\hat\tens g)]
\\
&=\Delta[f\hat\tens(rb-(-)^rbr) +(-)^r(fb-bf)\hat\tens r 
+r\hat\tens(gb-bg) +(rb-(-)^rbr)\hat\tens g].
\end{align*}
\end{remark}

\appendix
\section{Reflective representable multicategories}
\label{ap-A-Reflective}
Let $\cv$ be a symmetric monoidal category, for instance, $\cv=\grAb$.
Let $\mcD$ be a lax representable plain/symmetric/\hspace{0pt}braided 
$\cv$\n-multicategory \cite[Definitions 3.7, 3.23]{BesLyuMan-book}, that
is, for all families \((M_i)_{i\in I}\) of objects of $\mcD$ the 
$\cv$\n-functors $\mcD((M_i)_{i\in I};-):\mcD\to\cv$ are representable.
By \cite[Theorem~3.24]{BesLyuMan-book} the $\cv$\n-multicategory $\mcD$ 
is isomorphic to $\cv$\n-multicategory $\wh\cd$ for a lax 
plain/\hspace{0pt}symmetric/braided monoidal $\cv$\n-category 
\(\cd=(\cd,\tens^I,\lambda^f,\rho^L)\).
The $\cv$\n-multicategory $\wh\cd$ has 
\(\wh{\cd}((M_i)_{i\in I};N)=\cd(\tens^I(M_i),N)\) 
(see \cite[Proposition~3.22]{BesLyuMan-book} for details).
We may and we will take for $\cd$ the category $\mcD$, that is, 
\(\Ob\cd=\Ob\mcD\), \(\cd(M,N)=\mcD(M;N)\).
Denote by $\tilde\mcD$ the plain/symmetric/braided multicategory with 
\(\Ob\tilde\mcD=\Ob\mcD\), 
\(\tilde\mcD((M_i)_{i\in I};N)=\cv(\1_\cv,\mcD((M_i)_{i\in I};N))\).
For instance, when $\cv=\grAb$ we have \(\tilde\mcD=\mcD^0\).
Instead of morphism \(f:\1_\cv\to\mcD((M_i)_{i\in I};N)\in\cv\) we write
\(f:(M_i)_{i\in I}\to N\in\tilde\mcD\).
The multicategory $\tilde\mcD$ is represented by the lax 
plain/symmetric/braided monoidal category \(\tilde\cd\) with 
\(\Ob\tilde\cd=\Ob\cd\), 
\(\tilde\cd(M,N)=\cv(\1_\cv,\cd(M,N))=\tilde\mcD(M;N)\).

Assume that $\Ob\mcD$ contains a subset $\Ob\cc$ such that the full 
subcategory $\cc\subset\cd$ is reflective.
Recall that this is equivalent to giving a morphism 
\(\yi_M:M\to\hat M\in\tilde\cd\) for every \(M\in\Ob\cd\), where 
\(\hat{M}\in\Ob\cc\) and for all \(N\in\Ob\cc\) the morphism 
\(\cd(\yi_M,N):\cd(\hat M,N)\to\cd(M,N)\) is invertible.
In other words, the inclusion $\cv$\n-functor 
\(\inj:\cc\hookrightarrow\cd\) has a left adjoint 
\(\wh{\text-}:\cd\to\cc\).
The unit of this adjunction is \(\yi:\Id_\cd\to\inj\circ\wh{\text-}\).
Denote by $\mcC$ the full plain/symmetric/braided 
$\cv$\n-submulticategory of $\mcD$ with \(\Ob\mcC=\Ob\cc\).

\begin{proposition}\label{prop-C-lax-monoidal}
The $\cv$\n-multicategory $\mcC$ is lax representable by a lax 
plain/symmetric/\hspace{0pt}braided monoidal $\cv$\n-category
\(\cc=(\cc,\hat{\tens}^I,\hat{\lambda}^f,\hat{\rho}^L)\) with
\begin{align}
\hat\tens^{i\in I}M_i &=\wh{\tens^{i\in I}M_i},
\label{eq-tens-hat-Mi}
\\
\hat{\lambda}^f &=\bigl[ \wh{\tens^{i\in I}M_i} \rTTo^{\wh{\lambda^f}} 
(\tens^{j\in J}\tens^{i\in f^{-1}j}M_i)\sphat 
\rTTo^{\wh{\tens^{j\in J}\yi}} 
(\tens^{j\in J}(\tens^{i\in f^{-1}j}M_i)\sphat\;)\sphat\; \bigr],
\\
\hat{\rho}^L &=\bigl( \hat{\tens}^LM =\wh{\tens^LM} 
\rTTo^{\wh{\rho^L}} \hat{M} \rTTo^{\yi_M^{-1}} M \bigr).
\label{eq-rho-hat-L}
\end{align}
\end{proposition}

\begin{proof}
Without loss of generality we assume that \(\mcD=\wh\cd\), that is, 
\(\mcD((M_i)_{i\in I};N)=\cd(\tens^{i\in I}M_i,N)\).
In particular, $\tau:(M_i)_{i\in I}\to\tens^{i\in I}M_i\in\tilde{\mcD}$
corresponds to \(\id_{\tens^{i\in I}M_i}\in\tilde\cd\).
Supposing that \(N,M_i\in\Ob\mcC\) for $i\in I$ we get isomorphisms
\[ \mcC(\hat\tau;N) =\bigl[ \mcC(\wh{\tens^{i\in I}M_i};N) 
\rTTo^{\mcD(\yi;N)} \mcD(\tens^{i\in I}M_i;N) \rTTo^{\mcD(\tau;N)} 
\mcC((M_i)_{i\in I};N) \bigr],
\]
where
\[ \hat\tau =\bigl[ (M_i)_{i\in I} \rTTo^\tau \tens^{i\in I}M_i 
\rTTo^\yi \wh{\tens^{i\in I}M_i} \bigr] \in \tilde{\mcC}.
\]
Therefore, $\mcC$ is lax representable.

By the proof of Theorem~3.24 of \cite{BesLyuMan-book} the lax 
plain/symmetric/braided monoidal $\cv$\n-category 
\((\cc,\hat{\tens}^I,\hat{\lambda}^f,\hat{\rho}^L)\) has structure 
elements given precisely by \eqref{eq-tens-hat-Mi}--\eqref{eq-rho-hat-L}.
For instance, an expression from [\textit{ibid.}] is the top-right path 
in the following diagram
\begin{diagram}[h=2.15em]
\tens^{i\in I}M_i &\rTTo^{\lambda^f} 
&\tens^{j\in J}\tens^{i\in f^{-1}j}M_i &\rTTo^{\tens^{j\in J}\yi} 
&\tens^{j\in J}(\tens^{i\in f^{-1}j}M_i)\sphat\;
\\
\dTTo<\yi &&\dTTo>\yi &&\dTTo>\yi
\\
\wh{\tens^{i\in I}M_i} &\rTTo^{\wh{\lambda^f}} 
&(\tens^{j\in J}\tens^{i\in f^{-1}j}M_i)\sphat 
&\rTTo^{\wh{\tens^{j\in J}\yi}} 
&(\tens^{j\in J}(\tens^{i\in f^{-1}j}M_i)\sphat\;)\sphat\;
\end{diagram}
The same expression has to be equal to \(\yi\cdot\hat{\lambda}^f\).
Since the diagram commutes, \(\hat{\lambda}^f\) is equal to the bottom 
composition \(\wh{\lambda^f}\cdot\wh{\tens^{j\in J}\yi}\).
\end{proof}

Since \(\cc\subset\cd\) is a full reflective subcategory there is an 
idempotent monad \(\hat{\text-}:\cd\to\cd\), \(M\mapsto\hat{M}\), with 
the unit \(\yi:\Id_\cd\to\hat{\text-}\), \(M\to\hat{M}\), and 
multiplication \(\mu_M:\hat{\hat{M}}\to\hat{M}\) inverse to
\(\yi_{\hat{M}}=\hat{\yi}_M:\hat{M}\to\hat{\hat{M}}\) 
\cite[Corollary~4.2.4]{Borceux:CatHandbook2} (see enriched version at 
the end of Chapter~1 of \cite{KellyGM:bascec}).

\textbf{Assume furthermore} that \(\hat{\text-}\) extends to a lax 
plain/symmetric/braided monoidal functor \((\hat{\text-},\phi^n)\),
\(\phi^n:\tens_{i=1}^n\hat{M}_i\to\wh{\tens_{i=1}^nM_i}\), and the 
unit $\yi$ satisfies condition~\eqref{eq-yiyi-phi=yi}:
\begin{equation*}
\bigl(M_1\tdt M_n \rTTo^{\yi_1\tdt\yi_n} \hat{M}_1\tdt\hat{M}_n 
\rTTo^{\phi^n} \wh{M_1\!\!\tens\!\cdot\!\!\cdot\!\!\cdot\!\tens\!\!M_n} 
\bigr) =\yi_{M_1\tdt M_n}.
\end{equation*}
That is, $\yi$ is a monoidal transformation in the sense of 
\cite[Definition~2.7]{BesLyuMan-book}.
Therefore, \(\mu:(\hat{\text-},\phi^n)^2\to(\hat{\text-},\phi^n)\) 
is a monoidal transformation as well.
Indeed, this follows from the commutative diagram
\begin{diagram}[w=5em]
\tens_{i=1}^n\hat{\hat{M}}_i&\rTTo^{\phi^n} &\wh{\tens_{i=1}^n\hat{M}_i}
&\rTTo^{\wh{\phi^n}} &\wh{\wh{\tens_{i=1}^nM_i}}
	\\
\dTTo<{\tens\mu_{M_i}} \uTTo>{\tens\wh{\yi_{M_i}}} &= 
&&\luTTo<{\wh{\tens\yi_{M_i}}}>= &\uTTo<{\wh{\yi_{\tens M_i}}} 
\dTTo>{\mu_{\tens M_i}}
	\\
	\tens_{i=1}^n\hat{M}_i &&\rTTo^{\phi^n} &&\wh{\tens_{i=1}^nM_i}
\end{diagram}
Summing up, \(((\hat{\text-},\phi^n),\yi,\mu)\) is an idempotent lax 
plain/symmetric/braided monoidal monad.

Let \(\{1,2,\dots,n\}=S\sqcup P\).
Given \(M_i\in\Ob\cd\), define
\[ N_i =
\begin{cases}
M_i, &\quad \text{if } i\in S,
\\
\hat{M}_i, &\quad \text{if } i\in P.
\end{cases}
\]
Let \(\yi^0=1\), \(\yi^1=\yi\).
Define
\[ \chi(i\in P) =
\begin{cases}
0, &\quad \text{if } i\in S,
\\
1, &\quad \text{if } i\in P
\end{cases}
\]
and \(\chi(i\in S)=1-\chi(i\in P)\).
Similarly to \cite[Proposition~2.27]{MR3856926} we prove

\begin{proposition}\label{pro-1yi1yi1yi1-invertible}
The morphism 
\(\wh{\tens\yi_{M_i}^{\chi(i\in P)}}:
\wh{\tens_{i=1}^nM_i}\to\wh{\tens_{i=1}^nN_i}\) 
is invertible.
\end{proposition}

\begin{proof}
There is a unique morphism 
\(\xi:\wh{\tens_{i=1}^nN_i}\to\wh{\tens_{i=1}^nM_i}\) which forces 
the diagram
	\begin{diagram}[w=5em,h=2.2em]
		&&&&\wh{\tens_{i=1}^nN_i}
		\\
		&&&\ruTTo(4,2)^{\yi_{\tens N_i}} = &\dTTo>\xi
		\\
\tens_{i=1}^nN_i &\rTTo_{\tens\yi_{N_i}^{\chi(i\in S)}} 
&\tens_{i=1}^n\hat{M}_i &\rTTo_{\phi^n} &\wh{\tens_{i=1}^nM_i}
		\\
&\rdTTo<{\tens\yi_{N_i}}>{{\quad}\ =} 
&\dTTo<{\tens\yi_{\hat{M}_i}^{\chi(i\in P)}}~=>{\tens\wh{\yi_{M_i}^{\chi(i\in P)}}}
&= &
		\\
\dLine &&\tens_{i=1}^n\hat{N}_i 
&&\dTTo>{\wh{\tens\yi_{M_i}^{\chi(i\in P)}}}
		\\
		&= &&\rdTTo<{\phi^n} &
		\\
		\HmeetV &&\rTTo^{\yi_{\tens N_i}} &&\wh{\tens_{i=1}^nN_i}
	\end{diagram}
	to commute.
The reflectivity implies that 
\(\xi\cdot\wh{\tens\yi_{M_i}^{\chi(i\in P)}}=1\).
Commutativity of the composite rectangle in this diagram implies 
commutativity of square \ovalbox{1} in the following diagram
	
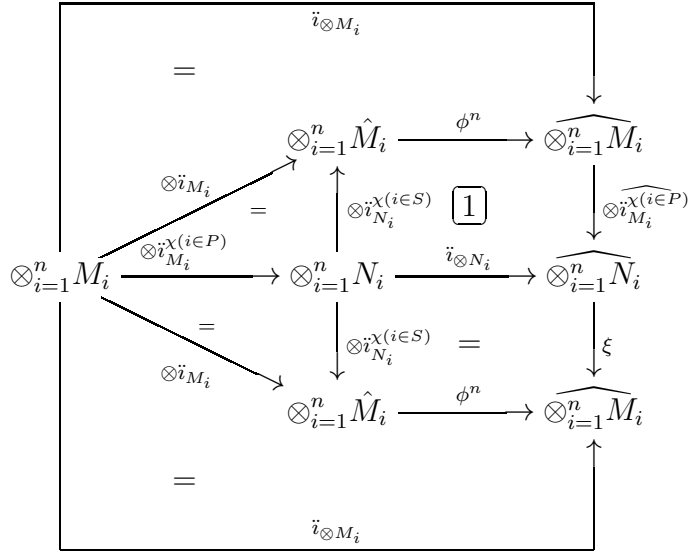
\begin{diagram}[w=4em,h=2.2em]
		\HmeetV &&\rLine_{\yi_{\tens M_i}} &&\HmeetV
		\\
		&= &&&\dTTo
		\\
\uLine &&\tens_{i=1}^n\hat{M}_i &\rTTo^{\phi^n} &\wh{\tens_{i=1}^nM_i}
		\\
&\ruTTo<{\tens\yi_{M_i}} &\uTTo<{=\qquad}>{\tens\yi_{N_i}^{\chi(i\in S)}}
&\ovalbox{1} &\dTTo>{\wh{\tens\yi_{M_i}^{\chi(i\in P)}}}
		\\
\tens_{i=1}^nM_i &\rTTo^{\tens\yi_{M_i}^{\chi(i\in P)}} 
&\tens_{i=1}^nN_i &\rTTo^{\yi_{\tens N_i}} &\wh{\tens_{i=1}^nN_i}
		\\
&\rdTTo<{\tens\yi_{M_i}}>= &\dTTo>{\tens\yi_{N_i}^{\chi(i\in S)}} 
&= &\dTTo>\xi
		\\
\dLine &&\tens_{i=1}^n\hat{M}_i &\rTTo^{\phi^n} &\wh{\tens_{i=1}^nM_i}
		\\
		&= &&&\uTTo
		\\
		\HmeetV &&\rLine^{\yi_{\tens M_i}} &&\HmeetV
	\end{diagram}
Commutativity of the above together with reflectivity implies that 
\(\wh{\tens\yi_{M_i}^{\chi(i\in P)}}\cdot\xi=1\).
	Thus, $\xi$ is inverse to \(\wh{\tens\yi_{M_i}^{\chi(i\in P)}}\).
\end{proof}

The unit object of \((\cd,\tens^n)\) is \(\1=\tens^0(*)\), therefore, 
the unit object of \((\cc,\hat\tens^n)\) is \(\hat\1=\hat\tens^0(*)\).

\begin{corollary}\label{cor-D-monoidal-C-monoidal}
When $\cd$ is a plain/symmetric/braided monoidal category (all 
$\lambda^f$, $\rho^L$ are invertible), so is $\cc$.
\end{corollary}

In fact, invertibility of $\wh{\lambda^f}$ and of 
\(\wh{\tens^{j\in J}\yi}\) implies invertibility of their composition 
\(\hat{\lambda}^f\).

\subsection{Algebras and coalgebras.}
Assume that the category $\cd$ is monoidal (all $\lambda^f$, $\rho^L$ 
are invertible).
Hence, the same for $\cc$.
The category of algebras (monoids) in $\cd$ (resp. $\cc$) is denoted 
$\Alg_{\tilde{\cd}}$ (resp. $\Alg_{\tilde{\cc}}$).

\begin{proposition}
The full and faithful functor 
\(\inj:\Alg_{\tilde{\cc}}\to\Alg_{\tilde{\cd}}\), 
\((B,\mu_B:B\hat{\tens}B\to B,\eta_B:\hat{\1}\to B)\mapsto
\bigl(B,B\tens B\rto\yi \wh{B\tens B}\rto{\mu_B} B,\1\rto\yi 
\hat{\1}\rto{\eta_B} B\bigr)\) 
turns $\Alg_{\tilde{\cc}}$ into a reflective subcategory of 
$\Alg_{\tilde{\cd}}$.
\end{proposition}

\begin{proof}
First of all, $\inj B$ is an algebra in $\cd$ (the proof is left to 
the reader).
Secondly, any morphism \(f:A\to B\in\Alg_{\tilde{\cc}}\) induces 
\(f:\inj A\to\inj B\in\Alg_{\tilde{\cd}}\).
Clearly, the functor $\inj$ is faithful.
One can show that it is full.
This functor has a left adjoint, namely, the completion functor 
\(\hat{\text-}:\Alg_{\tilde{\cd}}\to\Alg_{\tilde{\cc}}\),
\((A,\mu_A,\eta_A)\mapsto\bigl(\hat{A},\mu_{\hat{A}}
=(\hat{A}\hat{\tens}\hat{A}
=\wh{\hat{A}\tens\hat{A}}\rTTo^{\wh{\,\yi\tens\yi\,}^{-1}} 
\wh{A\tens A}\rto{\wh{\mu_A}} \hat{A}),\eta_{\hat{A}}
=\wh{\eta_A}:\hat{\1}\to\hat{A}\bigr)\).
The natural transformation 
\(\yi:A\to\hat{A}=\bigl(\hat{A},\hat{A}\tens\hat{A}\rto\yi 
\wh{\hat{A}\tens\hat{A}}\rTTo^{\wh{\,\yi\tens\yi\,}^{-1}} \wh{A\tens A}
\rto{\wh{\mu_A}} \hat{A},\1\rto\yi \hat{\1}\rto{\wh{\eta_A}} 
\hat{A}\bigr)\) 
is given precisely by $\yi:A\to\hat{A}$.

The proof of these statement is left to the reader.
\end{proof}

Let $\tilde\cd$ contain arbitrary small coproducts.
Then for any \(M\in\Ob\cd\) there is the tensor algebra 
\(TM=\coprod_{n\ge0}M^{\tens n}\equiv
\oplus_{n\ge0}M^{\tens n}\in\tilde\cd\).
The functor \(T:\tilde\cd\to\Alg_{\tilde\cd}\) is left adjoint to the 
underlying functor \(\cu:\Alg_{\tilde\cd}\to\tilde{\cd}\).
Then for \(B\in\Alg_{\tilde{\cc}}\), \(X\in\cc\) there are natural 
bijections
\[ \Alg_{\tilde{\cc}}(\wh{TX},B) \cong\Alg_{\tilde{\cd}}(TX,\inj B) 
\cong\tilde{\cd}(X,\cu\inj B) =\tilde{\cc}(X,\cu B).
\]
Hence, the functor \(\tilde{\cc}\to\Alg_{\tilde{\cc}}\), 
\(X\mapsto\wh{TX}\) is left adjoint to 
\(\cu:\Alg_{\tilde{\cc}}\to\tilde{\cc}\).
Multiplication in the algebra \(\hat{A}=\wh{TX}\) with \(A=TX\)
\begin{equation}
\mu_{\hat{A}}^{(I)} =\bigl[(\wh{TX})^{\hat\tens I} 
=\wh{(\wh{T\!X})^{\tens I}} \rTTo^{\wh{\yi^{\tens I}}^{-1}} 
\wh{(T\!X)^{\tens I}} \rTTo^{\wh{\mu_A^{(I)}}} \wh{TX}\bigr]
\label{eq-Multiplication-algebra-hat-TX}
\end{equation}
is denoted also as \(\hat{\tens}\) (by abuse of notation).

\subsubsection{Completion of coalgebras.}
\label{sec-Completion-coalgebras}
The category of coalgebras (comonoids) in $\cd$ (resp. $\cc$) is denoted
$\Coalg_{\tilde{\cd}}$ (resp. $\Coalg_{\tilde{\cc}}$).
The completion functor extends to a functor 
\(\hat{\text-}:\Coalg_{\tilde{\cd}}\to\Coalg_{\tilde{\cc}}\),
\((C,\Delta_C,\eps_C)\mapsto\bigl(\hat{C},\Delta_{\hat{C}}
=(\hat{C}\rto{\wh{\Delta_C}} \wh{C\tens C}\rTTo^{\wh{\yi\tens\yi}} 
\wh{\hat{C}\tens\hat{C}}= \hat{C}\hat{\tens}\hat{C}),
\eps_{\hat{C}}=\wh{\eps_C}:\hat{C}\to\hat{\1}\bigr)\).
The proof is left to the reader.
Notice that
\[ \bigl(\hat{C}\rto{\wh{\Delta_C^{(k)}}} \wh{C^{\tens k}}
\rTTo^{\wh{\yi^{\tens k}}} \wh{\hat{C}^{\tens k}}
=\hat{C}^{\hat\tens k}\bigr) =\Delta_{\hat{C}}^{(k)}.
\]

%\bibliographystyle{amsalpha}
%\bibliography{yuri8}

\begin{thebibliography}{Kelly, 1982}

\bibitem[Bespalov, Lyubashenko, Manzyuk, 2008]{BesLyuMan-book}
Yuri Bespalov, Volodymyr Lyubashenko, and Oleksandr Manzyuk, 
\emph{Pretriangulated ${A}_\infty$-categories}, Proceedings of the 
Institute of Mathematics of NAS of Ukraine. Mathematics and its 
Applications, vol.~76, Institute of Mathematics of NAS of Ukraine, 
Kyiv, 2008, \url{http://www.math.ksu.edu/~lub/papers.html}.

\bibitem[Borceux, 1994]{Borceux:CatHandbook2}
Francis Borceux, \emph{Handbook of categorical algebra. 2}, Encyclopedia
of Mathematics and its Applications, vol.~51, Cambridge University 
Press, 1994, Categories and structures.

\bibitem[Bourbaki, 1971]{zbMATH03395017}
Nicolas Bourbaki, \emph{Topologie g\'en\'erale}, \'El\'ements de
math\'ematique, Hermann, Paris, 1971, Chap. 1 \`a 4, 351 pp.

\bibitem[De~Deken, Lowen, 2018]{MR3856926}
Olivier De~Deken and Wendy Lowen, \emph{Filtered 
{$cA_{\infty}$}-categories and functor categories}, Appl. Categ. 
Structures \textbf{26} (2018), no.~5,
943--996, \url{https://doi.org/10.1007/s10485-018-9526-2}.

\bibitem[Day, Street, 2003]{DayStreet-subst}
Brian~J. Day and Ross~H. Street, \emph{Abstract substitution in enriched
categories}, J. Pure Appl. Algebra \textbf{179} (2003), no.~1-2, 49--63.

\bibitem[Fukaya, 2002]{Fukaya:FloerMirror-II}
Kenji Fukaya, \emph{Floer homology and mirror symmetry. {II}}, Minimal
surfaces, geometric analysis and symplectic geometry (Baltimore, MD, 
1999), Adv. Stud. Pure Math., vol.~34, Math. Soc. Japan, Tokyo, 2002, 
pp.~31--127.

\bibitem[Fukaya, Oh, Ohta, Ono, 2009]{FukayaOhOhtaOno:Anomaly}
Kenji Fukaya, Yong-Geun Oh, Hiroshi Ohta, and Kaoru Ono, 
\emph{Lagrangian intersection {F}loer theory: Anomaly and obstruction}, 
AMS/IP Studies in Advanced Mathematics Series, vol.~46, American 
Mathematical Society, 2009.

\bibitem[Hermida, 2000]{Hermida:MultiRep}
Claudio Hermida, \emph{Representable multicategories}, Advances in Math.
\textbf{151} (2000), no.~2, 164--225.

\bibitem[Kelly, 1982]{KellyGM:bascec}
Gregory~Maxwell Kelly, \emph{Basic concepts of enriched category 
theory}, Theory Appl. Categ. \textbf{64} (1982, 2005), no.~10, vi+137 
pp. (electronic), 
\url{http://www.tac.mta.ca/tac/reprints/articles/10/tr10abs.html} 
Reprint of the 1982 original [Cambridge Univ. Press, Cambridge].

\bibitem[Kontsevich, 1995]{Kontsevich:alg-geom/9411018}
Maxim Kontsevich, \emph{Homological algebra of mirror symmetry}, Proc.
Internat. Cong. Math., Z\"urich, Switzerland 1994 (Basel), vol.~1,
Birkh\"auser Verlag, 1995, 
\href{http://arXiv.org/abs/alg-geom/9411018}{{\tt 
	arXiv:\linebreak[1]alg-geom/9411018}}, pp.~120--139.

\bibitem[Lyubashenko, 2003]{Lyu-AinfCat}
Volodymyr Lyubashenko, \emph{Category of ${A}_\infty$-categories}, Homology,
Homotopy Appl. \textbf{5} (2003), no.~1, 1--48,
\href{http://arXiv.org/abs/math/0210047}{{\tt
		arXiv:\linebreak[1]math/0210047}}
\url{http://intlpress.com/HHA/v5/n1/a1/}.

\bibitem[Polishchuk, Positselski, 2012]{1010.0982}
Alexander Polishchuk and Leonid Positselski, \emph{Hochschild 
(co)homology of the second kind {I}}, Trans. Amer. Math. Soc. 
\textbf{364} (2012), no.~10, 5311--5368, 
\href{http://arXiv.org/abs/1010.0982}{{\tt
		arXiv:\linebreak[1]1010.0982}}.

\bibitem[Tarski, 1939]{Tarski-well-ordered-subsets}
Alfred Tarski, \emph{On well-ordered subsets of any set}, Fundamenta
Mathematicae \textbf{32} (1939), 176--183,
\url{http://matwbn.icm.edu.pl/ksiazki/fm/fm32/fm32115.pdf}.

\end{thebibliography}

\end{document}